\documentclass[12pt]{article}
\usepackage{amsthm,amsfonts,amsmath,amscd,amssymb,latexsym}    
\usepackage{mathrsfs}
\usepackage{epsfig,graphicx,color}     
\usepackage[all]{xy}  
\usepackage{makeidx}
\usepackage{hyperref}
\usepackage{accents}
\usepackage{multirow}
\usepackage{array}
\usepackage{scalerel}[2014/03/10]
\usepackage{stackengine,wasysym}

\title{
Complex structures, moment maps, \\
and the Ricci form (Extended Version)
}

\author{
Oscar Garc\'{\i}a-Prada\\
ICMAT Madrid 
\and
Dietmar~A.~Salamon\\
ETH Z\"urich
\and
Samuel Trautwein\\
ETH Z\"urich
}

\date{8 March 2021}

\newtheorem{PARA}{}[section] 
\newtheorem{theorem}[PARA]{Theorem} 
\newtheorem{corollary}[PARA]{Corollary} 
\newtheorem{lemma}[PARA]{Lemma} 
 
\newtheorem{definition}[PARA]{Definition}    
\newtheorem{remark}[PARA]{Remark} 
\newtheorem{example}[PARA]{Example}

\numberwithin{equation}{section}
\newcommand{\MAT}[1]{\left[\begin{array}{#1}}
\newcommand{\RIX}{\end{array}\right]}
\newcommand{\p}{\partial}
\newcommand{\one}{{{\mathchoice {\rm 1\mskip-4mu l} {\rm 1\mskip-4mu l}
{\rm 1\mskip-4.5mu l} {\rm 1\mskip-5mu l}}}}

\newcommand{\C}{{\mathbb{C}}}

\newcommand{\R}{{\mathbb{R}}}
\newcommand{\T}{{\mathbb{T}}}
\newcommand{\Z}{{\mathbb{Z}}}

\newcommand{\cK}{{\mathcal K}}
\newcommand{\cL}{{\mathcal L}}

\newcommand{\cT}{{\mathcal T}}

\newcommand{\sA}{\mathscr{A}}    
\newcommand{\sB}{\mathscr{B}}

\newcommand{\sE}{\mathscr{E}}    
    
\newcommand{\sG}{\mathscr{G}}    
\newcommand{\sH}{\mathscr{H}}    
    
\newcommand{\sJ}{\mathscr{J}}    
\newcommand{\sK}{\mathscr{K}}    
    
\newcommand{\sM}{\mathscr{M}}

\newcommand{\sS}{\mathscr{S}}    
\newcommand{\sT}{\mathscr{T}}

\newcommand{\sW}{\mathscr{W}}

\newcommand{\sn}{{\mathsf{n}}}
\newcommand{\Om}{{\Omega}}
\newcommand{\om}{{\omega}}

\newcommand{\xhat}{{\widehat{x}}}

\newcommand{\Jhat}{{\widehat{J}}}

\newcommand{\Xhat}{{\widehat{X}}}
\newcommand{\Yhat}{{\widehat{Y}}}
\newcommand{\Zhat}{{\widehat{Z}}}

\newcommand{\betahat}{{\widehat{\beta}}}
\newcommand{\lambdahat}{{\widehat{\lambda}}}

\newcommand{\thetahat}{{\widehat{\theta}}}

\newcommand{\omhat}{{\widehat{\omega}}}
\newcommand{\tauhat}{{\widehat{\tau}}}

\newcommand{\Richat}{{\widehat{\Ric}}}
\newcommand{\oa}{{\bar{a}}}
\newcommand{\oc}{{\bar{c}}}
\newcommand{\og}{{\bar{g}}}
\newcommand{\oh}{{\bar{h}}}
\newcommand{\oz}{{\bar{z}}}
\newcommand{\obeta}{{\overline{\beta}}}
\newcommand{\oeta}{{\overline{\eta}}}
\newcommand{\otheta}{{\overline{\theta}}}

\newcommand{\im}{{\mathrm{im}}}

\newcommand{\DIV}{{\mathrm{div}}}   
\newcommand{\trace}{{\mathrm{trace}}} 
 
\newcommand{\id}{{\mathrm{id}}}

\newcommand{\INT}{{\mathrm{int}}}

\newcommand{\RE}{{\mathrm{Re}}}  
\renewcommand{\Re}{{\mathrm{Re}}}  
\renewcommand{\Im}{{\mathrm{Im}}}  
\newcommand{\Lie}{{\mathrm{Lie}}}

\newcommand{\Diff}{{\mathrm{Diff}}} 
\newcommand{\Vect}{{\mathrm{Vect}}}  
\newcommand{\Symp}{{\mathrm{Symp}}}   
\newcommand{\Ham}{{\mathrm{Ham}}}

\newcommand{\End}{{\mathrm{End}}}

\newcommand{\bi}{{\mathbf{i}}}

\newcommand{\ex}{{\rm ex}}

\newcommand{\norm}{{\rm norm}}

\newcommand{\dvol}{{\rm dvol}}

\newcommand{\rG}{{\mathrm{G}}} 
\newcommand{\rH}{{\mathrm{H}}}

\newcommand{\SU}{{\mathrm{SU}}} 
 
\newcommand{\SL}{{\mathrm{SL}}}

\newcommand{\Sp}{{\mathrm{Sp}}}

\newcommand{\fsl}{{\mathfrak{sl}}} 
\newcommand{\fsp}{{\mathfrak{sp}}} 
\newcommand{\fg}{{\mathfrak{g}}} 
\newcommand{\fh}{{\mathfrak{h}}} 

\newcommand{\CP}{{\mathbb{C}\mathrm{P}}}

\newcommand{\BC}{{\mathrm{BC}}}

\newcommand{\Ric}{{\mathrm{Ric}}}

\newcommand{\dR}{{\mathrm{dR}}}

\newcommand{\inner}[2]{\langle #1, #2\rangle}
\newcommand{\INNER}[2]{\left\langle #1, #2\right\rangle}

\newcommand{\winner}[2]{\langle #1{\wedge}#2\rangle}

\newlength{\dtildeheight}

\newlength{\dhatheight}

\newcommand{\tabla}{{\widetilde{\nabla}}}
\newcommand{\habla}{{\widehat{\nabla}}}
\newcommand{\dhabla}{{\widehat{\vphantom{\rule{1pt}{9.6pt}}\smash{\widehat{\nabla}}}}}
\newcommand{\sdhabla}{{\widehat{\vphantom{\rule{1pt}{6.8pt}}\smash{\widehat{\nabla}}}}}

\def\NABLA#1{{\mathop{\nabla\kern-.5ex\lower1ex\hbox{$#1$}}}}
\def\Nabla#1{{\nabla\kern-.5ex{}_{#1}}}
\def\Tabla#1{{\widetilde\nabla\kern-.5ex{}_{#1}}}
\def\DTabla#1{{{\widetilde{\vphantom{\rule{1pt}{9.6pt}}\smash{\widetilde{\nabla}}}}\kern-.5ex{}_{#1}}}
\def\SDTabla#1{{{\widetilde{\vphantom{\rule{1pt}{6.8pt}}\smash{\widetilde{\nabla}}}}\kern-.5ex{}_{#1}}}
\def\Habla#1{{\widehat\nabla\kern-.5ex{}_{#1}}}
\def\DHabla#1{{{\widehat{\vphantom{\rule{1pt}{9.6pt}}\smash{\widehat{\nabla}}}}\kern-.5ex{}_{#1}}}
\def\SDHabla#1{{{\widehat{\vphantom{\rule{1pt}{6.8pt}}\smash{\widehat{\nabla}}}}\kern-.5ex{}_{#1}}}
\def\abs#1{\mathopen|#1\mathclose|}

\def\norm#1{\mathopen\|#1\mathclose\|}
\def\Norm#1{\left\|#1\right\|}
\renewcommand{\p}{{\partial}}

\begin{document}
\maketitle

\begin{abstract}
The Ricci form is a moment map for the action 
of the group of exact volume preserving diffeomorphisms 
on the space of almost complex structures. 
This observation yields a new approach to the 
Weil--Petersson symplectic form on the Teichm\"uller 
space of isotopy classes of complex structures with 
real first Chern class zero and nonempty K\"ahler cone.
This extended version of the paper includes a proof of the
Bochner--Kodaira--Nakano identity (Appendix~\ref{app:BKN}),
a brief exposition of Bott--Chern cohomology 
(Appendix~\ref{app:BOTTCHERN}), and a discussion
of the relation between complex structures and differential 
forms of middle degree (Appendix~\ref{app:NFORM}).
\end{abstract}

\vspace{-20pt}

{\tiny\tableofcontents}


\section{Introduction}\label{sec:INTRO}

This paper is based on a remark by Simon Donaldson.
The remark is that the space of linear complex structures
on~$\R^{2\sn}$ can be viewed as a co-adjoint $\SL(2\sn,\R)$-orbit
and hence is equipped with a canonical symplectic form 
and a Hamiltonian $\SL(2\sn,\R)$-action. 
Thus, for any volume form~$\rho$
on a closed oriented $2\sn$-manifold~$M$,
the space~$\sJ(M)$ of almost complex 
structures carries a natural symplectic structure.
Following~\cite{DON3}, one can then deduce
that the action of the group of exact volume preserving 
diffeomorphisms on~$\sJ(M)$ is a Hamiltonian group action
with the Ricci form as a moment map. 
In the integrable case this yields
a new approach to the Weil--Petersson symplectic form 
on the Teichm\"uller space of isotopy classes of 
complex structures with real first Chern class zero 
and nonempty K\"ahler cone. Here are the details.

Fix a closed connected oriented $2\sn$-manifold~$M$
and a positive volume form~$\rho$ and denote 
by~$\sJ(M)$ the space of almost complex structures
compatible with the orientation. This space is equipped 
with a symplectic form
\begin{equation}\label{eq:OMRHOJ}
\Om_{\rho,J}(\Jhat_1,\Jhat_2) 
:= \tfrac{1}{2}\int_M\trace\Bigl(\Jhat_1J\Jhat_2\Bigr)\rho
\qquad\mbox{for }\Jhat_1,\Jhat_2\in\Om^{0,1}_J(M,TM).
\end{equation}
The {\bf Ricci form}~${\Ric_{\rho,J}\in\Om^2(M)}$ 
associated to~$\rho$ and~$J$ is defined by
\begin{equation*}
\begin{split}
\Ric_{\rho,J}(u,v) 
:= \tfrac{1}{4}\trace\bigl((\Nabla{u}J)J(\Nabla{v}J)\bigr) 
+ \tfrac{1}{2}\trace\bigl(JR^\nabla(u,v)\bigr) 
+ \tfrac{1}{2}d\lambda^\nabla_J(u,v)
\end{split}
\end{equation*}
for~${u,v\in\Vect(M)}$, where~$\nabla$ is a torsion-free
$\rho$-connection and the $1$-form~${\lambda^\nabla_J}$ on~$M$ 
is defined by~${\lambda^\nabla_J(u):=\trace\bigl((\nabla J)u\bigr)}$
for~${u\in\Vect(M)}$.  In the integrable case~$\bi\Ric_{\rho,J}$ 
is the curvature of the Chern connection on the canonical bundle 
associated to the Hermitian structure determined by~$\rho$.

\medskip\noindent{\bf Theorem~A (The Ricci Form).}
{\it The Ricci form is independent of the choice of the 
torsion-free $\rho$-connection $\nabla$ used to define it. 
It is closed, represents the cohomology class~$2\pi c_1(TM,J)$, 
satisfies~${\phi^*\Ric_{\rho,J}=\Ric_{\phi^*\rho,\phi^*J}}$
for every diffeomorphism~$\phi$, and 
$\Ric_{e^f\rho,J}=\Ric_{\rho,J}+\tfrac{1}{2}d(df\circ J)$
for all~${f\in\Om^0(M)}$.
Moreover, the map~${J\mapsto2\Ric_{\rho,J}}$ is a moment 
map for the action of the group~$\Diff^\ex(M,\rho)$
of exact volume preserving diffeomorphisms on~$\sJ(M)$,
i.e.\ if~${t\mapsto J_t}$ is a smooth path of 
almost complex structures on~$M$, then
\begin{equation}\label{eq:RICMOMENT}
\frac{d}{dt}\int_M2\Ric_{\rho,J_t}\wedge\alpha
= \tfrac{1}{2}\int_M\trace\Bigl((\p_tJ_t)J_t(\cL_{v_\alpha}J_t)\Bigr)\rho
\end{equation}
for~${t\in\R}$ and~${\alpha\in\Om^{2\sn-2}(M)}$,
where~${v_\alpha\in\Vect(M)}$ is defined 
by~${\iota(v_\alpha)\rho=d\alpha}$.}

\begin{proof}
See Theorem~\ref{thm:RICCI}. 
\end{proof}

\bigbreak

The proof of Theorem~A is based on the aforementioned 
observation that the space of linear complex structures 
is a co-adjoint~$\SL(2\sn,\R)$-orbit.  Theorem~A can then be 
derived from a general result of Donaldson~\cite{DON3} 
about the action of the group~$\Diff^\ex(M,\rho)$ 
on a suitable space of sections of a fibration over~$M$. 
In Section~\ref{sec:RICCI} we give a direct proof which does 
not rely on~\cite{DON3}.  That the Ricci form is closed 
and represents $2\pi$ times the first Chern class is a consequence of the 
formula
$
\Ric_{\rho,J} = \tfrac{1}{2}\trace(JR^\tabla)+d\lambda^\nabla_J,
$
where~${\om\in\Om^2(M)}$  is a nondegenerate $2$-form compatible 
with~$J$, $\nabla$~is the Levi--Civita connection of the 
metric~${\om(\cdot,J\cdot)}$, and~${\tabla := \nabla - \tfrac{1}{2}J\nabla J}$.  
Moreover, ${\lambda^\nabla_J=0}$ whenever~$\om$ is closed,
so one obtains the standard Ricci form in the symplectic case.
We emphasize that the dual space of the space of exact divergence-free
vector fields is the space of exact $2$-forms on~$M$, so one obtains 
a genuine moment map only for almost complex structures 
with real first Chern class zero.

Equation~\eqref{eq:RICMOMENT} extends to an 
identity that holds for all vector fields~$v$.
This identity takes the form
\begin{equation}\label{eq:LAMBDA1}
\int_M\Lambda_\rho(J,\Jhat)\wedge\iota(v)\rho
= \tfrac{1}{2}\int_M\trace\Bigl(\Jhat J\cL_vJ\Bigr)\rho
\end{equation}
for all~${\Jhat\in\Om^{0,1}_J(M,TM)}$ and all~${v\in\Vect(M)}$,
where~${\Lambda_\rho(J,\Jhat)\in\Om^1(M)}$
is defined by~${(\Lambda_\rho(J,\Jhat))(u)
:= \trace((\nabla\Jhat)u+\tfrac12\Jhat J\Nabla{u}J)}$
for~${u\in\Vect(M)}$.   Thus~$\Lambda_\rho$ is a $1$-form
on~$\sJ(M)$ with values in~$\Om^1(M)$.  The next theorem shows
that the differential of this $1$-form is a $2$-form on~$\sJ(M)$
with values in~$d\Om^0(M)$.

\medskip\noindent{\bf Theorem~B (The one-form $\Lambda_\rho$).}
{\it 
Let~${v\in\Vect(M)}$ and define~${f_v\in\Om^0(M)}$
by~${f_v\rho:=d\iota(v)\rho}$.  Then, for all~${J\in\sJ(M)}$,
\begin{equation}\label{eq:LAMBDA2}
\Lambda_\rho(J,\cL_vJ)
= 2\iota(v)\Ric_{\rho,J} - df_v\circ J + df_{Jv}.
\end{equation}
Moreover, if~${\R^2\to\sJ(M):(s,t)\mapsto J(s,t)}$ is a smooth map, then}
\begin{equation}\label{eq:LAMBDA3}
\p_s\Lambda_\rho(J,\p_tJ) - \p_t\Lambda_\rho(J,\p_sJ)
+ \tfrac{1}{2}d\trace((\p_sJ)J(\p_tJ)) = 0.
\end{equation}

\begin{proof}
See Theorem~\ref{thm:LAMBDA}.
\end{proof}

\noindent{\bf Theorem~C (The Integrable Case).}
{\it Let~${\rho\in\Om^{2\sn}(M)}$ be a positive volume form 
and let~${J\in\sJ(M)}$ be an integrable almost complex structure.
Then~${\tfrac{1}{2\pi}\Ric_{\rho,J}}$ is a~${(1,1)}$-form 
and represents the first Bott--Chern class of~$J$.
Moreover, the first Bott--Chern class of~$J$ vanishes
if and only if there exists a diffeomorphism~$\phi\in\Diff_0(M)$
such that~${\Ric_{\rho,\phi^*J}=0}$.
If~${\Ric_{\rho,J}=\Ric_{\rho,\phi^*J}=0}$
for some orientation preserving diffeomorphism~$\phi$,
then~${\phi^*\rho=\rho}$.}

\begin{proof}
See Theorem~\ref{thm:RICBC}.
\end{proof}

\bigbreak

Let~${\sJ_{\INT,0}(M)\subset\sJ(M)}$ be the space of 
integrable almost complex structures with real first Chern 
class zero and nonempty K\"ahler cone.
Then Theorem~C shows that the Teichm\"uller 
space~${\sT_0(M):=\sJ_{\INT,0}(M)/\Diff_0(M)}$ can be 
identified with the quotient space~${\sT_0(M,\rho):=\sJ_{\INT,0}(M,\rho)/\Diff_0(M,\rho)}$,
where~${\sJ_{\INT,0}(M,\rho):=\{J\in\sJ_{\INT,0}(M)\,|\,\Ric_{\rho,J}=0\}}$.
We emphasize that the quotient 
group~${\Diff_0(M,\rho)/\Diff^\ex(M,\rho)}$ acts trivially.  
The space~$\sJ(M)$ carries a complex 
structure~${\Jhat\mapsto-J\Jhat}$ and the 
symplectic form~$\Om_\rho$ in~\eqref{eq:OMRHOJ} is of type~${(1,1)}$.
However, it is not K\"ahler because the symmetric  
pairing~${\inner{\Jhat_1}{\Jhat_2}=\tfrac{1}{2}\int_M\trace(\Jhat_1\Jhat_2)\rho}$
is indefinite in general. Thus complex submanifolds of~$\sJ(M)$
need not be symplectic.  The space~${\sJ_{\INT,0}(M)}$
is an example. Its tangent space at~$J$ is the kernel 
of~${\bar\p_J:\Om^{0,1}_J(M,TM)\to\Om^{0,2}_J(M,TM)}$.
If~${\Ric_{\rho,J}=0}$ and~${\bar\p_J\Jhat=0}$,
then Theorem~B implies that there exist unique smooth 
functions~${f=f_{\rho,\Jhat}}$ and~${g=f_{\rho,J\Jhat}}$ 
such that
\begin{equation}\label{eq:LAMBDA4}
\Lambda_\rho(J,\Jhat) = -df\circ J + dg,\qquad 
\int_Mf\rho=\int_Mg\rho = 0.
\end{equation}
This implies that the restriction 
of the $2$-form~$\Om_{\rho,J}$ to~$\ker\bar\p_J$ 
vanishes on the subspace~${\{\cL_vJ\,|\,f_v=f_{Jv}=0\}}$. 
It turns out that~$\Om_\rho$ descends to a symplectic form 
on the Teich\-m\"uller space~${\sT_0(M,\rho)\cong\sT_0(M)}$ 
that is independent of~$\rho$.   For~${J\in\sJ_{\INT,0}(M)}$ let~$\rho_J$ 
be the volume form with~${\Ric_{\rho_J,J}=0}$ and~${\int_M\rho_J=V}$. 

\medskip\noindent{\bf Theorem~D (Teichm\"uller Space).}
{\it The formula
\begin{equation}\label{eq:TEICH1}
\Om_J(\Jhat_1,\Jhat_2) 
:= \int_M\Bigl(
\tfrac{1}{2}\trace\bigl(\Jhat_1J\Jhat_2\bigr)
- f_1g_2 + f_2g_1
\Bigr)\rho_J,
\end{equation}
for~${J\in\sJ_{\INT,0}(M)}$ and~${\Jhat_i\in\Om^{0,1}_J(M,TM)}$
with~${\bar\p_J\Jhat_i=0}$ and~$f_i,g_i$ as in~\eqref{eq:LAMBDA4},
defines a symplectic form on the Teichm\"uller space~$\sT_0(M)$.
It satisfies the naturality 
condition~${\Om_{\phi^*J}(\phi^*\Jhat_1,\phi^*\Jhat_2)
= \phi^*\Om_J(\Jhat_1,\Jhat_2)}$
for every~$\phi\in\Diff^+(M)$ and thus the mapping class 
group acts on~$\sT_0(M)$ by symplectomorphisms.}

\begin{proof}
See Theorem~\ref{thm:TEICH}.
\end{proof}

Theorem~D gives an alternative construction of 
the Weil--Petersson symplectic form on Calabi--Yau
Teichm\"uller spaces (see~\cite{HUY3,KOISO,N,S2,S4,SIU1}
for the polarized case and~\cite[Ch~16]{FHW}
for the symplectic form on~$\sT_0(M)$
for the K3 surface). The proof relies on 
Yau's theorem and the observations, 
for Ricci-flat K\"ahler manifolds~$(M,\om,J)$,
that a vector field~$v$ is holomorphic 
if and only if~${\iota(v)\om}$ is harmonic 
(Lemma~\ref{le:HOLV}), and that the space of 
$\bar\p_J$-harmonic $1$-forms~${\Jhat\in\Om^{0,1}_J(M,TM)}$
is invariant under the map~${\Jhat\mapsto\Jhat^*}$
(Lemma~\ref{le:JHATSTAR}). 

\bigbreak

Associated to the symplectic form~\eqref{eq:TEICH1}
on~$\sT_0(M)$ and the complex structure~${\Jhat\mapsto-J\Jhat}$
is the symmetric bilinear form
\begin{equation}\label{eq:TEICH3}
\inner{\Jhat_1}{\Jhat_2}
= \int_M\Bigl(
\tfrac{1}{2}\trace\bigl(\Jhat_1\Jhat_2\bigr)
- f_1f_2 - g_1g_2
\Bigr)\rho_J.
\end{equation}
This is indefinite in general, so~$\sT_0(M)$ need not be K\"ahler.
If~$\om$ is a K\"ahler form with~${\om^\sn/\sn!=\rho_J}$, 
then the subspace of self-adjoint harmonic 
endomorphisms~${\Jhat=\Jhat^*\in\Om^{0,1}_J(M,TM)}$
is positive for~\eqref{eq:TEICH3} 
(and tangent to the Teich\-m\"uller space 
of~$\om$-compatible complex structures). 
Its symplectic complement is the negative
subspace of skew-adjoint harmonic endomorphisms.
The $2$-form~\eqref{eq:TEICH1} defines a 
symplectic connection on the space~$\sE_0(M)$ 
of isotopy classes of Ricci-flat K\"ahler 
structures, fibered over the space~$\sB_0(M)$ 
of isotopy classes of K\"ahlerable symplectic
forms with real first Chern class zero,
whose fiber over~${[\om]}$ is the space~$\sT_0(M,\om)$
of~$\om$-compatible (integrable)
complex structures~$J$ with~${\Ric_{\om,J}=0}$ 
modulo~${\Symp(M,\om)\cap\Diff_0(M)}$.

\medskip\noindent{\bf Theorem~E (A Connection).}
{\it The projection~${\sE_0(M)\to\sB_0(M)}$ 
is a submersion and the $2$-form~\eqref{eq:TEICH1} defines a symplectic 
connection on~$\sE_0(M)$.  The connection $1$-form~$\sA$
assigns to each Ricci-flat K\"ahler structure~${(\om,J)}$
and each closed~$2$-form~$\omhat$ the 
unique element~${\Jhat=\sA_{\om,J}(\omhat)\in\Om^{0,1}_J(M,TM)}$ 
that satisfies~${\bar\p_J\Jhat = 0}$
and~${\Lambda_\rho(J,\Jhat)=-d\inner{\omhat}{\om}\circ J}$
and~${\omhat-J^*\omhat = \inner{(\Jhat-\Jhat^*)\cdot}{\cdot}}$
and~${\Om_J(\Jhat,\Jhat')=0}$ 
for all~${\Jhat'\in\Om^{0,1}_J(M,TM)}$ 
with~${\bar\p_J\Jhat'=0}$ and~${\Jhat'=(\Jhat')^*}$.
The connection is $\Diff^+(M)$-equivariant and is given by
\begin{equation}\label{eq:CONNECTION}
\begin{split}
\sA_{\om,J}(\omhat) = \cL_vJ+\Jhat_0,\qquad
\inner{\Jhat_0\cdot}{\cdot} 
= \tfrac{1}{2}\bigl((\omhat-d\lambdahat)-J^*(\omhat-d\lambdahat)\bigr),
\end{split}
\end{equation}
where ${v\in\Vect(M)}$ and~${\lambdahat=\iota(v)\om\in\Om^1(M)}$
satisfy~${d^*(\omhat-d\lambdahat)=0}$, ${d^*\lambdahat=0}$.}

\begin{proof}
See Lemma~\ref{le:A} and Theorem~\ref{thm:CONNECTION}.
\end{proof}

The Weil--Petersson metric on the fiber~$\sT_0(M,\om)$ in
Theorem~E is K\"ahler and has been studied by many authors
(see e.g.~\cite{BCS,F,FS,KOISO,LS}, \cite{N}-\cite{TROMBA}, 
\cite{VIEHWEG,WANG} and the references therein).  
An important special case arises when~${H^{2,0}_J(M)=0}$
for all~${J\in\sJ_{\INT,0}(M)}$. In this case~${\sT_0(M)}$ is K\"ahler,
each polarized fiber~$\sT_0(M,\om)$ is an open subset of~$\sT_0(M)$,
the symplectic forms on the fibers agree on the overlaps
(as noted by Todorov~\cite[p~328]{TODOROV}),
and the connection is trivial.

\medskip\noindent{\bf Acknowledgement.}
Thanks to Simon Donaldson for suggesting this problem.
Thanks to Paul Biran, Ron Donagi, Andrew Kresch, Rahul Pandharipande,
Yanir Rubinstein, and Claire Voisin for helpful discussions.


\section{The Ricci form}\label{sec:RICCI}


\subsection*{Linear complex structures}

The standard orientation of~$\R^{2\sn}$ with the 
coordinates~$x_1,\dots,x_\sn,y_1,\dots,y_\sn$ is determined by the volume 
form~$dx_1\wedge dy_1\wedge\cdots\wedge dx_\sn\wedge dy_\sn$.
The space of linear complex structures on~$\R^{2\sn}$
compatible with the orientation is given by 
\begin{equation}\label{eq:Jn}
\begin{split}
\sJ_\sn
=
\left\{gJ_0g^{-1}\,\Big|\,g\in\SL(2\sn,\R)\right\},\qquad
J_0 := \left(\begin{array}{rr}
0 & -\one \\ \one & 0
\end{array}\right).
\end{split}
\end{equation}
This is a co-adjoint orbit equipped with a Hamiltonian~$\SL(2\sn,\R)$-action.
Abbreviate~${\rG:=\SL(2\sn,\R)}$ and~${\fg:=\Lie(\rG)=\fsl(2\sn,\R)}$
and note that~${\sJ_\sn\subset\fg}$.

\begin{lemma}\label{le:LCS}
The set~$\sJ_\sn\subset\R^{2\sn\times2\sn}$
is a connected $2\sn^2$-dimensional submani\-fold 
and its tangent space at~${J\in\sJ_\sn}$ is given by
\begin{equation}\label{eq:TJ}
\begin{split}
T_J\sJ_\sn 
= 
\bigl\{\Jhat\in\R^{2\sn\times2\sn}\,\big|\,
\Jhat J+J\Jhat = 0\bigr\} 
= 
\bigl\{[\xi,J]\,\big|\,\xi\in\fg\bigr\}.
\end{split}
\end{equation}
The formula~$\Jhat\mapsto-J\Jhat$ defines a complex structure on~$\sJ_\sn$
and the formula
\begin{equation}\label{eq:omJ}
\tau_J\bigl(\Jhat_1,\Jhat_2\bigr)
:= \tfrac12\trace\bigl(\Jhat_1J\Jhat_2\bigr)
= -\trace\bigl([\xi_1,\xi_2]J\bigr)
\end{equation}
for~${\xi_i\in\fg}$ and~${\Jhat_i:=[\xi_i,J]}$
defines a symplectic form~${\tau\in\Om^2(\sJ_\sn)}$.
The $\rG$-action ${\rG\times\sJ_\sn\to\sJ_\sn:(g,J)\mapsto gJg^{-1}}$
is Hamil\-tonian and is generated by the 
$\rG$-equi\-variant moment map~${\mu:\sJ_n\to\fg}$ 
given by~${\mu(J) = -J}$ for~${J\in\sJ_\sn}$.
\end{lemma}

\begin{proof}
The set~${\rH := \{h\in\SL(2n,\R)\,|\,hJ_0=J_0h\}}$
is a Lie subgroup of~$\rG$ and is isomorphic to 
the group of complex~${\sn\times\sn}$-matrices with determinant
in the unit circle.  So~${\dim\rH=2\sn^2-1}$ and~${\dim\rG=4\sn^2-1}$
and thus the homogeneous space~$\rG/\rH$ is a manifold of dimension~$2\sn^2$.
Since~$\rG$ is connected, so is~$\rG/\rH$.  Next we claim 
that the map~${\rG\to\R^{2\sn\times2\sn}:g\mapsto gJ_0g^{-1}}$ descends to 
a proper injective immersion~${\iota:\rG/\rH\to\R^{2\sn\times2\sn}}$. 
It is injective by definition.  To see that~$\iota$ is an immersion, 
observe that~${T_{[g]}\rG/\rH\cong g\fg/g\fh}$ 
and~${d\iota([g])[g\xi] = g[\xi,J_0]g^{-1}}$ for~${g\in\rG}$ and~${\xi\in\fg}$.
To prove that~$\iota$ is proper, choose~${g_k\in\rG}$ 
such that the sequence~${J_k:=g_kJ_0g_k^{-1}}$ 
converges to~$J_0$, and 
define~$h_k:=g_k^{-1}[e_1\cdots e_\sn\,J_ke_1\cdots J_ke_\sn]$,
where the vectors~$e_1,\dots,e_\sn\in\R^{2\sn}$ form the standard
basis of~${\R^\sn\times\{0\}}$.  Then~$h_k\in\rH$ for~$k$ sufficiently large
and~${\lim_{k\to\infty}g_kh_k=\one}$.  This shows that the 
map~${\iota:\rG/\rH\to\R^{2\sn\times2\sn}}$ is a proper injective 
immersion.  Hence its image~${\sJ_\sn=\iota(\rG/\rH)}$ is 
a connected $2\sn^2$-dimensional submanifold of~$\R^{2\sn\times2\sn}$. 

\bigbreak

Now let~$J\in\sJ_\sn$.  Then~$gJg^{-1}\in\sJ_\sn$ for all~$g\in\rG$
and so~${[\xi,J]\in T_J\sJ_\sn}$ for all~${\xi\in\fg}$.
Thus
$
\{[\xi,J]\,|\,\xi\in\fg\}
\subset
T_J\sJ_\sn 
\subset
\{\Jhat\in\R^{2\sn\times2\sn}\,|\,
\Jhat J+J\Jhat = 0\}.
$
Since all three spaces have dimension~$2\sn^2$,
equality holds and this proves~\eqref{eq:TJ}. 
The formula~\eqref{eq:omJ} follows by direct calculation. 
To show that the $2$-form~$\tau$ in~\eqref{eq:omJ}
is nondegenerate, let~${\Jhat=[\xi,J]\in T_J\sJ_\sn\setminus\{0\}}$
and define~${\eta:=[\xi,J]^T}$ and~${\Jhat':=[\eta,J]}$.
Then~$\tau_J(\Jhat,\Jhat') = \trace(\eta[\xi,J]) = \trace([\xi,J]^T [\xi,J]) > 0$.  
The $2$-form~$\tau$ is closed and the 
complex structure~$\Jhat\mapsto-J\Jhat$ is integrable
by Lemma~\ref{le:TORSION}, as both structures are preserved 
by the torsion-free connection
$$
\Nabla{t}\Jhat
:= \tfrac{d}{dt}\Jhat
+ \tfrac{1}{2}\Jhat J\dot J
+ \tfrac{1}{2}\dot{J}J\Jhat.
$$
The map~${\sJ_\sn\to\fg:J\mapsto\mu(J):=-J}$
is a moment map for the $\rG$-action because
$\tau_J([\xi,J],\Jhat)= - \trace(\xi\Jhat)=\trace((d\mu(J)\Jhat)\xi)$
for~${J\in\sJ_\sn}$, ${\Jhat\in T_J\sJ_\sn}$,
and~${\xi\in\fg}$. This proves Lemma~\ref{le:LCS}.
\end{proof}

\begin{remark}\label{rmk:complexJ}\rm
The symplectic form~$\tau$
in~\eqref{eq:omJ} is a $(1,1)$-form with respect 
to the complex structure~$\Jhat\mapsto-J\Jhat$. 
For~${\sn>1}$ it is not a K\"ahler form,
because the bilinear form
$
{\inner{\Jhat_1}{\Jhat_2}
=\tfrac12\trace(\Jhat_1\Jhat_2)}
$
is indefinite on each tangent space.
\end{remark}

\begin{remark}\label{rmk:compatibleJ}\rm
Let~${\om_0:=\sum_{i=1}^\sn dx_i\wedge dy_i}$
denote the standard symplectic form on~$\R^{2\sn}$
and consider the space of $\om_0$-compatible 
linear complex structures
\begin{equation}\label{eq:compatibleJ}
\sJ_{\sn,0} 
:= \left\{J\in\sJ_\sn\,\bigg|\,
\begin{array}{l}
J^*\om_0=\om_0\mbox{ and }
\om_0(\zeta,J\zeta)>0\\
\mbox{for all }
\zeta\in\R^{2n}\setminus\{0\}
\end{array}\right\}.
\end{equation}
This is a complex submanifold of~$\sJ_\sn$
of real dimension~$\sn^2+\sn$ and the symplectic 
form~\eqref{eq:omJ} restricts to a K\"ahler form 
on~$\sJ_{\sn,0}$.  The symplectic linear group~$\Sp(2\sn)$ 
acts on~$\sJ_{\sn,0}$ by K\"ahler isometries
and a moment map~${\mu:\sJ_{\sn,0}\to\fsp(2\sn)}$
for this action is again given by~${\mu(J) = - J}$.
\end{remark}

\begin{remark}\label{rmk:siegelJ}\rm
The group~$\Sp(2\sn)$ acts on Siegel upper half space
${\sS_\sn\subset\C^{\sn\times\sn}}$ of symmetric matrices
with positive definite imaginary part via
$$
g_*Z := (AZ +B)(CZ+D)^{-1},\qquad
g=: \left(\begin{array}{cc} A & B \\ C & D\end{array}\right)
$$
for~${g\in\Sp(2\sn)}$ and~${Z\in\sS_\sn}$.
There is a unique $\Sp(2\sn)$-equivariant diffeomorphism
from~$\sS_\sn$ to $\sJ_{\sn,0}$ that sends~${\bi\one\in\sS_\sn}$
to~${J_0\in\sJ_{\sn,0}}$.  It is given by 
$$
J(Z) = \left(\begin{array}{cc}
XY^{-1} & -Y-XY^{-1}X \\ Y^{-1} & -Y^{-1}X
\end{array}\right)\in\sJ_{\sn,0},\qquad
Z=X+\bi Y\in\sS_\sn.
$$
This diffeomorphism is a K\"ahler isometry with 
respect to the K\"ahler metric on~$\sS_\sn$ given 
by~${\abs{\Zhat}^2=\trace((Y^{-1}\Xhat)^2+(Y^{-1}\Yhat)^2)}$
for~${\Zhat=\Xhat+\bi\Yhat\in T_Z\sS_\sn}$.
\end{remark}


\subsection*{Definition of the Ricci form}

By Lemma~\ref{le:LCS} the space~$\sJ_\sn$ fits as a fiber 
into the general framework developed by Donaldson~\cite{DON3}.
Starting from this observation we will show that 
the action of the group of exact volume preserving 
diffeomorphisms on the space of almost complex structures 
is a Hamiltonian group action with twice the Ricci form 
as a moment map. Let~$M$ be a closed connected 
oriented $2\sn$-manifold.  Assume~$M$ admits an almost 
complex structure compatible with the orientation and denote 
the space of such almost complex structures by
\begin{equation}\label{eq:J}
\sJ(M) := \left\{J\in\Om^0(M,\End(TM))\,\Bigg|\,
\begin{array}{l}
J^2=-\one\mbox{ and}\\
J\mbox{ is compatible with}\\
\mbox{the orientation of }M
\end{array}\right\}.
\end{equation}
Thus~$\sJ(M)$ is the space of sections of a bundle
each of whose fibers is equipped with a natural symplectic form
by Lemma~\ref{le:LCS}.  It can be viewed formally as
an infinite-dimensional manifold whose tangent space
at~$J$ is the space $T_J\sJ(M) = \{\Jhat\in\Om^0(M,\End(TM))\,|\,
\Jhat J + J\Jhat = 0\} = \Om^{0,1}_J(M,TM)$ of complex
anti-linear $1$-forms on~$M$ with values in~$TM$.
Every positive volume form~${\rho\in\Om^{2\sn}(M)}$ 
determines a symplectic form~$\Om_\rho$ on~$\sJ(M)$ defined by 
\begin{equation}\label{eq:OMRHO}
\Om_{\rho,J}(\Jhat_1,\Jhat_2) 
:= \tfrac12\int_M\trace\left(\Jhat_1J\Jhat_2\right)\rho
\end{equation}
for~${J\in\sJ(M)}$ and~${\Jhat_1,\Jhat_2\in T_J\sJ(M)}$. The  
group~${\sG=\Diff(M,\rho)}$ of volume preserving diffeomorphisms 
acts on~$\sJ(M)$ contravariantly by~${J\mapsto\phi^*J}$
for~${\phi\in\sG}$ and~${J\in\sJ(M)}$.  
This action preserves the symplectic form~$\Om_\rho$.  

\begin{definition}[{\bf Ricci Form}]\label{def:RICCI}
Fix a positive volume form~${\rho\in\Om^{2\sn}(M)}$,
an almost complex structure~${J\in\sJ(M)}$,
and a torsion-free $\rho$-connection $\nabla$ on~$TM$.
The {\bf Ricci form} of the pair~$(\rho,J)$
is the $2$-form
\begin{equation}\label{eq:RICCI}
\Ric_{\rho,J} := \tfrac{1}{2}\left(\tau^\nabla_J+d\lambda^\nabla_J\right),
\end{equation}
where~${\tau_J^\nabla\in\Om^2(M)}$ and~${\lambda_J^\nabla\in\Om^1(M)}$ 
are defined by
\begin{equation}\label{eq:RICCI2}
\begin{split}
\tau^\nabla_J(u,v)
&:= 
\tfrac{1}{2}\trace\bigl((\Nabla{u}J)J(\Nabla{v}J)\bigr)
+ \trace\bigl(JR^\nabla(u,v)\bigr), \\
\lambda^\nabla_J(u) &:= \trace\bigl((\nabla J)u\bigr)
\end{split}
\end{equation}
for~${u,v\in\Vect(M)}$.
For~${\Jhat\in\Om^{0,1}_J(M,TM)}$
define~${\Lambda_\rho(J,\Jhat)\in\Om^1(M)}$ by
\begin{equation}\label{eq:LAMBDA}
\bigl(\Lambda_\rho(J,\Jhat)\bigr)(u)
:= \trace\bigl((\nabla\Jhat)u+\tfrac12\Jhat J\Nabla{u}J\bigr)
\qquad\mbox{for }u\in\Vect(M). 
\end{equation}
\end{definition}


\subsection*{The Ricci form as a moment map}

The next theorem is the main result of this section.
It asserts that the action of the subgroup
\begin{equation}\label{eq:GEX}
\sG^\ex 
:= 
\left\{
\phi\in\Diff(M)\,
\left|\,
\begin{array}{l}
\mbox{there exists a smooth isotopy}\\
{[0,1]\times\Diff(M):t\mapsto\phi_t}\\
\mbox{and a smooth family of vector fields}\\
{[0,1]\to\Vect(M):t\mapsto v_t}\\
\mbox{such that }\iota(v_t)\rho\mbox{ is exact for all }t\\
\mbox{and }\p_t\phi_t=v_t\circ\phi_t\mbox{ for all }t\\
\mbox{and }\phi_0=\id\mbox{ and }\phi_1=\phi
\end{array}
\right.
\right\}
\end{equation}
of exact volume preserving diffeomorphisms on~$\sJ(M)$ 
is a Hamiltonian group action and is generated by 
the~$\sG$-equivariant moment map which assigns 
to each~${J\in\sJ(M)}$ twice the Ricci form~$\Ric_{\rho,J}$. 
The moment map must take values in the {\it dual space} 
of the {\it Lie algebra}
$$
\Lie(\sG^\ex) 
= \Vect^\ex(M,\rho) 
= \left\{v\in\Vect(M)\,|\,\iota(v)\rho\mbox{ is exact}\right\}.
$$
Every~$(2\sn-2)$-form~${\alpha\in\Om^{2\sn-2}(M)}$ determines an 
exact divergence-free vector field~${v_\alpha\in\Vect^\ex(M,\rho)}$
via
$$
\iota(v_\alpha)\rho=d\alpha.
$$  
Thus~$\Vect^\ex(M,\rho)$ can be identified 
with the quotient of the space~$\Om^{2\sn-2}(M)$
by the space of closed~$(2\sn-2)$-forms on~$M$.
Its dual space can be viewed formally as the space 
of exact $2$-forms on~$M$, in that every exact 
$2$-form~$\tau$ on~$M$ determines a linear functional
$$
\Vect^\ex(M,\rho)\to\R:
v_\alpha\mapsto\int_M \tau\wedge\alpha.
$$
With this understood, equation~\eqref{eq:RICHATRHO} 
in the following theorem is the assertion that
the map~${J\mapsto2\Ric_{\rho,J}}$ is a moment map
for the action of~$\sG^\ex$ on~$\sJ(M)$.  
In general, however, the Ricci form is only closed and not exact;
only its differential in the direction of an infinitesimal 
almost complex structure is always exact.
Thus the map~${J\mapsto2\Ric_{\rho,J}}$ is only a moment 
in the strict sense of the word when restricted to the space
of almost complex structures with real first Chern class zero.
One could attempt to rectify this situation by subtracting 
a closed $2$-form in the appropriate cohomology class from 
the Ricci form, however such a modification would destroy 
the~$\sG^\ex$-equivariance of the moment map unless~$M$
has real dimension two.

\begin{theorem}\label{thm:RICCI}
Let~${\rho\in\Om^{2\sn}(M)}$ be a positive volume form, 
let~${J\in\sJ(M)}$, and let~${\Jhat\in\Om^{0,1}_J(M,TM)}$.  
Then the following holds. 

\smallskip\noindent{\bf (i)}  
The Ricci form~$\Ric_{\rho,J}$ and the $1$-form~$\Lambda_\rho(J,\Jhat)$
are independent of the choice of the torsion-free $\rho$-connection~$\nabla$ 
used to define them.  Moreover,
\begin{equation}\label{eq:RICRHOF}
\Ric_{e^f\rho,J}=\Ric_{\rho,J} + \tfrac{1}{2}d(df\circ J),\quad
\Lambda_{e^f\rho}(J,\Jhat)=\Lambda_\rho(J,\Jhat)+df\circ\Jhat
\end{equation}
for all~${f\in\Om^0(M)}$ and the Ricci form and~$\Lambda_\rho$
satisfy the naturality condition
\begin{equation}\label{eq:NATURALITY}
\phi^*\Ric_{\rho,J} = \Ric_{\phi^*\rho,\phi^*J},\qquad
\phi^*\Lambda_\rho(J,\Jhat) = \Lambda_{\phi^*\rho}(\phi^*J,\phi^*\Jhat)
\end{equation}
for all~${\phi\in\Diff(M)}$.

\smallskip\noindent{\bf (ii)}
Every vector field~${v\in\Vect(M)}$ satisfies 
\begin{equation}\label{eq:LAMBDARHO}
\int_M\Lambda_\rho(J,\Jhat)\wedge\iota(v)\rho
= \tfrac{1}{2}\int_M\trace\bigl(\Jhat J\cL_vJ\bigr)\rho.
\end{equation}
Moreover, every smooth path~${\R\to\sJ(M):t\mapsto J_t}$ 
of almost complex structures with~${J_0=J}$ 
and~${\left.\tfrac{d}{dt}\right|_{t=0}J_t=\Jhat}$
satisfies the equations
\begin{equation}\label{eq:RICHAT}
\Richat_\rho(J,\Jhat) 
:= \left.\frac{d}{dt}\right|_{t=0}\Ric_{\rho,J_t} 
= \tfrac{1}{2}d\bigl(\Lambda_\rho(J,\Jhat)\bigr)
\end{equation}
and
\begin{equation}\label{eq:RICHATRHO}
\int_M2\Richat_\rho(J,\Jhat) \wedge\alpha
= \tfrac12\int_M\trace\bigl(\Jhat J\cL_{v_\alpha}J\bigr)\rho
\end{equation}
for~${\alpha\in\Om^{2\sn-2}(M)}$, where~${v_\alpha\in\Vect(M)}$
is defined by~${\iota(v_\alpha)\rho = d\alpha}$.

\smallskip\noindent{\bf (iii)}
Let~$\om\in\Om^2(M)$ be a nondegenerate $2$-form 
compatible with~$J$ such that~$\om^\sn/\sn!=\rho$, 
let~$\nabla$ be the Levi-Civita connection of the 
Riemannian metric~${\inner{\cdot}{\cdot}=\om(\cdot,J\cdot)}$,
and define
\begin{equation}\label{eq:TABLA}
\tabla:=\nabla-\tfrac{1}{2}J\nabla J.
\end{equation}
Then~$\tabla$ is a Hermitian connection and
\begin{equation}\label{eq:RICCItilde}
\Ric_{\rho,J}=\tfrac12\bigl(\trace(JR^\tabla) + d\lambda^\nabla_J\bigr).
\end{equation}
Thus~$\Ric_{\rho,J}$ is closed and represents
the class~${2\pi c_1(TM,J)}$.  Moreover,
\begin{equation}\label{eq:RICSYMP}
d\om=0\qquad\implies\qquad
\lambda^\nabla_J=0,\qquad
\Ric_{\rho,J}=\tfrac12\trace(JR^\tabla).
\end{equation}
\end{theorem}

\begin{proof}
We prove part~(i).   Choose a smooth function
$$
[0,1]\times M\to\R:(t,p)\mapsto f_t(p)
$$ 
with~${f_0=0}$ and~${f_1=f}$,
define~${\rho_t:=e^{f_t}\rho}$ for~${0\le t\le 1}$, and choose 
a smooth path of torsion-free connections~$\Nabla{t}$ 
on~$TM$ such that~${\Nabla{t}\rho_t=0}$ for all~$t$.  
For~${0\le t\le 1}$ define the $1$-forms~${A_t\in\Om^1(M,\End(TM))}$
and~${\alpha_t\in\Om^1(M)}$ by
\begin{equation}\label{eq:altau}
A_t := \tfrac{d}{dt}\Nabla{\,t},\qquad
\alpha_t(u) := \trace\bigl(JA_t(u)\bigr)
\end{equation}
for~${u\in\Vect(M)}$. Then, for all~$t$ 
and all~${u,v\in\Vect(M)}$, we have
\begin{equation}\label{eq:A}
A_t(u)v=A_t(v)u,\qquad
\trace(A_t(u)) = d(\p_tf_t)(u),
\end{equation}
\begin{equation}\label{eq:RAt}
\tfrac{d}{dt}\Nabla{t,u}J=[A_t(u),J],\qquad
\tfrac{d}{dt}R^{\Nabla{\,t}}=d^{\Nabla{\,t}}\!\!A_t.
\end{equation}
It follows from~\eqref{eq:RICCI2}, \eqref{eq:altau}, 
\eqref{eq:A}, and~\eqref{eq:RAt} that
\begin{equation*}
\begin{split}
\tfrac{d}{dt}\tau^{\Nabla{\,t}}_J(u,v) 
&= 
\trace\bigl((\Nabla{t,u}J)A_t(v)-(\Nabla{t,v}J)A_t(u)\bigr)
+ \trace\bigl(Jd^{\Nabla{\,t}}\!\!A_t(u,v)\bigr) \\
&=
\trace\bigl(d^{\Nabla{\,t}}(JA_t)(u,v)\bigr) \\
&=
d\alpha_t(u,v)
\end{split}
\end{equation*}
and
\begin{equation*}
\begin{split}
\left.\tfrac{d}{dt}\right|_{t=0}\lambda^{\Nabla{\,t}}_J(u)
&= 
\trace\bigl([A_t,J]u\bigr)  \\
&= 
\trace\bigl(A_t(Ju)\bigr) - \trace(JA_t(u)) \\
&= 
d(\p_tf_t)(Ju) - \alpha_t(u)
\end{split}
\end{equation*}
for all~$t$ and all~${u,v\in\Vect(M)}$.  Hence
$$
\left.\tfrac{d}{dt}\right|_{t=0}(\tau^{\Nabla{\,t}}_J + d\lambda^{\Nabla{\,t}}_J)
= d\bigl(d(\p_tf_t)\circ J\bigr).
$$
Integrate this formula to obtain the first equation in~\eqref{eq:RICRHOF}
and consider the case where~$\rho_t=\rho$ is independent of~$t$
to deduce that the $2$-form~$\Ric_{\rho,J}$ is independent 
of the choice of the torsion-free $\rho$-connection~$\nabla$ used to define it.
Moreover, it follows from~\eqref{eq:altau}, \eqref{eq:A}, 
and~\eqref{eq:RAt} that 
\begin{equation*}
\begin{split}
\tfrac{d}{dt}\trace\bigl((\Nabla{t}\Jhat)u+\tfrac12\Jhat J\Nabla{t,u}J\bigr)
&= 
\trace\bigl([A_t,\Jhat]u+\tfrac{1}{2}\Jhat J[A_t(u),J]\bigr)  \\
&= 
\trace\bigl(A_t(\Jhat u)\bigr) \\
&= 
d(\p_tf_t)(\Jhat u).
\end{split}
\end{equation*}
for all~$t$ and all~${u\in\Vect(M)}$. 
Integrate this formula to obtain the second equation in~\eqref{eq:RICRHOF}
and consider the case where~$\rho_t=\rho$ is independent of~$t$
to deduce that the $1$-form~$\Lambda_\rho(J,\Jhat)$ is independent 
of the choice of the torsion-free $\rho$-connection~$\nabla$ used to define it.
The naturality condition~\eqref{eq:NATURALITY} follows directly ftom 
the definitions and this proves~(i).

\bigbreak

To prove part~(ii) we use the formulas
\begin{equation}\label{eq:nablaurho}
\trace(\nabla u)\rho=d\iota(u)\rho,
\end{equation}
\begin{equation}\label{eq:LvJac}
(\cL_vJ)u = J\Nabla{u}v-\Nabla{Ju}v+(\Nabla{v}J)u
\end{equation}
for~${u,v\in\Vect(M)}$. By~\eqref{eq:LvJac}, we have
\begin{equation*}
\begin{split}
\trace\bigl(\Jhat J\cL_vJ\bigr)
&=
\trace\bigl(-\Jhat\nabla v-\Jhat J\Nabla{J\cdot}v+\Jhat J\Nabla{v}J\bigr) \\
&=
\trace\bigl(-2\Jhat\nabla v+\Jhat J\Nabla{v}J\bigr)
\end{split}
\end{equation*}
for all~${u,v\in\Vect(M)}$. Here the second equality holds 
because two endomorphisms~$\Phi$ and~$-J\Phi J$ are conjugate 
and so have the same trace. Thus
\begin{equation*}
\begin{split}
\Lambda_\rho(J,\Jhat)(v)
&= 
\trace\bigl((\nabla \Jhat)v + \tfrac{1}{2}\Jhat J\Nabla{v}J\bigr) \\
&= 
\trace\bigl(\nabla(\Jhat v) - \Jhat\nabla v + \tfrac{1}{2}\Jhat J\Nabla{v}J\bigr) \\
&=
\trace\bigl(\nabla(\Jhat v)\bigr) 
+ \tfrac{1}{2}\trace\bigl(\Jhat J\cL_vJ\bigr)
\end{split}
\end{equation*}
for all~${v\in\Vect(M)}$.
Hence it follows from~\eqref{eq:nablaurho} with~${u=\Jhat v}$ that
$$
\int_M\Lambda_\rho(J,\Jhat)\wedge\iota(v)\rho 
= \int_M\Lambda_\rho(J,\Jhat)(v)\rho 
= \tfrac{1}{2}\int_M\trace\bigl(\Jhat J\cL_vJ\bigr)\rho
$$
for all~${v\in\Vect(M)}$.  This proves~\eqref{eq:LAMBDARHO}.

Now fix a torsion-free $\rho$-connection~$\nabla$ and abbreviate
$$
\lambdahat(u) 
:= \trace((\nabla\Jhat)u)
= \left.\tfrac{d}{dt}\right|_{t=0}\lambda^\nabla_{J_t}(u),\qquad
\betahat(u) := \tfrac{1}{2}\trace\bigl(\Jhat J\Nabla{u}J\bigr)
$$
for~${u\in\Vect(M)}$.   
Then~${\Lambda_\rho(J,\Jhat)=\lambdahat+\betahat}$ and
\begin{equation*}
\begin{split}
&d\betahat(u,v) \\
&=
\tfrac{1}{2}\cL_u\trace\bigl(\Jhat J\Nabla{v}J\bigr)
- \tfrac{1}{2}\cL_v\trace\bigl(\Jhat J\Nabla{u}J\bigr)
+ \tfrac{1}{2}\trace\bigl(\Jhat J\Nabla{[u,v]}J\bigr) \\
&=
\tfrac{1}{2}\trace\bigl((\Nabla{u}(\Jhat J))\Nabla{v}J\bigr)
- \tfrac{1}{2}\trace\bigl((\Nabla{v}(\Jhat J))\Nabla{u}J\bigr) \\
&\quad
+ \tfrac{1}{2}\trace\bigl(\Jhat J\bigl(\Nabla{u}\Nabla{v}J
-\Nabla{v}\Nabla{u}J + \Nabla{[u,v]}J\bigr)\bigr) \\
&=
\tfrac{1}{2}\trace\bigl((\Nabla{u}\Jhat)J(\Nabla{v}J)\bigr)
- \tfrac{1}{2}\trace\bigl((\Nabla{v}\Jhat)J(\Nabla{u}J)\bigr) 
+ \tfrac{1}{2}\trace\bigl(\Jhat J[R^{\nabla}(u,v),J]\bigr) \\
&=
\tfrac{1}{2}\trace\bigl((\Nabla{u}\Jhat)J(\Nabla{v}J)\bigr)
+ \tfrac{1}{2}\trace\bigl((\Nabla{u}J)J(\Nabla{v}\Jhat)\bigr)
+ \trace\bigl(\Jhat R^{\nabla}(u,v)\bigr) \\
&=
\left.\tfrac{d}{dt}\right|_{t=0}\tau^\nabla_{J_t}(u,v)
\end{split}
\end{equation*}
for all~${u,v\in\Vect(M)}$.  
Since~${\Ric_{\rho,J_t}=\tfrac{1}{2}(\tau^\nabla_{J_t}+d\lambda^\nabla_{J_t})}$
this proves~\eqref{eq:RICHAT}.  Equation~\eqref{eq:RICHATRHO} follows directly
from~\eqref{eq:LAMBDARHO}, \eqref{eq:RICHAT}, and Stokes' theorem
and this proves part~(ii).

\bigbreak

We prove part~(iii).
The connection~$\tabla$ in~\eqref{eq:TABLA} will in general 
no longer be torsion-free.  However, since the 
endomorphism~${J\Nabla{u}J}$ is skew-adjoint
for all~${u\in\Vect(M)}$, it preserves the Riemannian metric 
on~$M$ and the volume form~$\rho$. In addition it preserves 
the almost complex structure~$J$ because
$$
\Tabla{u}J 
= \Nabla{u}J-\tfrac12[J\Nabla{u}J,J] 
=  \Nabla{u}J -\tfrac12J(\Nabla{u}J)J + \tfrac12JJ\Nabla{u}J
= 0
$$
for all~${u\in\Vect(M)}$.  Next we compute the curvature tensor 
of~$\tabla$.  Fix three vector fields~${u,v,w\in\Vect(M)}$.
Then~${\Tabla{v}w = \Nabla{v}w - \tfrac12J(\Nabla{v}J)w}$
and so
\begin{equation*}
\begin{split}
\Tabla{u}\Tabla{v}w 
&= 
\Tabla{u}\bigl(\Nabla{v}w - \tfrac12J(\Nabla{v}J)w\bigr) 
= 
\Tabla{u}\Nabla{v}w - \tfrac12J\Tabla{u}\bigl((\Nabla{v}J)w\bigr) \\
&= 
\Nabla{u}\Nabla{v}w 
- \tfrac12J(\Nabla{u}J)\Nabla{v}w 
- \tfrac12J\Nabla{u}\bigl((\Nabla{v}J)w\bigr) 
- \tfrac14(\Nabla{u}J)(\Nabla{v}J)w \\
&= 
\Nabla{u}\Nabla{v}w 
- \tfrac12J\bigl(\Nabla{u}\Nabla{v}J\bigr)w 
- \tfrac14(\Nabla{u}J)(\Nabla{v}J)w \\
&\quad
- \tfrac12J(\Nabla{u}J)\Nabla{v}w 
- \tfrac12J(\Nabla{v}J)\Nabla{u}w.
\end{split}
\end{equation*}
Hence
\begin{equation*}
\begin{split}
R^\tabla(u,v)w
&=
\Tabla{u}\Tabla{v}w - \Tabla{u}\Tabla{v}w + \Tabla{[u,v]}w \\
&=
\Nabla{u}\Nabla{v}w 
- \tfrac12J\bigl(\Nabla{u}\Nabla{v}J\bigr)w 
- \tfrac14(\Nabla{u}J)(\Nabla{v}J)w \\
&\quad
- \Nabla{v}\Nabla{u}w 
+ \tfrac12J\bigl(\Nabla{v}\Nabla{u}J\bigr)w 
+ \tfrac14(\Nabla{v}J)(\Nabla{u}J)w \\
&\quad
+ \Nabla{[u,v]}w - \tfrac12J(\Nabla{[u,v]}J)w  \\
&=
R^\nabla(u,v)w - \tfrac{1}{2}J[R^\nabla(u,v),J]w 
- \tfrac{1}{4}[\Nabla{u}J,\Nabla{v}J]w \\
&=
\tfrac{1}{2}R^\nabla(u,v)w - \tfrac{1}{2}JR^\nabla(u,v)Jw 
- \tfrac{1}{4}[\Nabla{u}J,\Nabla{v}J]w.
\end{split}
\end{equation*}
This implies
\begin{equation}\label{eq:Rtabla}
\begin{split}
JR^\tabla(u,v)
&=
\tfrac{1}{2}JR^\nabla(u,v) 
+ \tfrac{1}{2}R^\nabla(u,v)J 
- \tfrac{1}{4}J[\Nabla{u}J,\Nabla{v}J]
\end{split}
\end{equation}
and hence
\begin{equation}\label{eq:Rtablatrace}
\begin{split}
\trace\bigl(JR^\tabla(u,v)\bigr)
=
\trace\bigl(JR^\nabla(u,v)\bigr) 
+ \tfrac12\trace\bigl((\Nabla{u}J)J(\Nabla{v}J)\bigr).
\end{split}
\end{equation}
Thus~$\trace(JR^\tabla)=\tau^\nabla_J$
and this proves~\eqref{eq:RICCItilde}.  
Since~$\tabla$ is a Hermitian connection,
the $2$-form~${\trace(\tfrac{1}{4\pi}JR^\tabla)
=\trace^c(\tfrac{1}{2\pi}JR^\tabla)\in\Om^2(M)}$
is closed and represents the first Chern class 
of~$(TM,J)$. 

If~$\om$ is closed, then~${\Nabla{Jv}J = -J(\Nabla{v}J)}$ 
for every vector field~${v\in\Vect(M)}$ 
by~\cite[Lemma~4.1.14]{MS},
so the endomorphism~${v\mapsto(\Nabla{v}J)u}$ 
anti-commutes with~$J$ and therefore has trace zero. 
Hence~$\lambda^\nabla_J=0$.
This proves part~(iii) and Theorem~\ref{thm:RICCI}. 
\end{proof}


For~${u\in\Vect(M)}$ define~${f_u:=f_{\rho,u}:=\DIV_\rho(u)\in\Om^0(M)}$, so that
\begin{equation}\label{eq:fv}
f_u\rho = d\iota(u)\rho.
\end{equation}

\begin{theorem}\label{thm:LAMBDA}
Let~${\rho\in\Om^{2\sn}(M)}$ be a positive volume form, let~${J\in\sJ(M)}$,
and let~${u\in\Vect(M)}$.  Then
\begin{equation}\label{eq:LAMBDAu}
\Lambda_\rho(J,\cL_uJ) = 2\iota(u)\Ric_{\rho,J} - df_u\circ J + df_{Ju}.
\end{equation}
Moreover, every smooth map~${\R^2\to\sJ(M):(s,t)\mapsto J(s,t})$ satisfies
\begin{equation}\label{eq:LAMBDAst}  
\p_s\Lambda_\rho(J,\p_tJ) - \p_t\Lambda_\rho(J,\p_sJ) 
+ \tfrac{1}{2}d\trace\bigl((\p_sJ)J(\p_tJ)\bigr) = 0.
\end{equation}
\end{theorem}

\begin{proof}
The proof has six steps.

\medskip\noindent{\bf Step~1.}
{\it We prove~\eqref{eq:LAMBDAst}.} 

\medskip\noindent
Let~${v\in\Vect(M)}$.  Then it follows from 
equation~\eqref{eq:LAMBDARHO} that
\begin{equation*}
\begin{split}
&\int_M\bigl( 
\p_s\Lambda_\rho(J,\p_tJ) - \p_t\Lambda_\rho(J,\p_sJ) 
\bigr)\wedge\iota(v)\rho \\
&=
\tfrac{1}{2}\p_s\int_M\trace\bigl((\p_tJ)J(\cL_vJ)\bigr)\rho
- \tfrac{1}{2}\p_t\int_M\trace\bigl((\p_sJ)J(\cL_vJ)\bigr)\rho \\
&=
\tfrac{1}{2}\int_M\trace\bigl(
(\cL_v\p_tJ)J(\p_sJ) 
+ (\p_tJ)(\cL_vJ)(\p_sJ)
+ (\p_tJ)J(\cL_v\p_sJ)
\bigr)\rho  \\
&=
\tfrac{1}{2}\int_M \bigl(\cL_v\trace\bigl((\p_tJ)J(\p_sJ)\bigr)\bigr)\rho  
=
\tfrac{1}{2}\int_M d\trace\bigl((\p_tJ)J(\p_sJ)\bigr)\wedge\iota(v)\rho.
\end{split}
\end{equation*}
This proves Step~1.

\medskip\noindent{\bf Step~2.}
{\it ${d\Lambda_\rho(J,\cL_uJ) = 2d\iota(u)\Ric_{\rho,J} - d(df_u\circ J)}$.}

\medskip\noindent
Let~$\phi_t$ be the flow of~$u$. 
Then~${\phi_t^*\Ric_{\rho,J}=\Ric_{\phi_t^*\rho,\phi_t^*J}}$
by part~(i) of Theorem~\ref{thm:RICCI}.
Differentiate this equation and use 
parts~(i) and~(ii) of Theorem~\ref{thm:RICCI} 
to get~${d\iota(u)\Ric_{\rho,J} 
= \Richat_\rho(J,\cL_uJ) + \tfrac{1}{2}d(df_u\circ J) 
= \tfrac{1}{2}\bigl(d\Lambda_\rho(J,\cL_uJ) + d(df_u\circ J)\bigr)}$.

\medskip\noindent{\bf Step~3.}
{\it Suppose~${\iota(u)\rho}$ is exact. 
Then~$u$ satisfies~\eqref{eq:LAMBDAu}.}

\medskip\noindent
Choose~${\alpha\in\Om^{2\sn-2}(M)}$
such that~${\iota(u)\rho=d\alpha}$.
Then, for all~${v\in\Vect(M)}$,
\begin{equation*}
\begin{split}
&\int_M 2\iota(u)\Ric_{\rho,J} \wedge \iota(v)\rho
=
\int_M 2\Ric_{\rho,J}\wedge\iota(v)d\alpha 
= 
- \int_M 2d\iota(v)\Ric_{\rho,J}\wedge\alpha \\
&=
- \int_M d\bigl(\Lambda_\rho(J,\cL_vJ)+df_v\circ J\bigr)\wedge\alpha 
=
- \int_M \bigl(\Lambda_\rho(J,\cL_vJ)+df_v\circ J\bigr)\wedge \iota(u)\rho \\
&= 
\int_M \bigl(\Lambda_\rho(J,\cL_uJ)-df_{Ju}\bigr)\wedge\iota(v)\rho.
\end{split}
\end{equation*}
Here the third equality follows from Step~2. This proves Step~3.

\bigbreak

\medskip\noindent{\bf Step~4.}
{\it Let~${\lambda\in\Om^0(M)}$ and~${\Jhat\in\Om^{0,1}_J(M,TM)}$. Then}
\begin{equation}\label{eq:flau}
f_{\lambda u} = \lambda f_u +d\lambda(u),
\end{equation}
\begin{equation}\label{eq:Lala}
\Lambda_\rho(J,\lambda\Jhat) 
= \lambda\Lambda_\rho(J,\Jhat) + d\lambda\circ\Jhat,
\end{equation}
\begin{equation}\label{eq:LlauJ}
\cL_{\lambda u}J = \lambda\cL_u J + \Jhat_u,\qquad
\Jhat_u := Ju\otimes d\lambda - u\otimes d\lambda\circ J.
\end{equation}
\noindent
\eqref{eq:flau} and~\eqref{eq:Lala} follow from 
the definitions and~\eqref{eq:LlauJ} follows from~\eqref{eq:LvJac}.

\medskip\noindent{\bf Step~5.}
{\it  If~$u$ satisfies~\eqref{eq:LAMBDAu}
and~${\lambda\in\Om^0(M)}$, then~$\lambda u$
satisfies~\eqref{eq:LAMBDAu}.}

\medskip\noindent
Let~${\Jhat_u\in\Om^{0,1}_J(M,TM)}$ be as in~\eqref{eq:LlauJ}.
We prove the identity
\begin{equation}\label{eq:LAu}
\begin{split}
\Lambda_\rho(J,\Jhat_u) + d\lambda\circ\cL_uJ
&=
f_{Ju}d\lambda - f_ud\lambda\circ J
+ d\cL_{Ju}\lambda - d\cL_u\lambda\circ J.
\end{split}
\end{equation}
To see this, let~${v,w\in\Vect(M)}$.  Then
\begin{equation*}
\begin{split}
(\Nabla{w}\Jhat_u)v
&=
\Nabla{w}\bigl(d\lambda(v)Ju - d\lambda(Jv)u\bigr)
- d\lambda(\Nabla{w}v)Ju + d\lambda(J\Nabla{w}v) u \\
&=
d\lambda(v)\Nabla{w}(Ju) + (\cL_w\cL_v\lambda)Ju - (\cL_{\Nabla{w}v}\lambda)Ju \\
&\quad
- d\lambda(Jv)\Nabla{w}u - (\cL_w\cL_{Jv}\lambda)u + (\cL_{J\Nabla{w}v}\lambda)u.
\end{split}
\end{equation*}
Hence it follows from~\eqref{eq:nablaurho} and~\eqref{eq:LvJac} that
\begin{equation*}
\begin{split}
\trace\bigl((\nabla\Jhat_u)v\bigr)
&= 
d\lambda(v)\trace\bigl(\nabla(Ju)\bigr)
+ \cL_{Ju}\cL_v\lambda - \cL_{\Nabla{Ju}v}\lambda \\
&\quad
- d\lambda(Jv)\trace\bigl(\nabla u\bigr)  
- \cL_u\cL_{Jv}\lambda + \cL_{J\Nabla{u}v}\lambda \\
&=
d\lambda(v)f_{Ju} - d\lambda(Jv)f_u 
+ \cL_v\cL_{Ju}\lambda - \cL_{Jv}\cL_u\lambda \\
&\quad
- d\lambda((\Nabla{v}J)u+(\cL_uJ)v).
\end{split}
\end{equation*}
Since~${(\Lambda_\rho(J,\Jhat_u))(v)=\trace((\nabla\Jhat_u)v)+d\lambda((\Nabla{v}J)u)}$,
this proves~\eqref{eq:LAu}. Now suppose~$u$ satisfies~\eqref{eq:LAMBDAu}
and let~${v\in\Vect(M)}$.  Then, by Step~4, we have
\begin{equation*}
\begin{split}
\bigl(\Lambda_\rho(J,\cL_{\lambda u}J)\bigr)(v)
&= 
\bigl(\Lambda_\rho(J,\lambda\cL_uJ+\Jhat_u)\bigr)(v) \\
&= 
\lambda\bigl(\Lambda_\rho(J,\cL_uJ)\bigr)(v)
+ \bigl(\Lambda_\rho(J,\Jhat_u)\bigr)(v) 
+ d\lambda((\cL_uJ)v) \\
&= 
\lambda\bigl(2\Ric_{\rho,J}(u,v) - df_u(Jv) + df_{Ju}(v)\bigr)  \\
&\quad
+ d\lambda(v)f_{Ju} - d\lambda(Jv)f_u 
+ \cL_v\cL_{Ju}\lambda - \cL_{Jv}\cL_u\lambda \\
&=
2\Ric_{\rho,J}(\lambda u,v) - df_{\lambda u}(Jv) + df_{\lambda Ju}(v).
\end{split}
\end{equation*}
Here the third equality uses~\eqref{eq:LAu}.  This proves Step~5.

\medskip\noindent{\bf Step~6.}
{\it We prove~\eqref{eq:LAMBDAu}.} 

\medskip\noindent
There exist finitely many exact divergence-free vector fields~$u_i$ and smooth 
functions~$\lambda_i$ such that~${u=\sum_i\lambda_iu_i}$.  
For each~$i$ the vector field~$\lambda_iu_i$ satisfies~\eqref{eq:LAMBDAu} 
by Steps~3 and~5.  Hence so does~$u$ and this proves Theorem~\ref{thm:LAMBDA}.
\end{proof}

\bigbreak

Equation~\eqref{eq:LAMBDAu} is equivalent to the formula
\begin{equation}\label{eq:RICuv}
\Om_{\rho,J}(\cL_uJ,\cL_vJ) 
= \int_M\bigl(
2\Ric_{\rho,J}(u,v) + f_uf_{Jv}-f_{Ju}f_v
\bigr)\rho
\end{equation}
for~${u,v\in\Vect(M)}$.
For exact divergence-free vector fields~${u,v}$ this is 
the analogue of the identity~${\om(L_x\xi,L_x\eta)=\inner{\mu(x)}{[\xi,\eta]}}$
for Hamiltonian group actions on finite-dimensional symplectic manifolds.
The analogue in the scalar curvature setting is discussed 
in Remark~\ref{rmk:REDUCTIVE} below.


\subsection*{Scalar curvature}

Let~$(M,\om)$ be a $2\sn$-dimensional closed 
connected symplectic manifold and denote by 
$$
\sJ(M,\om) := \left\{J\in\Om^0(M,\End(TM))\,\Bigg|\,
\begin{array}{l}
J^2=-\one\mbox{ and }J^*\om=\om\\
\mbox{and }\om(\xhat,J\xhat)>0\\
\mbox{for all }\xhat\in T_xM\setminus\{0\}
\end{array}\right\}
$$
the space of all almost complex structures that are compatible 
with~$\om$.  This is an infinite-dimensional 
K\"ahler submanifold of~$\sJ(M)$ with the tangent 
spaces~${T_J\sJ(M,\om)=\{\Jhat\in\Om^{0,1}_J(M,TM)\,|\,
\om(\Jhat\cdot,\cdot)+\om(\cdot,\Jhat\cdot)=0\}}$,
the symplectic form~$\Om_\rho$ in~\eqref{eq:OMRHO},
and the complex structure~$\Jhat\mapsto-J\Jhat$.

\begin{definition}[{\bf Scalar Curvature}]\label{def:SCALAR}
Let~$\om$ be a symplectic form on~$M$, 
let~$J$ be an $\om$-compatible almost complex structure on~$M$,
let~$\nabla$ be the Levi-Civita connection 
of the metric~${\inner{\cdot}{\cdot}=\om(\cdot,J\cdot)}$, 
and let~${\tabla:=\nabla-\tfrac12J(\nabla J)}$.  
Define the {\bf Ricci form} of~$(\om,J)$ 
by~${\Ric_{\om,J} := \Ric_{\om^\sn/\sn!,J} 
= \tfrac12\trace(JR^\tabla)}$
and define the {\bf scalar curvature} by 
\begin{equation}\label{eq:SCALAR}
S_{\om,J} 
:= 2\inner{\Ric_{\om,J}}{\om}
:= \frac{2\Ric_{\om,J} \wedge\om^{\sn-1}/(\sn-1)!}{\om^\sn/\sn!}
\in\Om^0(M).
\end{equation}
\end{definition}

By Theorem~\ref{thm:RICCI} the scalar 
curvature~$S_{\om,J}$ in~\eqref{eq:SCALAR} satisfies 
\begin{equation}\label{eq:INTS}
\int_MS_{\om,J}\frac{\om^\sn}{\sn!}
=4\pi\INNER{c_1(TM,J)\smile\frac{[\om]^{\sn-1}}{(\sn-1)!}}{[M]}
\end{equation}
and~${\phi^*S_{\om,J} = S_{\phi^*\om,\phi^*J}}$
for every diffeomorphism~${\phi:M\to M}$.  
The following result was proved by Donaldson~\cite{DON2}, 
and independently by Fujiki~\cite{F} (in the integrable case) 
and Quillen (for Riemann surfaces).

\begin{corollary}[{\bf Fujiki--Quillen--Donaldson}]\label{thm:SCALAR}
The map~${J\mapsto S_{\om,J}}$ is an equivariant moment map 
for the action of~$\Ham(M,\om)$ on~$\sJ(M,\om)$, i.e.\ 
if~${H\in\Om^0(M)}$ and~${v_H\in\Vect(M)}$ is the Hamiltonian 
vector field defined by~${\iota(v_H)\om = dH}$, then
every smooth path~${\R\to\sJ(M,\om):t\mapsto J_t}$ satisfies 
\begin{equation}\label{eq:momentJHAM}
\frac{d}{dt}\int_MS_{\om,J_t}H\frac{\om^\sn}{\sn!}
= \tfrac12\int_M\trace\Bigl((\p_tJ_t)J_t(\cL_{v_H}J_t)\Bigr)\frac{\om^\sn}{\sn!}.
\end{equation}
\end{corollary}

\begin{proof}
Define~${J:=J_0}$,~${\Jhat:=\left.\tfrac{d}{dt}\right|_{t=0}J_t}$, and~${\rho:=\om^\sn/\sn!}$.  
Then
\begin{equation*}
\begin{split}
\left.\frac{d}{dt}\right|_{t=0}
\int_MS_{\om,J_t}H\frac{\om^\sn}{\sn!}
&=
\int_M 2H\Richat_\rho(J,\Jhat)\wedge\frac{\om^{\sn-1}}{(\sn-1)!} \\
&= 
\int_M Hd\Lambda_\rho(J,\Jhat)\wedge\frac{\om^{\sn-1}}{(\sn-1)!}  \\
&=
\int_M \Lambda_\rho(J,\Jhat)\wedge \iota(v_H)\rho.
\end{split}
\end{equation*}
Hence the assertion follows from Theorem~\ref{thm:RICCI}.
\end{proof}

\begin{remark}\label{rmk:REDUCTIVE}\rm
For a closed connected symplectic $2\sn$-manifold~$(M,\om)$ with volume 
form~${\rho:=\om^\sn/\sn!}$, an almost complex structure~${J\in\sJ(M,\om)}$, 
and two Hamiltonian functions~${F,G:M\to\R}$
equation~\eqref{eq:RICuv} takes the form
\begin{equation}\label{eq:SFG}
\begin{split}
\Om_{\rho,J}(\cL_{v_F}J,\cL_{v_G}J)
&= 
\int_M 2\Ric_{\om,J}(v_F,v_G)\rho \\
&=
\int_M 2\Ric_{\om,J}\wedge\iota(v_G)\iota(v_F)\rho \\
&=
\int_M 2\Ric_{\om,J}\wedge\iota(v_G)\left(dF\wedge\frac{\om^{\sn-1}}{(\sn-1)!}\right) \\
&=
\int_M 2\Ric_{\om,J}\wedge\{F,G\}\frac{\om^{\sn-1}}{(\sn-1)!} \\
&=
\int_MS_{\om,J}\{F,G\}\rho.
\end{split}
\end{equation}
Here~${\{F,G\}:=\om(v_F,v_G)}$ denotes the Poisson bracket.
If the scalar curvature is constant, equation~\eqref{eq:SFG} 
implies that~$\cL_{v_F}J$ and~$J\cL_{v_G}J$ are $L^2$ orthogonal 
and hence~${\Norm{\cL_{v_F}J+J\cL_{v_G}J}^2
=\Norm{\cL_{v_F}}^2+\Norm{\cL_{v_G}J}^2}$
for all~${F,G\in\Om^0(M)}$.  
If~$J$ is integrable, the scalar curvature is constant, 
and~${H^1(M;\R)=0}$, this in turn implies that
the Lie algebra of holomorphic vector fields
is the complexification of the Lie algebra of Killing vector fields
and is therefore reductive (Matsushima’s Theorem).
\end{remark}


\subsection*{Symplectic complements}

The next theorem examines symplectic complements 
in~${T_J\sJ(M)}$.  It shows that the regular part of 
the Marsden--Weinstein quotient
\begin{equation}\label{eq:MRHO}
\sW_0(M,\rho) := \{J\in\sJ(M)\,|\,\Ric_{\rho,J}=0\}/\Diff^\ex(M,\rho)
\end{equation}
is an infinite-dimensional symplectic manifold.  

\begin{theorem}[{\bf Complements}]\label{thm:COMP}
Let~${\rho\in\Om^{2\sn}(M)}$ be a positive volume form
and~${J\in\sJ(M)}$,~${\Jhat\in\Om^{0,1}_J(M,TM)}$,
${\lambdahat\in\Om^1(M)}$. Then the following holds.

\smallskip\noindent{\bf (i)}
There exists a~${\Jhat'\in\Om^{0,1}_J(M,TM)}$
such that~${\Lambda_\rho(J,\Jhat')=\lambdahat}$ 
if and only if~${\int_M\lambdahat\wedge\iota(v)\rho=0}$
for all~${v\in\Vect(M)}$ with~${\cL_vJ=0}$.

\smallskip\noindent{\bf (ii)}
There exists a~${v\in\Vect(M)}$ with~${\cL_vJ=\Jhat}$
if and only if~${\Om_{\rho,J}(\Jhat,\Jhat')=0}$ for 
all~${\Jhat'\in\Om^{0,1}_J(M,TM)}$ with~${\Lambda_\rho(J,\Jhat')=0}$.

\smallskip\noindent{\bf (iii)}
There exists a~${\Jhat'\in\Om^{0,1}_J(M,TM)}$
such that~${\Richat_\rho(J,\Jhat')=d\lambdahat}$ 
if and only if~${\int_Md\lambdahat\wedge\alpha=0}$
for all~${\alpha\in\Om^{2\sn-2}(M)}$ with~${\cL_{v_\alpha}J=0}$.

\smallskip\noindent{\bf (iv)}
There exists a~${v\in\Vect(M)}$ 
such that~${\cL_vJ=\Jhat}$ and~$\iota(v)\rho$ is exact
if and only if~${\Om_{\rho,J}(\Jhat,\Jhat')=0}$
for all~${\Jhat'\in\Om^{0,1}_J(M,TM)}$ 
such that~${\Richat_\rho(J,\Jhat')=0}$.
\end{theorem}

\begin{proof}
See page~\pageref{proof:COMP}.
\end{proof}

To prove Theorem~\ref{thm:COMP} it is convenient to choose 
a nondegenerate $2$-form~${\om\in\Om^2(M)}$ that is 
compatible with~$J$ and satisfies~${\om^\sn/\sn!=\rho}$.
Let~$\nabla$ be the Levi-Civita connection of the Riemannian 
metric~${\inner{\cdot}{\cdot}=\om(\cdot,J\cdot)}$ and define the linear 
operator~${\bar\p_J:\Om^0(M,TM)\to\Om^{0,1}_J(M,TM)}$ by
\begin{equation}\label{eq:DJ}
(\bar\p_Jv)u
:= -\tfrac{1}{2}J(\cL_vJ)u
 = \tfrac{1}{2}\bigl(\Nabla{u}v+J\Nabla{Ju}v - J(\Nabla{v}J)u\bigr) 
\end{equation}
for~${u,v\in\Vect(M)}$. 
Let~$\bar\p_J^*$ be the formal adjoint operator of~$\bar\p_J$ with
respect to the standard $L^2$-inner products.  Then both~$\bar\p_J$
and~$\bar\p_J^*$ are bounded linear operators with closed images
between appropriate Sobolev completions.  

\begin{lemma}\label{le:DJ}
Let~${\Jhat\in\Om^{0,1}_J(M,TM)}$. 
Then~${\Lambda_\rho(J,\Jhat) = \iota(J\bar\p_J^*\Jhat^*)\om}$.
\end{lemma}

\begin{proof}
Let~${v\in\Vect(M)}$. Then part~(ii) of Theorem~\ref{thm:RICCI} yields
\begin{equation}\label{eq:DJLAMBDA}
\begin{split}
&\int_M\Lambda_\rho(J,\Jhat)\wedge\iota(v)\rho
=
\tfrac{1}{2}\int_M\trace\bigl(\Jhat J\cL_vJ\bigr)\rho 
= 
- \inner{\Jhat^*}{\bar\p_Jv}_{L^2}  \\
&= 
- \inner{\bar\p_J^*\Jhat^*}{v}_{L^2} 
= 
\int_M\om(J\bar\p_J^*\Jhat^*,v)\rho
= 
\int_M\iota(J\bar\p_J^*\Jhat^*)\om\wedge\iota(v)\rho.
\end{split}
\end{equation}
This proves Lemma~\ref{le:DJ}.
\end{proof}

\begin{proof}[Proof of Theorem~\ref{thm:COMP}]\label{proof:COMP}
Choose~$\om$ as in Lemma~\ref{le:DJ}. We prove part~(i).
The condition is necessary by~\eqref{eq:LAMBDARHO}. 
Conversely, assume~${\int_M\lambdahat\wedge\iota(v)\rho=0}$
for all~${v\in\Vect(M)}$ with~${\cL_vJ=0}$.
Define the vector field~$u$
by~${\iota(Ju)\om := \lambdahat}$.
Then~${\inner{u}{v}_{L^2} = \int_M\om(u,Jv)\rho
= - \int_M\lambdahat\wedge\iota(v)\rho = 0}$
for all~${v\in\ker\bar\p_J}$.
Hence there exists a~${\Jhat'\in\Om^{0,1}_J(M,TM)}$
such that~${\bar\p_J^*(\Jhat')^*=u}$ and so
by Lemma~\ref{le:DJ} we have~${\lambdahat=\iota(Ju)\om
=\iota(J\bar\p_J^*(\Jhat')^*)\om=\Lambda_\rho(J,\Jhat')}$.
This proves~(i).

We prove part~(ii).  The condition is necessary by~\eqref{eq:LAMBDARHO}. 
Conversely, assume~${\Om_{\rho,J}(\Jhat,\Jhat')=0}$ for 
all~${\Jhat'\in\Om^{0,1}_J(M,TM)}$ that satisfy~${\Lambda_\rho(J,\Jhat')=0}$.
Let~${v\in\Vect(M)}$ with~${\bar\p_J^*\bigl(\bar\p_Jv+\tfrac{1}{2}J\Jhat\bigr)=0}$
and define~${\Jhat':=(\bar\p_Jv+\tfrac{1}{2}J\Jhat)^*}$.
Then~${\bar\p_J^*(\Jhat')^*=0}$, hence~${\Lambda_\rho(J,\Jhat')=0}$
by~\eqref{eq:DJLAMBDA}, and hence~${\Om_{\rho,J}(\Jhat,\Jhat')=0}$. 
This implies
$\int_M\abs{\Jhat'}^2\rho
= 
\int_M\trace\bigl(\Jhat'(\bar\p_Jv+\tfrac{1}{2}J\Jhat)\bigr)\rho  
=
\inner{(\Jhat')^*}{\bar\p_Jv}_{L^2} 
=
0$.
Thus~$\Jhat'=0$ and so
$
\Jhat=2J\bar\p_Jv=\cL_vJ
$
by~\eqref{eq:DJ}. This proves~(ii).

We prove part~(iii).  
The condition is necessary by~\eqref{eq:RICHATRHO}.  
Conversely, assume~${\int_Md\lambdahat\wedge\alpha=0}$
for all~${\alpha\in\Om^{2\sn-2}(M)}$ with~${\cL_{v_\alpha}J=0}$.
Choose a basis~${u_1,\dots,u_\ell}$ 
of~${V:=\left\{u\in\Vect(M)\,|\,\cL_uJ=0\right\}}$
such that~${u_{k+1},\dots,u_\ell}$ form a basis 
of~${\left\{u\in V\,|\,\iota(u)\rho\in\im d\right\}}$.
Then~$\iota(u_1)\rho,\dots,\iota(u_k)\rho$ 
are linearly independent in the quotient~$\Om^{2\sn-1}(M)/\im d$.
Hence, by Poincar\'e duality, there exist closed
$1$-forms~${\lambda_1,\dots,\lambda_k\in\Om^1(M)}$ 
such that~${\int_M\lambda_i\wedge\iota(u_j)\rho = \delta_{ij}}$
for~${i,j\le k}$.
Define~${\lambdahat':=
\lambdahat-\sum_{i=1}^k(\int_M\lambdahat\wedge\iota(u_i)\rho)
\lambda_i}$.
Then we have~${\int_M\lambdahat'\wedge\iota(u_j)\rho=0}$
for~${j=1,\dots,\ell}$. Hence by~(i) there 
exists a $1$-form~${\Jhat'\in\Om^{0,1}_J(M,TM)}$
such that~${\Lambda_\rho(J,\Jhat')=2\lambdahat'}$.
Thus~${\Richat_\rho(J,\Jhat')=d\lambdahat'=d\lambdahat}$
and this proves~(iii).

We prove part~(iv).  The condition is necessary by~\eqref{eq:RICHATRHO}.  
Conversely, assume~${\Om_{\rho,J}(\Jhat,\Jhat')=0}$
for all~${\Jhat'\in\Om^{0,1}_J(M,TM)}$ 
such that~${\Richat_\rho(J,\Jhat')=0}$.
Then by~(ii) there is a~${v\in\Vect(M)}$ with~${\cL_vJ=\Jhat}$.  
Choose~${u_i,\lambda_i}$ as in the proof of part~(iii)
and define~${v_0:=v-\sum_{i=1}^kx_iu_i}$,
${x_i:=\int_M\lambda_i\wedge\iota(v)\rho}$.
Then~${\cL_{v_0}J=\Jhat}$.
We prove that~${\iota(v_0)\rho}$ is exact.
To see this, let~$\lambdahat\in\Om^1(M)$ be any closed $1$-form
and define~$\lambdahat':=\lambdahat - \sum_{i=1}^ky_i\lambda_i$,
$y_i := \int_M\lambdahat\wedge\iota(u_i)\rho$.
Then~${\int_M\lambdahat'\wedge\iota(u_j)\rho=0}$
for~${j=1,\dots,\ell}$.  Hence by~(i) there exists
a $1$-form~$\Jhat'\in\Om^{0,1}_J(M,TM)$ 
such that~$\Lambda_\rho(J,\Jhat')=\lambdahat'$.
Thus~$\Richat_\rho(J,\Jhat')=0$, 
hence~${\Om_{\rho,J}(\Jhat,\Jhat')=0}$,
and therefore
\begin{equation*}
\begin{split}
\int_M\lambdahat\wedge\iota(v_0)\rho
= 
\int_M\lambdahat\wedge\iota(v)\rho 
- \sum_{i=1}^kx_iy_i 
= 
\int_M\lambdahat'\wedge\iota(v)\rho 
= 
\Om_{\rho,J}\bigl(\Jhat',\Jhat\bigr)
= 0.
\end{split}
\end{equation*}
This shows that~$\iota(v_0)\rho$ is exact
and completes the proof of Theorem~\ref{thm:COMP}.
\end{proof}


\section{The integrable case}\label{sec:INTEGRABLE}

Let~$M$ be a closed connected oriented $2\sn$-manifold.  
In this section we restrict attention to (integrable) complex 
structures that are compatible with the orientation.  
Denote the space of such complex structures by~${\sJ_\INT(M)}$.


\subsection*{The Ricci form in the integrable case}

Let~${J\in\sJ_\INT(M)}$.  Then~$(TM,J)$ is a holomorphic vector 
bundle with the Cauchy--Riemann 
operator~${\bar\p_J:\Om^{0,q}_J(M,TM)\to\Om^{0,q+1}_J(M,TM)}$.
It satisfies
\begin{equation}\label{eq:LvJ}
2J\bar\p_{J,u}v = J\Nabla{u}v - \Nabla{Ju}v = (\cL_vJ)u
\end{equation}
and
\begin{equation}\label{eq:dNJ}
\begin{split}
2J\bar\p_J\Jhat(u,v)
&=
J(\Nabla{u}\Jhat)v 
- J(\Nabla{v}\Jhat)u 
- J(\Nabla{Ju}\Jhat)Jv
+ J(\Nabla{Jv}\Jhat)Ju \\
&=
-\left.\tfrac{d}{dt}\right|_{t=0}N_{J_t}(u,v)
\end{split}
\end{equation}
for all~${u,v\in\Vect(M)}$, all~${\Jhat\in\Om^{0,1}_J(M,TM)}$, and 
every smooth path of almost complex structures~${\R\to\sJ(M):t\mapsto J_t}$ 
with~${J_0=J}$ and~${\frac{d}{dt}|_{t=0}J_t=\Jhat}$.
Here~$\nabla$ is a torsion-free connection on~$TM$ with~${\nabla J=0}$,
equation~\eqref{eq:LvJ} follows from~\eqref{eq:LvJac},
and~\eqref{eq:dNJ} follows by differentiating~\eqref{eq:NJnabla}.

Next observe that
\begin{equation}\label{eq:dfJN}
d(df\circ J)(u,v) - d(df\circ J)(Ju,Jv) = df(JN_J(u,v)).
\end{equation}
for all~${f\in\Om^0(M)}$ and all~${u,v\in\Vect(M)}$. 
Hence an almost complex structure~$J$ is integrable if and only if the
$2$-form~${d(df\circ J)}$ is of type $(1,1)$ for all~${f\in\Om^0(M)}$. 
Theorem~\ref{thm:RICBC} below uses the Bott--Chern cohomology 
group~${H^{1,1}_\BC(M,J):=(\ker d\cap\Om^{1,1}_J(M))/\{d(df\circ J)\,|\,f\in\Om^0(M)\}}$
\cite{AT1,AT2,BGS,BC}. 
It shows that~$\Ric_{\rho,J}$ is the standard Ricci form in the integrable case. 

\begin{theorem}\label{thm:RICBC}
Let~${\rho\in\Om^{2\sn}(M)}$ be a positive volume form,
let~${J\in\sJ_\INT(M)}$, and let~$\nabla$ be a torsion-free 
$\rho$-connection with~${\nabla J=0}$.  
The following holds.

\smallskip\noindent{\bf (i)}
${\Ric_{\rho,J} = \tfrac12\trace(JR^\nabla)}$ is a closed~$(1,1)$-form
and~$\tfrac{1}{2\pi}\Ric_{\rho,J}$ represents the 
first Bott--Chern class of the holomorphic tangent bundle~$(TM,J)$.

\smallskip\noindent{\bf (ii)}
There exists a diffeomorphism~${\phi\in\Diff_0(M)}$ 
such that~${\Ric_{\rho,\phi^*J}=0}$ if and only if the first 
Bott--Chern class of~$(TM,J)$ vanishes.

\smallskip\noindent{\bf (iii)}
Let~${\phi:M\to M}$ be an  orientation preserving diffeomorphism
and suppose that~${\Ric_{\rho,J}=\Ric_{\rho,\phi^*J}=0}$.
Then~${\phi^*\rho=\rho}$.   If in addition~$\phi$ is isotopic
to the identity, then~${\phi\in\Diff_0(M,\rho)}$.
\end{theorem}

\begin{proof}
The formula~${\Ric_{\rho,J} = \tfrac12\trace(JR^\nabla)}$
follows from Definition~\ref{def:RICCI}.
Moreover,~${\Ric_{\rho,J}}$ is independent of 
the choice of~$\nabla$ by part~(i) of Theorem~\ref{thm:RICCI},
is closed and represents the cohomology
class~${2\pi c_1(TM,J)\in H^2(M;\R)}$ 
by part~(iii) of Theorem~\ref{thm:RICCI},
and is a $(1,1)$-form by Lemma~\ref{le:RICCI11}.

Now choose a
nondegenerate $2$-form~${\om\in\Om^2(M)}$,
compatible with~$J$, such that~$\rho$ is the volume form 
of the metric~${\inner{\cdot}{\cdot}=\om(\cdot,J\cdot)}$. 
Let~$\nabla$ be the Levi--Civita connection of this metric and define
\begin{equation}\label{eq:habla1}
\tabla:=\nabla-\tfrac12J\nabla J,\qquad
\dhabla := \tabla - \tfrac14(A-A^*),
\end{equation}
where~$A\in\Om^1(M,\End(TM))$ is the endomorphism 
valued $1$-form defined by
\begin{equation}\label{eq:habla2}
A(u)v := J(\Nabla{v}J)u + (\Nabla{Jv}J)u
\end{equation}
for~${u,v\in\Vect(M)}$.  Then, for all~${u\in\Vect(M)}$,
\begin{equation}\label{eq:habla3}
A(u)J = JA(u) = -A(Ju),\qquad
A(u)^*J=JA(u)^* = A(Ju)^*.
\end{equation}
This shows that~$\dhabla$ is a Hermitian connection on~$TM$
and induces the same Cauchy--Riemann operator on~$TM$ 
as the connection~${\tabla-\tfrac14A}$. The latter preserves~$J$
by~\eqref{eq:habla3} and is torsion-free by~\eqref{eq:NJnabla}
(but it need not preserve~$\rho$). Hence, for all~${u,v\in\Vect(M)}$, 
we have 
$$
\bar\p^\sdhabla_{J,u}v
= \bar\p^{\tabla-\frac14A}_{J,u}v
= \bar\p_{J,u}v 
= \tfrac12\Bigl(\Nabla{u}v+J\Nabla{Ju}v - J(\Nabla{v}J)u\Bigr).
$$
Here the last equality holds because $\nabla$ is torsion-free and~$J$ is integrable. 
Thus~$\dhabla$ is the unique Hermitian connection on~$TM$ 
with~${\bar\p^\sdhabla_J=\bar\p_J}$. 

The curvature tensor of~$\dhabla$ is given by 
\begin{equation*}
R^\sdhabla = R^\tabla + \tfrac{1}{4}d^\tabla(A^*-A) + \tfrac{1}{32}[(A^*-A)\wedge(A^*-A)].
\end{equation*}
Since $J$ commutes with~${A^*-A}$ by~\eqref{eq:habla3}, we obtain
\begin{equation*}
\begin{split}
\trace(JR^\sdhabla)
&=
\trace(JR^\tabla) + \tfrac14\trace\bigl(Jd^\tabla(A^*-A)\bigr) \\
&=
\trace(JR^\tabla) + \tfrac14\trace\bigl(d^\tabla(JA^*-JA)\bigr) \\
&=
\trace(JR^\tabla) + \tfrac12d\bigl(\trace(A)\circ J\bigr) \\
&=
\trace(JR^\tabla) + d\lambda^\nabla_J 
=
2\Ric_{\rho,J}.
\end{split}
\end{equation*}
Here the third equality follows from~\eqref{eq:habla3}
and the fact that the endomorphisms~$A(Ju)$ 
and~$A(Ju)^*$ have the same trace, 
the fourth equality uses the fact that the two summands 
in~${v\mapsto A(Ju)v=(\Nabla{v}J)u+(\Nabla{Jv}J)Ju}$
have the same trace, both equal to~$\lambda_J^\nabla(u)$
(see equation~\eqref{eq:RICCI2}), and the last equality 
follows from part~(iii) of Theorem~\ref{thm:RICCI}. 
This proves~(i).

\bigbreak

We prove part~(ii). Let~${\phi\in\Diff_0(M)}$
such that~${\Ric_{\rho,\phi^*J}=0}$.  Then we 
have~${\Ric_{\phi_*\rho,J} = \phi_*\Ric_{\rho,\phi^*J} = 0}$
by part~(i) of Theorem~\ref{thm:RICCI}.
Define the function~${f\in\Om^0(M)}$ by~${e^{-f}\rho := \phi_*\rho}$.
Then
\begin{equation*}
\begin{split}
\Ric_{\rho,J}
=
\Ric_{\rho,J}-\Ric_{\phi_*\rho,J} 
=
\Ric_{\rho,J}-\Ric_{e^{-f}\rho,J} 
=
\tfrac{1}{2}d(df\circ J).
\end{split}
\end{equation*}
Here the last equality uses~\eqref{eq:RICRHOF}.
Since~$\Ric_{\rho,J}$ represents~$2\pi$ 
times the first Bott--Chern class of~$(TM,J)$
by~(i), this shows that~${c_{1,\BC}(TM,J)=0}$.

Conversely, assume~${c_{1,\BC}(TM,J)=0}$. 
Then, by part~(i), there exists a smooth function~${f:M\to\R}$ 
such that~${\Ric_{\rho,J} = \tfrac{1}{2}d(df\circ J)}$.
Choose~${c\in\R}$ such that~${e^c\int_M\rho=\int_Me^{-f}\rho}$ 
and replace~$f$ by~${f+c}$ to obtain~${\int_Me^{-f}\rho = \int_M\rho}$.
Then by Moser isotopy there exists a smooth 
isotopy~$\{\phi_t\}_{0\le t\le1}$ of~$M$ such that~${\phi_0=\id}$ 
and~${\phi_t^*\bigl((1-t)\rho+te^{-f}\rho\bigr) = \rho}$
for~${0\le t\le 1}$. Thus the diffeomorphism~${\phi:=\phi_1}$
is isotopic to the identity and satisfies~${\phi^*(e^{-f}\rho)=\rho}$.
Hence 
\begin{equation*}
\begin{split}
\Ric_{\rho,\phi^*J}
=
\Ric_{\phi^*(e^{-f}\rho),\phi^*J} 
= 
\phi^*\Ric_{e^{-f}\rho,J} 
= 
\phi^*\Bigl(\Ric_{\rho,J}-\tfrac{1}{2}d(df\circ J)\Bigr) 
=
0.
\end{split}
\end{equation*}
This proves~(ii).

We prove part~(iii). Let~${\phi\in\Diff(M)}$ be orientation preserving,
assume that~${\Ric_{\rho,\phi^*J} = \Ric_{\rho,J} = 0}$, and 
define~${f\in\Om^0(M)}$ by~${e^{-f}\rho:=\phi_*\rho}$. Then 
\begin{equation*}
\begin{split}
\tfrac{1}{2}d(df\circ J)
= 
\Ric_{\rho,J} - \Ric_{e^{-f}\rho,J} 
=
- \Ric_{\phi_*\rho,J} 
=
- \phi_*\Ric_{\rho,\phi^*J} 
= 
0.
\end{split}
\end{equation*}
Thus~$f$ is constant.
Since~${\int_Me^{-f}\rho=\int_M\phi_*\rho=\int_M\rho}$, 
it follows that~${f=0}$ and so~${\phi_*\rho = \rho}$.  
Moreover,~${\Diff_0(M,\rho)=\Diff(M,\rho)\cap\Diff_0(M)}$
by Moser isotopy. This proves part~(iii) and Theorem~\ref{thm:RICBC}.
\end{proof}

\begin{example}\label{ex:RIEMANN}\rm
Assume~${\sn=1}$, suppose~$M$ has genus~${g\ge1}$,
define~${V:=\int_M\rho}$ and~${c:=2\pi(2-2g)V^{-1}\le0}$, 
and let~${K_{\rho,J}:=\Ric_{\rho,J}/\rho}$ be the 
Gau{\ss}ian curvature. Then the moment map
$$
\sJ(M)\to\Om^2(M):J\mapsto2(\Ric_{\rho,J}-c\rho)
=2(K_{\rho,J}-c)\rho
$$
is $\sG$-equivariant and takes values in the space of 
exact $2$-forms. The uniformization theorem for Riemann 
surfaces asserts that for every~${J\in\sJ(M)}$ 
there exists a diffeomorphism~${\phi\in\Diff_0(M)}$ 
such that~${K_{\phi_*\rho,J}=c}$ and therefore~${\Ric_{\rho,\phi^*J}=c\rho}$.
Moreover, if~${\Ric_{\rho,J}=\Ric_{\rho,\phi^*J}=c\rho}$
for some orientation preserving diffeomorphism~$\phi$
and~$\phi_*\rho=:e^f\rho$, 
then~${\tfrac{1}{2}d(df\circ J)=c(e^f-1)\rho}$.
Hence~${d^*df=2c(e^f-1)}$ and this 
implies~${\int_M\abs{df}^2\rho=2c\int_Mf(e^f-1)\rho\le0}$. 
Thus~$f$ is constant 
and~${\int_Me^f\rho=\int_M\phi_*\rho=\int_M\rho}$,
so~${f\equiv0}$ and~${\phi^*\rho=\rho}$.
\end{example}

Let $(M,\om,J)$ be a closed connected K\"ahler manifold.
For a K\"ahler potential~$h:M\to\R$ (with mean value zero) 
let~${\om_h:=\om+\bi\bar\p\p h=\om+\tfrac{1}{2}d(dh\circ J)}$ 
be the associated symplectic form and let~${\rho_h:=\om_h^\sn/\sn!}$. 
The Calabi conjecture asserts that the map~${h\mapsto\Ric_{\rho_h,J}}$ 
is a bijection onto the space 
of closed $(1,1)$-forms representing the cohomology class~$2\pi c_1(TM,J)$. 
Injectivity was proved by Calabi~\cite{CALABI1,CALABI2}
and surjectivity by Yau~\cite{YAU0,YAU1}.

\begin{corollary}[{\bf Calabi--Yau}]\label{cor:CY}
Let $(M,\om,J)$ be a closed connected K\"ahler manifold
and let~$\rho\in\Om^{2\sn}(M)$ be a positive volume form
with~${\int_M\rho=\int_M\om^\sn/\sn!}$.
Then the following holds.

\smallskip\noindent{\bf (i)}
There exists a unique K\"ahler potential~${h:M\to\R}$ 
such that~${\rho_h=\rho}$.

\smallskip\noindent{\bf (ii)}
Assume~$\om^\sn/\sn!=\rho$ and~${c_1(TM,J)=0\in H^2(M;\R)}$.
Then there exists a diffeomorphism~${\phi\in\Diff_0(M)}$ such that
\begin{equation}\label{eq:CY}
\Ric_{\rho,\phi^*J}=0
\mbox{ and }\phi^*J\mbox{ is compatible with }\om.
\end{equation}

\smallskip\noindent{\bf (iii)}
Assume~${\om^\sn/\sn!=\rho}$ and~${\Ric_{\rho,J}=0}$.
Suppose~${\phi\in\Diff(M)}$ satisfies~\eqref{eq:CY} 
and the $2$-form~${\phi^*\om-\om}$ is exact.  
Then~${\phi^*\om=\om}$.
\end{corollary}

\begin{proof}
We prove part~(i).  By part~(i) of Theorem~\ref{thm:RICBC}, $\Ric_{\rho,J}$ 
is a closed $(1,1)$-form representing the cohomology class~$2\pi c_1(TM,J)$.
Hence, by Yau's existence theorem~\cite{YAU0,YAU1}
and Calabi's uniqueness theorem~\cite{CALABI1,CALABI2}, 
there exists a unique K\"ahler potential~$h$ such that~${\Ric_{\rho_h,J}=\Ric_{\rho,J}}$.  
Since~${\int_M\rho_h=\int_M\rho}$ by assumption, this implies~${\rho_h=\rho}$ 
by equation~\eqref{eq:RICRHOF} in part~(i) of Theorem~\ref{thm:RICCI}.

We prove part~(ii).  By assumption and part~(i) 
of Theorem~\ref{thm:RICBC}~$\Ric_{\rho,J}$ is an exact~$(1,1)$-form.
Since~$J$ admits a compatible K\"ahler form, this implies that
there exists a function~${f\in\Om^0(M)}$ such 
that
$$
\Ric_{\rho,J}=\tfrac{1}{2}d(df\circ J),\qquad
\int_Me^{-f}\rho=\int_M\rho.
$$
Hence~${\Ric_{e^{-f}\rho,J}=0}$ by part~(i) of Theorem~\ref{thm:RICCI}.
Now it follows from~(i) that there exists a K\"ahler potential~$h$ 
such that~${\rho_h=e^{-f}\rho}$.  Since~$\om_h$ and~$\om$ are
compatible with~$J$, Moser isotopy yields a 
diffeomorphism~${\phi\in\Diff_0(M)}$ with~${\phi^*\om_h=\om}$. 
Thus~$\phi^*J$ is compatible with~$\om$ and~$\phi^*\rho_h=\rho$.
This implies~${\Ric_{\rho,\phi^*J}=\phi^*\Ric_{\rho_h,J}=0}$
by part~(i) of Theorem~\ref{thm:RICCI}.

To prove~(iii), note that~$(\phi^{-1})^*\om$ is compatible with~$J$ 
and represents the cohomology class of~$\om$.  
Thus there is a K\"ahler potential~$h$ with~${\om_h=(\phi^{-1})^*\om}$.  
Hence~${\phi^*\rho_h=\rho}$ and~${\phi^*\Ric_{\rho_h,J}=\Ric_{\rho,\phi^*J}=0}$
by part~(i) of Theorem~\ref{thm:RICCI}.
Thus~${h=0}$ by Ca\-labi uniqueness, so~${\phi^*\om=\om}$. 
This proves Corollary~\ref{cor:CY}.
\end{proof}


\subsection*{Ricci-flat K\"ahler manifolds}

Let~${\rho\in\Om^{2\sn}(M)}$ be a positive volume form.
Then the symplectic form~$\Om_\rho$ on~$\sJ(M)$ 
is a~$(1,1)$-form for the complex structure~${\Jhat\mapsto-J\Jhat}$.
However, the resulting symmetric bilinear 
form~${\inner{\Jhat_1}{\Jhat_2}_{\rho,J} 
= \tfrac{1}{2}\int_M\trace(\Jhat_1\Jhat_2)\rho}$
is indefinite, so~$\sJ(M)$ is not K\"ahler 
and complex submanifolds need not be symplectic.
An example is the space of (integrable) complex structures 
with real first Chern class zero and nonempty K\"ahler cone.  
It is denoted by
\begin{equation*}
\begin{split}
\sJ_{\INT,0}(M)
:=
\biggl\{J\in\sJ_\INT(M)\,\bigg|\,
\begin{array}{l}c_1(TM,J)=0\in H^2(M;\R) \\
\mbox{and }J\mbox{ admits a K\"ahler form}
\end{array}\biggr\}.
\end{split}
\end{equation*}
Its tangent space at~$J$ is the kernel 
of~${\bar\p_J:\Om^{0,1}_J(M,TM)\to\Om^{0,2}_J(M,TM)}$.

\begin{theorem}\label{thm:JINTZERO}
Let~${J\in\sJ_{\INT,0}(M)}$ with~${\Ric_{\rho,J}=0}$
and let~${\Jhat\in\Om^{0,1}_J(M,TM)}$ such that~${\bar\p_J\Jhat=0}$. 
Then the following holds.

\smallskip\noindent{\bf (i)}
${\Om_{\rho,J}(\Jhat,\Jhat')=0}$ for all~$\Jhat'$ with~${\bar\p_J\Jhat'=0}$
if and only if there exists a vector field~$v$ such that~${f_v=f_{Jv}=0}$
and~${\cL_vJ=\Jhat}$. 

\smallskip\noindent{\bf (ii)}
Assume~${\Richat_\rho(J,\Jhat)=0}$.
Then~${\Om_{\rho,J}(\Jhat,\Jhat')=0}$ for all~$\Jhat'$
with~${\bar\p_J\Jhat'=0}$ and~${\Richat_\rho(J,\Jhat')=0}$
if and only if there exists a vector field~$v$ such that~${f_v=0}$
and~${\cL_vJ=\Jhat}$ or, equivalently, there exists 
an~${\alpha\in\Om^{2\sn-2}(M)}$ with~${\cL_{v_\alpha}J=\Jhat}$. 
\end{theorem}

\begin{proof}
See page~\pageref{proof:JINTZERO}.
\end{proof}

Define~${\sJ_{\INT,0}(M,\rho):=\{J\in\sJ_{\INT,0}(M)\,|\,\Ric_{\rho,J}=0\}}$.
Part~(ii) of Theorem~\ref{thm:JINTZERO}
(compare with part~(iv) of Theorem~\ref{thm:COMP})
implies that the Teich\-m\"uller space
$
\sT_0(M,\rho):= \sJ_{\INT,0}(M,\rho)/\Diff_0(M,\rho)
$
is a symplectic submanifold of the 
infinite-dimensional symplectic quotient~$\sW_0(M,\rho)$
in~\eqref{eq:MRHO}. The Teichm\"uller space will be
discussed further in Section~\ref{sec:TEICH}.

The proof of Theorem~\ref{thm:JINTZERO}
relies on three lemmas about Ricci-flat K\"ahler manifolds, 
which examine~$\Lambda_\rho$
(Lemma~\ref{le:LAMBDAFG}), show that holomorphic
vector fields correspond to harmonic $1$-forms
(Lemma~\ref{le:HOLV}), and show that the space
of harmonic infinitesimal complex structures
is invariant under~${\Jhat\mapsto\Jhat^*}$
(Lemma~\ref{le:JHATSTAR}).
These in turn require three preparatory lemmas about Hamiltonian
and gradient vector fields (Lemma~\ref{le:DIVH}),
about infinitesimal compati\-bility (Lemma~\ref{le:ONEONE}),
and about vector fields~$v$ such that~$\cL_vJ$ is self-adjoint
(Lemma~\ref{le:SELFADJOINT}).  While some of this material is well-known,
we include full proofs for completeness of the exposition.
For a symplectic manifold~$(M,\om)$ and~${J\in\sJ(M,\om)}$
the Hamiltonian and gradient vector fields of~${H\in\Om^0(M)}$ 
are given by~${\iota(v_H)\om=dH}$ and~${\nabla H=Jv_H}$.

\begin{lemma}
[{\bf Hamiltonian and Gradient Vector Fields}]
\label{le:DIVH}
Let~$(M,\om)$ be a symplectic $2\sn$-manifold,
let~${J\in\sJ(M,\om)}$, let~${H\in\Om^0(M)}$,
and define~${\rho:=\om^\sn/\sn!}$. 
Then~${f_{v_H}=0}$ and~${f_{\nabla H}=-d^*dH}$. 
Moreover, if~${\Ric_{\rho,J}=0}$, 
then 
$
\Lambda_\rho(J,\cL_{\nabla H}J) = dd^*dH\circ J
$
and
$
\Lambda_\rho(J,\cL_{v_H}J) = -dd^*dH.
$
\end{lemma}

\begin{proof}
Since~${{*\lambda}=-(\lambda\circ J)\wedge\om^{\sn-1}/(\sn-1)!}$
for all~${\lambda\in\Om^1(M)}$, we have
\begin{equation}\label{eq:STARV}
\begin{split}
*\iota(v)\om
= -(\iota(v)\om\circ J)\wedge\tfrac{\om^{\sn-1}}{(\sn-1)!} 
= \iota(Jv)\om\wedge\tfrac{\om^{\sn-1}}{(\sn-1)!}
= \iota(Jv)\rho
\end{split}
\end{equation}
for all~${v\in\Vect(M)}$.
Hence~${f_{\nabla H} = *d\iota(Jv_H)\rho = *{d{*dH}} = - d^*dH}$
and the remaining assertions follow from~\eqref{eq:LAMBDAu}. 
This proves Lemma~\ref{le:DIVH}.
\end{proof}

\begin{lemma}[{\bf Infinitesimal Compatibility}]\label{le:ONEONE}
Let $M$ be an oriented $2\sn$-manifold
and let~${J\in\sJ(M)}$.  If~${\tau\in\Om^{1,1}_J(M)}$
and~${v\in\Vect(M)}$, then the Lie derivatives~${\tauhat:=\cL_v\tau}$ 
and~${\Jhat:=\cL_vJ}$ satisfy the equation
\begin{equation}\label{eq:tauhat}
\tauhat(u,u')-\tauhat(Ju,Ju') = \tau(Ju,\Jhat u') + \tau(\Jhat u,Ju')
\end{equation}
for all~${u,u'\in\Vect(M)}$.  If~$J$ is integrable and ${\Jhat\in\Om^{0,1}_J(M,TM)}$ 
satisfies the equation~${\bar\p_J\Jhat=0}$, then~${\tau:=\Ric_{\rho,J}}$ 
and~${\tauhat:=\Richat_\rho(J,\Jhat)}$ satisfy equation~\eqref{eq:tauhat}
for every positive volume form~${\rho\in\Om^{2\sn}(M)}$.
\end{lemma}

\begin{proof}
Let~${\tau\in\Om^{1,1}_J(M)}$, 
let~$\phi_t$ be the flow of~${v\in\Vect(M)}$,
and let~${J_t:=\phi_t^*J}$.
If~${\tauhat:=\cL_v\tau}$ and~${\Jhat:=\cL_vJ}$, differentiate the 
identity~${\tau_t(u,u')=\tau_t(J_tu,J_tu')}$ 
with~${\tau_t:=\phi_t^*\tau}$ to obtain~\eqref{eq:tauhat}.  
Now assume~$J$ is integrable and~${\tau=\Ric_{\rho,J}}$.
If~${\tauhat:=\Richat_\rho(J,\cL_vJ)}$,
then~${\cL_v\tau-\tauhat=\tfrac{1}{2}d(df_v\circ J)\in\Om^{1,1}_J(M)}$   
by Theorem~\ref{thm:LAMBDA}, 
so~\eqref{eq:tauhat} holds with~${\Jhat=\cL_vJ}$.  
Now use the holomorphic Poincar\'e Lemma.
\end{proof}

\begin{lemma}[{\bf Self-Adjoint Lie Derivative~${\cL_vJ}$}]\label{le:SELFADJOINT}
Let~$(M,\om)$ be a closed connected symplectic $2\sn$-manifold,
let~${J\in\sJ(M,\om)}$, and let~${v\in\Vect(M)}$.  
Then the following holds (with~${\rho=\om^\sn/\sn!}$ in part~(iii)).

\smallskip\noindent{\bf (i)}
${\cL_vJ}$ is self-adjoint  if and only if ${d\iota(v)\om\in\Om^{1,1}_J(M)}$.

\smallskip\noindent{\bf (ii)}
If~$J$ is integrable, then~${\cL_vJ}$ is self-adjoint 
if and only if there exists a function~${F\in\Om^0(M)}$ 
such that~${d\iota(v+\nabla F)\om=0}$.

\smallskip\noindent{\bf (iii)}
${\iota(v)\om}$ is harmonic if and only 
if~${f_v=f_{Jv}=0}$ and~${\cL_vJ=(\cL_vJ)^*}$.
\end{lemma}

\begin{proof}
Part~(i) follows from Lemma~\ref{le:ONEONE} with~$\tau=\om$. 

Now suppose~$J$ is integrable.  Then~${\cL_vJ}$
and~${\cL_{Jv}J=J\cL_vJ}$ are self-adjoint 
for every symplectic vector field~$v$ by~(i).
Conversely, assume~${\cL_vJ=(\cL_vJ)^*}$.
Then~${d\iota(v)\om\in\Om^{1,1}_J(M)}$ by~(i),
and so there exists a function~${F\in\Om^0(M)}$ 
such that~${d\iota(v)\om=d(dF\circ J)=-d\iota(\nabla F)\om}$. 
This proves~(ii).  

\bigbreak

To prove~(iii) define~${\rho:=\om^\sn/\sn!}$. 
Then~\eqref{eq:STARV} shows 
that~${d^*\iota(v)\om=0}$ if and only if~${f_{Jv}=0}$.  
If~${d\iota(v)\om=0}$, then~${f_v=0}$ and~${\cL_vJ}$ 
is self-adjoint by~(i).  
Conversely, assume~${f_v=0}$ and ${\cL_vJ=(\cL_vJ)^*}$.
Then~${d\iota(v)\om\in\Om^{1,1}_J(M)}$ by~(i)
and~${\inner{d\iota(v)\om}{\om}=f_v=0}$.
Thus
$$
*d\iota(v)\om=-d\iota(v)\om\wedge\tfrac{\om^{\sn-2}}{(\sn-2)!},
$$
so~${d{*d\iota(v)\om}=0}$, and hence~${d\iota(v)\om=0}$.  
This proves Lemma~\ref{le:SELFADJOINT}.
\end{proof}

\begin{lemma}[{\bf $\Lambda_\rho$ in the Ricci-flat Case}]\label{le:LAMBDAFG}
Let~$(M,J,\om)$ be a closed connected ${2\sn}$-dimensional Ricci-flat
K\"ahler mani\-fold with volume form~${\rho=\om^\sn/\sn!}$ 
and let~${\Jhat\in\Om^{0,1}_J(M,TM)}$ such that~${\bar\p_J\Jhat=0}$.
Then~${\Richat_\rho(J,\Jhat)\in\Om^{1,1}_J(M)}$ 
and there exists a unique pair of  
functions~${f=f_\Jhat,g=f_{J\Jhat}\in\Om^0(M)}$ 
such that
\begin{equation}\label{eq:LAMBDAFG}
\Lambda_\rho(J,\Jhat)= -df\circ J+dg,\qquad
\int_Mf\rho=\int_Mg\rho=0.
\end{equation}
Moreover, if~${\bar\p_J^*\Jhat=0}$ then~${\Lambda_\rho(J,\Jhat)=0}$.
\end{lemma}

\begin{proof}
It follows directly from Lemma~\ref{le:ONEONE} 
that~${\Richat_\rho(J,\Jhat)\in\Om^{1,1}_J(M)}$. 
Now assume~${\bar\p_J^*\Jhat=0}$ 
and let~${v:=\bar\p_J^*\Jhat^*\in\Vect(M)}$.
Then~$\iota(v)\om=-\Lambda_\rho(J,J\Jhat)$ by Lemma~\ref{le:DJ},
hence~${d\iota(v)\om=-2\Richat_\rho(J,J\Jhat)}$ is an exact~$(1,1)$-form,
thus~$\cL_vJ$ is self-adjoint by Lemma~\ref{le:SELFADJOINT},
and so is ${J\cL_vJ = -2\bar\p_Jv = 2\bar\p_J\bar\p_J^*(\Jhat-\Jhat^*)}$.
Thus~${\bar\p_J\bar\p_J^*(\Jhat^*-\Jhat)}$ is $L^2$ orthogonal 
to~${\Jhat^*-\Jhat}$ and so~${\bar\p_J^*\Jhat^*=\bar\p_J^*(\Jhat^*-\Jhat)=0}$.
Hence~${\Lambda_\rho(J,\Jhat)=0}$ by Lemma~\ref{le:DJ}.
To prove~\eqref{eq:LAMBDAFG},  choose~${v\in\Vect(M)}$ 
such that~${\bar\p_J^*(\Jhat-\cL_vJ)=0}$.  
Then~${\Lambda_\rho(J,\Jhat)=\Lambda_\rho(J,\cL_vJ)}$
and hence~${f:=f_v}$ and~${g:=f_{Jv}}$ satisfy~\eqref{eq:LAMBDAFG}
by Theorem~\ref{thm:LAMBDA}.  This proves Lemma~\ref{le:LAMBDAFG}.
\end{proof}

\begin{lemma}[{\bf Holomorphic Vector Fields}]\label{le:HOLV}
Let~$M$ be a closed connected oriented $2\sn$-manifold,
fix a positive volume form~${\rho\in\Om^{2\sn}(M)}$
and an almost complex structure~${J\in\sJ(M)}$ 
such that~${\Ric_{\rho,J}=0}$, and let~${v\in\Vect(M)}$.  
Then the following holds.

\smallskip\noindent{\bf (i)}
${\Lambda_\rho(J,\cL_vJ)=0}$ if and only if~${d\iota(v)\rho=d\iota(Jv)\rho=0}$.

\smallskip\noindent{\bf (ii)}
Assume~$J$ is compatible with a symplectic form~$\om$ 
such that~${\om^{\sn}/\sn!=\rho}$ and that~${\cL_vJ=0}$.
Then~$\iota(v)\om$ is a harmonic $1$-form.

\smallskip\noindent{\bf (iii)}
Assume~$J$ is integrable and compatible with a symplectic 
form~$\om$ such that~${\om^{\sn}/\sn!=\rho}$.
Then~${\cL_vJ=0}$ if and only if~$\iota(v)\om$ is harmonic.
If~${d\iota(v)\rho=0}$, then there exists a~${v_0\in\Vect(M)}$ 
such that~$\iota(v_0)\rho$ is exact and~${\cL_{v_0}J=\cL_vJ}$.
\end{lemma}

\begin{proof}
To prove~(i), observe that~${\Lambda_\rho(J,\cL_vJ)=-df_v\circ J+df_{Jv}}$ 
by~\eqref{eq:LAMBDAu}.  Assume~${\Lambda_\rho(J,\cL_vJ)=0}$ 
and choose a nondegenerate $2$-form~${\om\in\Om^2(M)}$ 
that is compatible with~$J$ and satisfies~${\om^\sn/\sn!=\rho}$.
Then equation~\eqref{eq:STARV} yields
$$
0 = *\left(\bigl(d(df_v\circ J)\bigr) \wedge \tfrac{\om^{\sn-1}}{(\sn-1)!}\right)
= d^*df_v - *\left((df_v\circ J)\wedge d\om\wedge\tfrac{\om^{\sn-2}}{(\sn-2)!}\right).
$$
By the Hopf maximum principle, this implies that~$f_v$ is locally constant.
Thus~$f_{Jv}$ is also locally constant.  Since~$f_v$ and~$f_{Jv}$ have 
mean value zero on each connected component of~$M$, 
it follows that~${f_v=f_{Jv}=0}$.  This proves~(i). 

Part~(ii) follows directly from~(i) and Lemma~\ref{le:SELFADJOINT}.
To prove~(iii), assume~$J$ is integrable and~$\om$ is closed. 
If~$\iota(v)\om$ is harmonic, then~${(\cL_vJ)^*=\cL_vJ}$ 
and~${f_v=f_{Jv}=0}$ by Lemma~\ref{le:SELFADJOINT}, 
thus it follows from~(i) and~Lemma~\ref{le:DJ} that
$$
0 = \Lambda_\rho(J,\cL_vJ)
= \iota(J\bar\p_J^*(\cL_vJ)^*)\om 
= \iota(J\bar\p_J^*\cL_vJ)\om 
= -2\iota(\bar\p_J^*\bar\p_Jv)\om,
$$
and so~${\cL_vJ=2J\bar\p_Jv=0}$.
Now assume~${d\iota(v)\rho=0}$, choose~${\alpha_0\in\Om^{2\sn-2}(M)}$ 
such that $d{*d}\alpha_0=d\iota(Jv)\om$, 
and define $v_0\in\Vect(M)$ by $\iota(v_0)\rho:=d\alpha_0$.
Then~${\iota(Jv_0)\om={*d\alpha_0}}$ by~\eqref{eq:STARV}.
Hence~${d\iota(Jv_0)\om=d{*d\alpha_0}=d\iota(Jv)\om}$.
We also have~${d{*\iota(J(v-v_0))\om}=-d\iota(v-v_0)\rho=0}$.
Thus~${\iota(J(v-v_0))\om}$ is harmonic
and so~${\cL_{v-v_0}J=-J\cL_{J(v-v_0)}J=0}$.  
This proves Lemma~\ref{le:HOLV}. 
\end{proof}

\begin{lemma}[{\bf Harmonic Complex Anti-Linear Endomorphisms~$\Jhat$}]\label{le:JHATSTAR}
Let~$(M,J,\om)$ be a closed connected $2\sn$-dimensional
Ricci-flat K\"ahler manifold with volume form~${\rho:=\om^\sn/\sn!}$
and let~${\Jhat\in\Om^{0,1}_J(M,TM)}$ 
such that~${\bar\p_J\Jhat=0}$ and~${\bar\p_J^*\Jhat=0}$.
Then~${\bar\p_J\Jhat^*=0}$ and~${\bar\p_J^*\Jhat^*=0}$.
\end{lemma}

\begin{proof}
Let~$\nabla$ be the Levi-Civita connection of a K\"ahler metric~${\om(\cdot,J\cdot)}$.
Then the Bochner--Kodaira--Nakano identity
for~${\Jhat\in\Om^{0,1}_J(M,TM)}$ takes the form
\begin{equation}\label{eq:BKN}
\begin{split}
\bar\p_J^*\bar\p_J\Jhat+\bar\p_J\bar\p_J^*\Jhat
= \tfrac{1}{2}\nabla^*\nabla\Jhat + \tfrac{1}{2}[JQ,\Jhat] + \cT(\Jhat),
\end{split}
\end{equation}
where
$$
\nabla^*\nabla\Jhat=-\sum_i(\Nabla{e_i}\Nabla{e_i}\Jhat+\DIV(e_i)\Nabla{e_i}\Jhat),\qquad
\cT(\Jhat)u=\sum_iR^{\nabla\!}(e_i,u)\Jhat e_i
$$
for a local orthonormal frame~${e_1,\dots,e_{2\sn}}$,
and the skew-adjoint endomorphism~$Q$ is defined by
$
\inner{Q\cdot}{\cdot}=\Ric_{\rho,J}.
$
(See~\cite{DEMAILLY}.)  Since~${\cT(\Jhat)^*=\cT(\Jhat^*)}$, it follows that the 
operator~${\bar\p_J^*\bar\p_J+\bar\p_J\bar\p_J^*}$ commutes 
with the operator~${\Jhat\mapsto\Jhat^*}$ in the K\"ahler--Einstein 
case~${Q=\kappa J}$. This proves Lemma~\ref{le:JHATSTAR}.
\end{proof}

\begin{proof}[Proof of Theorem~\ref{thm:JINTZERO}]
\label{proof:JINTZERO}
Let~${\Jhat\in\Om^{0,1}_J(M,TM)}$ with~${\bar\p_J\Jhat=0}$.
Then part~(iii) of Lemma~\ref{le:HOLV} shows that the last 
two assertions in~(ii) are equivalent.  Next observe the following.

\medskip\noindent{\bf Claim~1.}
{\it ${\Richat_\rho(J,\Jhat)=0}$ if and only if~${f_\Jhat=0}$.}

\medskip\noindent{\bf Claim~2.}
{\it $\Om_{\rho,J}(\Jhat,\cL_vJ) = \int_M\bigl(f_\Jhat f_{Jv} - f_{J\Jhat}f_v\bigr)\rho$
for all~${v\in\Vect(M)}$.}

\medskip\noindent
By~\eqref{eq:RICHAT} and Lemma~\ref{le:LAMBDAFG}
we have~${\Richat_\rho(J,\Jhat)=\tfrac{1}{2}d\Lambda_\rho(J,\Jhat)
=-\tfrac{1}{2}d(df_\Jhat\circ J)}$ and this proves Claim~1.
Claim~2 follows from~\eqref{eq:LAMBDARHO} 
and Lemma~\ref{le:LAMBDAFG}. 

Sufficiency in~(i) and~(ii) follows directly from Claim~1 and Claim~2.
 To prove necessity, choose a symplectic form~$\om$ such 
that~$J$ is compatible with~$\om$ and~${\om^\sn/\sn!=\rho}$. 
Then~$(M,J,\om)$ is a Ricci-flat K\"ahler manifold.

We prove part~(i).
Assume~${\Om_{\rho,J}(\Jhat,\Jhat')=0}$ for all~$\Jhat'$ with~${\bar\p_J\Jhat'=0}$
and choose~${v\in\Vect(M)}$ such that~${\bar\p_J^*(\Jhat-\cL_vJ)=0}$.
Then~${\Lambda_\rho(J,\Jhat-\cL_vJ)=0}$ by Lemma~\ref{le:LAMBDAFG} 
and~${\bar\p_J(\Jhat-\cL_vJ)^*=0}$ by Lemma~\ref{le:JHATSTAR}.  
Thus it follows from the assumption and equations~\eqref{eq:LAMBDARHO}
and~\eqref{eq:RICuv} that
\begin{equation*}
0 = \Om_{\rho,J}(\Jhat-\cL_vJ,\cL_{Jv}J) 
= -\Om_{\rho,J}(\cL_vJ,\cL_{Jv}J) 
= \int_M\bigl(f_v^2+f_{Jv}^2\bigr)\rho.
\end{equation*}
Hence~${f_v=f_{Jv}=0}$.
Now fix an element~${\Jhat'\in\Om^{0,1}_J(M,TM)}$ with~${\bar\p_J\Jhat'=0}$.
Then~${\Om_{\rho,J}(\Jhat-\cL_vJ,J\Jhat')=0}$ by assumption and Claim~2.  Thus
$$
\inner{(\Jhat-\cL_vJ)^*}{\Jhat'}_{L^2}
= \int_M\trace((\Jhat-\cL_vJ)\Jhat')\rho
= -2\Om_{\rho,J}(\Jhat-\cL_vJ,J\Jhat')
= 0.
$$
This implies that there exists a~${\tau\in\Om^{0,2}_J(M,TM)}$ 
with~${\bar\p_J^*\tau=(\Jhat-\cL_vJ)^*}$.
Hence~${\bar\p_J\bar\p_J^*\tau=\bar\p_J(\Jhat-\cL_vJ)^*=0}$ 
and so~${\Jhat=\cL_vJ}$.  This proves~(i).

We prove part~(ii).
Assume~${\Richat_\rho(J,\Jhat)=0}$ and~${\Om_{\rho,J}(\Jhat,\Jhat')=0}$ 
for all~$\Jhat'$ that satisfy~${\bar\p_J\Jhat'=0}$ and~${\Richat_\rho(J,\Jhat')=0}$.
Then~${f_\Jhat=0}$ by Claim~1 and we choose~${H\in\Om^0(M)}$ 
such that~${d^*dH=-f_{J\Jhat}}$.  
Then~${\Lambda_\rho(J,\Jhat-\cL_{v_H}J) = 0}$ 
by Lemma~\ref{le:DIVH} and Lemma~\ref{le:LAMBDAFG}.
Now let~${\Jhat'\in\Om^{0,1}_J(M,TM)}$ with~${\bar\p_J\Jhat'=0}$ 
and choose~${F\in\Om^0(M)}$ such that~${d^*dF=-f_{\Jhat'}}$.  
Then~${\Richat_\rho(J,\Jhat'-\cL_{\nabla F}J) = 0}$ 
by~\eqref{eq:RICHAT}, Lemma~\ref{le:DIVH}, 
and Lemma~\ref{le:LAMBDAFG}.  Hence 
\begin{equation*}
\begin{split}
\Om_{\rho,J}(\Jhat-\cL_{v_H}J,\Jhat')
= 
\Om_{\rho,J}(\Jhat-\cL_{v_H}J,\Jhat'-\cL_{\nabla F}J) 
= 
\Om_{\rho,J}(\Jhat,\Jhat'-\cL_{\nabla F}J) 
= 0
\end{split}
\end{equation*}
by assumption and Theorem~\ref{thm:RICCI}.
So part~(i) asserts that there exists a vector field~$u$ 
with~${f_u=f_{Ju}=0}$ and~${\cL_uJ=\Jhat-\cL_{v_H}J}$.
Thus~${v:=u+v_H}$ is a divergence-free vector field with~${\cL_vJ=\Jhat}$.
This proves Theorem~\ref{thm:JINTZERO}.
\end{proof}

\bigbreak

We close this section with a well known lemma (see~\cite{KM})
that is used in Theorem~\ref{thm:CONNECTION}.
We include a proof for completeness of the exposition.

\begin{lemma}[{\bf Harmonic~$(0,2)$-Forms}]\label{le:PARALLEL}
Let~$(M,J,\om)$ be a closed connected $2\sn$-dimensional 
Ricci-flat K\"ahler manifold, let~$\nabla$ be the Levi-Civita connection
of the K\"ahler metric, let~${\Jhat\in\Om^{0,1}_J(M,TM)}$ 
with~${\Jhat + \Jhat^* = 0}$, 
and define~${\omhat:=\inner{\Jhat\cdot}{\cdot}\in\Om^2(M)}$.
Then~${\omhat^{1,1}_J=0}$ and the following are equivalent. 

\smallskip\noindent{\bf (i)}
${\bar\p_J\Jhat=0}$ and~${\bar\p_J^*\Jhat=0}$.

\smallskip\noindent{\bf (ii)}
${\nabla\Jhat=0}$.

\smallskip\noindent{\bf (iii)}
${\nabla\omhat=0}$.

\smallskip\noindent{\bf (iv)}
$\omhat$ is a harmonic $2$-form.

\smallskip\noindent{\bf (v)}
${d\omhat=0}$.
\end{lemma}

\begin{proof}
It follows directly from the definition that~${\omhat^{1,1}_J=0}$.
To prove that~(i) is equivalent to~(ii), let~$\cT$ be the operator in the 
proof of Lemma~\ref{le:JHATSTAR} and assume that~$\Jhat$ is skew-adjoint.
Then, by the first Bianchi identity, we have
\begin{equation*}
\begin{split}
&\cT(\Jhat)u
= 
\sum_iR^{\nabla\!}(e_i,u)\Jhat e_i 
=
\sum_{i,j}\inner{e_j}{\Jhat e_i}R^{\nabla\!}(e_i,u)e_j \\
&=
\sum_{i,j}\inner{\Jhat e_j}{e_i}\bigl(R^{\nabla\!}(e_j,e_i)u+R^{\nabla\!}(u,e_j)e_i\bigr) 
= 
\sum_jR^{\nabla\!}(e_j,\Jhat e_j)u - \cT(\Jhat)u.
\end{split}
\end{equation*}
Hence, for a local orthonormal frame with~${e_{\sn+i}=Je_i}$, we obtain
\begin{equation*}
\begin{split}
\cT(\Jhat)
= 
\tfrac{1}{2}\sum_iR^{\nabla\!}(e_i,\Jhat e_i)  
= 
\tfrac{1}{4}\sum_i\left(R^{\nabla\!}(e_i,\Jhat e_i) + R^{\nabla\!}(Je_i,\Jhat Je_i)\right)  
= 
0.
\end{split}
\end{equation*}
Thus~${\norm{\bar\p_J\Jhat}^2+\norm{\bar\p_J^*\Jhat}^2=\norm{\nabla\Jhat}^2}$
by~\eqref{eq:BKN} with~${Q=0}$ and so~(i) is equivalent to~(ii).  
The equivalence of~(ii) and~(iii) follows directly from the definition of~$\omhat$.
That~(iii) implies~(v) follows from Lemma~\ref{le:TORSION},
and~(iv) is equivalent to~(v) because $\omhat^{1,1}_J=0$ 
and so~${{*\omhat}=\omhat\wedge\om^{\sn-2}/(\sn-2)!}$.
That~(iv) implies~(iii) is a consequence
of the Weitzenb\"ock formula
\begin{equation}\label{eq:WEITZEN}
\begin{split}
&\Bigl(d^*d\omhat-dd^*\omhat-\nabla^*\nabla\omhat\Bigr)(u,v) \\
&= \sum_i\Bigl(\omhat\bigl(e_i,R^{\nabla\!}(u,v)e_i\bigr)
- \omhat\bigl(u,R^{\nabla\!}(v,e_i)e_i\bigr)
+ \omhat\bigl(v,R^{\nabla\!}(u,e_i)e_i\bigr)\Bigr)
\end{split}
\end{equation}
for~${\omhat\in\Om^2(M)}$, ${u,v\in\Vect(M)}$,
and a local orthonormal frame~${e_1,\dots,e_{2\sn}}$. 
In the Ricci-flat K\"ahler case with~${\omhat^{1,1}_J=0}$
the right hand side in~\eqref{eq:WEITZEN} vanishes and 
hence~${\norm{d\omhat}^2+\norm{d^*\omhat}^2=\norm{\nabla\omhat}^2}$.
This proves Lemma~\ref{le:PARALLEL}. 
\end{proof}


\section{Teichm\"uller space}\label{sec:TEICH}

\subsection*{The Calabi--Yau Teichm\"uller space}

Fix a closed connected oriented $2\sn$-manifold~$M$.
The {\bf Calabi--Yau Teich\-m\"uller space} is the space of 
isotopy classes of complex structures~$J$ with real 
first Chern class zero and nonempty K\"ahler cone~$\cK_J$.
It is denoted by 
\begin{equation}\label{eq:TEICH}
\begin{split}
\sT_0(M) 
&:= \sJ_{\INT,0}(M)/\Diff_0(M),\\ 
\sJ_{\INT,0}(M)
&:=
\biggl\{J\in\sJ_\INT(M)\,\bigg|\,
\begin{array}{l}c_1(TM,J)=0\in H^2(M;\R) \\
\mbox{and }J\mbox{ admits a K\"ahler form}
\end{array}\biggr\}.
\end{split}
\end{equation}
For every~${J\in\sJ_{\INT,0}(M)}$ the space of holomorphic 
vector fields is isomorphic to~$H^1(M;\R)$ by part~(iii) 
of Lemma~\ref{le:HOLV} and the Calabi--Yau Theorem.
Moreover, the Bogomolov--Tian--Todorov theorem asserts
that the obstruction class vanishes~\cite{BOGOMOLOV,TIAN,TODOROV}.
Hence the cohomology of the complex
\begin{equation}\label{eq:TCOMPLEX}
\Om^0(M,TM)\stackrel{\bar\p_J}{\longrightarrow}
\Om^{0,1}_J(M,TM)\stackrel{\bar\p_J}{\longrightarrow}
\Om^{0,2}_J(M,TM)
\end{equation}
has constant dimension.  (This assertion can also be derived
from~\cite[Proposition~9.30]{VOISIN} and Lemma~\ref{le:BETA}.)
It follows that the Teichm\"uller space~$\sT_0(M)$ 
is a smooth manifold~\cite{CATANESE,KS1,KS2,KS3,KURANISHI1,KURANISHI2} 
whose tangent space at~$J\in\sJ_{\INT,0}(M)$ is the cohomology 
of the complex~\eqref{eq:TCOMPLEX}, i.e.\ 
\begin{equation}\label{eq:TTEICH}
T_{[J]}\sT_0(M) 
= \frac{\ker(\bar\p_J:\Om^{0,1}_J(M,TM)\to\Om^{0,2}_J(M,TM))}
{\im(\bar\p_J:\Om^0(M,TM)\to\Om^{0,1}_J(M,TM))}.
\end{equation}

\begin{remark}\label{rmk:K3}\rm
The Teichm\"uller space is in general not Hausdorff, 
even for the K3 surface~\cite{GHS,VERBITSKY1}. 
Let~$(M,J)$ be a K3-surface that admits 
an embedded holomorphic sphere~${C\subset M}$ 
with self-intersection number~${C\cdot C=-2}$, 
and let~${\tau:M\to M}$ be a Dehn twist about~$C$.   
Then there exists a smooth family of complex 
structures~$\{J_t\in\sJ_{\INT,0}(M)\}_{t\in\C}$ and a smooth family 
of diffeomorphisms~${\{\phi_t\in\Diff_0(M)\}_{t\in\C\setminus\{0\}}}$
such that~$J_0=J$ and~${\phi_t^*J_t=\tau^*J_{-t}}$ for all~${t\in\C\setminus\{0\}}$.
Thus~$J_t$ and~$\tau^*J_{-t}$ represent the same class 
in Teichm\"uller space, however, their limits~${\lim_{t\to0}J_t=J_0}$
and~${\lim_{t\to0}\tau^*J_{-t}=\tau^*J_0}$ do not represent the 
same class in Teichm\"uller space because their effective cones differ. 
Namely, the class~${[C]\in H_2(M;\Z)}$ belongs to the effective cone of~$J_0$ 
while the class~${-[C]\in H_2(M;\Z)}$ belongs to the effective cone of~$\tau^*J_0$.

For general hyperK\"ahler manifolds the Teichm\"uller space becomes Hausdorff 
after identifying inseparable complex structures (see Verbitsky~\cite{VERBITSKY1,VERBITSKY2}), 
which are biholomorphic by a theorem of Huybrechts~\cite{HUY2}.
\end{remark}

\subsection*{Teichm\"uller space as a symplectic quotient}

Fix a positive volume form~$\rho$ on~$M$ and define
\begin{equation}\label{eq:TEICHRHO}
\begin{split}
\sT_0(M,\rho)
&:=
\sJ_{\INT,0}(M,\rho)/\Diff_0(M,\rho), \\
\sJ_{\INT,0}(M,\rho)
&:=
\bigl\{J\in\sJ_{\INT,0}(M)\,|\,\Ric_{\rho,J}=0\bigr\}.
\end{split}
\end{equation}
The tangent space of~$\sT_0(M,\rho)$
at~$J\in\sJ_{\INT,0}(M,\rho)$ is the quotient 
\begin{equation}\label{eq:TTEICHRHO}
T_{[J]}\sT_0(M,\rho) 
= \frac{\bigl\{\Jhat\in\Om^{0,1}_J(M,TM)\,|\,
\bar\p_J\Jhat=0,\,\Richat_\rho(J,\Jhat)=0\bigr\}}
{\bigl\{\cL_vJ\,|\,v\in\Vect(M),\,d\iota(v)\rho=0\bigr\}}.
\end{equation}

\begin{lemma}\label{le:iotarho}
The inclusion~${\iota_\rho:\sT_0(M,\rho)\to\sT_0(M)}$
is a diffeomorphism.
\end{lemma}

\begin{proof}
The map~$\iota_\rho$ is bijective by Theorem~\ref{thm:RICBC}.  
The derivative of~$\iota_\rho$ at an
element~${J\in\sJ_{\INT,0}(M,\rho)}$ is the inclusion
of the quotient~\eqref{eq:TTEICHRHO}
into~\eqref{eq:TTEICH}. It is injective 
because~${\Richat_\rho(J,\cL_vJ)=-\tfrac{1}{2}d(df_v\circ J)}$,
and so~${\Richat_\rho(J,\cL_vJ)=0}$ implies~${f_v=0}$ 
and thus~${d\iota(v)\rho=0}$. It is surjective because, 
if~${\Jhat\in\Om^{0,1}_J(M,TM)}$ satisfies~${\bar\p_J\Jhat=0}$
and~${F\in\Om^0(M)}$ is the unique solution of~${d^*dF=f_\Jhat}$
with mean value zero, then
$
\Lambda_\rho(J,\cL_{\nabla F}J)=dd^*dF\circ J=df_\Jhat\circ J
$
by Lemma~\ref{le:DIVH} 
and so~${\Richat_\rho(J,\Jhat+\cL_{\nabla F}J)=0}$
by Lemma~\ref{le:LAMBDAFG}.
This proves Lemma~\ref{le:iotarho}.
\end{proof}

By Lemma~\ref{le:HOLV} the quotient group~${\Diff_0(M,\rho)/\Diff^\ex(M,\rho)}$
acts trivially on~$\sJ_{\INT,0}(M,\rho)/\Diff^\ex(M,\rho)$.
Hence~$\sT_0(M,\rho)$ is a sub\-mani\-fold of the infinite-dimensional 
symplectic quotient~${\sW_0(M,\rho)=\sJ_0(M,\rho)/\Diff^\ex(M,\rho)}$ 
in~\eqref{eq:MRHO}, which is regular near~$\sT_0(M,\rho)$
by Lemma~\ref{le:HOLV} and Theorem~\ref{thm:COMP}.
Moreover,~$\Om_\rho$ descends to a symplectic form
on~$\sT_0(M,\rho)$ by Theorem~\ref{thm:JINTZERO}.
Here is a formula for the pushforward of this symplectic form
under the diffeomorphism~$\iota_\rho$ in Lemma~\ref{le:iotarho}.
Let~${V>0}$.  By Theorem~\ref{thm:RICBC} 
every~${J\in\sJ_{\INT,0}(M)}$ admits 
a unique positive volume form~${\rho=\rho_J\in\Om^{2\sn}(M)}$ 
such that
\begin{equation}\label{eq:rhoJ}
\Ric_{\rho,J}=0,\qquad \int_M\rho=V.
\end{equation}

\begin{definition}[{\bf Weil--Petersson Symplectic Form}]
\label{def:TEICH}
For~${J\in\sJ_{\INT,0}(M)}$, 
for the volume form $\rho_J$ in~\eqref{eq:rhoJ}, 
and for $\Jhat_1,\Jhat_2\in\Om^{0,1}_J(M,TM)$
with~${\bar\p_J\Jhat_i=0}$ and~${f_i,g_i}$ as 
in Lemma~\ref{le:LAMBDAFG}, define
\begin{equation}\label{eq:TEICHSYMP}
\Om_J(\Jhat_1,\Jhat_2) 
:= \int_M\Bigl(
\tfrac{1}{2}\trace\bigl(\Jhat_1J\Jhat_2\bigr)
- f_1g_2 + f_2g_1
\Bigr)\rho_J,
\end{equation}
\begin{equation}\label{eq:TEICHINNER}
\inner{\Jhat_1}{\Jhat_2}_J
:= \Om_J(\Jhat_1,-J\Jhat_2)
= \int_M\Bigl(
\tfrac{1}{2}\trace\bigl(\Jhat_1\Jhat_2\bigr)
- f_1f_2 - g_1g_2
\Bigr)\rho_J.
\end{equation}
\end{definition}

\bigbreak

\begin{theorem}\label{thm:TEICH}
{\bf (i)}
The $2$-form~$\Om_J$ in~\eqref{eq:TEICHSYMP} descends to a nondegenerate 
$2$-form on the quotient space~\eqref{eq:TTEICH} and defines a 
symplectic form on~$\sT_0(M)$. Its pullback to~$\sT_0(M,\rho)$
under the diffeomorphism~$\iota_\rho$ of Lemma~\ref{le:iotarho}
is the symplectic form induced by~$\Om_\rho$.

\smallskip\noindent{\bf (ii)}
If~${\phi:M\to M}$ is an orientation preserving diffeomorphism then
\begin{equation}\label{eq:TEICHSYMP1}
\Om_{\phi^*J}(\phi^*\Jhat_1,\phi^*\Jhat_2) 
= \Om_J(\Jhat_1,\Jhat_2) 
\end{equation}
for all~${\Jhat_1,\Jhat_2\in\Om^{0,1}_J(M,TM)}$ such that~${\bar\p_J\Jhat_i=0}$.
Thus the mapping class group~${\Gamma:=\Diff^+(M)/\Diff_0(M)}$ 
acts on~$\sT_0(M)$ by symplectomorphisms.  

\smallskip\noindent{\bf (iii)}
For a symplectic form~${\om\in\Om^2(M)}$ 
with real first Chern class zero define
\begin{equation}\label{eq:TEICHOM}
\begin{split}
\sT(M,\om) := \sJ_\INT(M,\om)/\!\!\sim, \quad
\sJ_\INT(M,\om) := \sJ_\INT(M)\cap\sJ(M,\om),
\end{split}
\end{equation}
where~${J_0\sim J_1}$ iff there is a 
diffeomorphism~${\phi\in\Diff_0(M)}$ such that~${\phi^*J_0=J_1}$. 
This space is a complex submanifold of~$\sT_0(M)$ and the symplectic 
form~\eqref{eq:TEICHSYMP} restricts to the standard K\"ahler form 
on~$\sT(M,\om)$.  The symmetric bilinear form~\eqref{eq:TEICHINNER} 
is positive on~$T_{[J]}\sT(M,\om)$ and is negative on its symplectic complement.
\end{theorem}

\begin{proof}
The proof has three steps.

\medskip\noindent{\bf Step~1.}
Let~${J\in\sJ_{\INT,0}(M)}$, let~${\Jhat\in\Om^{0,1}_J(M,TM)}$ 
such that~${\bar\p_J\Jhat=0}$, and let~${v\in\Vect(M)}$.
Then~${\Om_J(\Jhat,\cL_vJ)=0}$.

\medskip\noindent
Let~${\rho:=\rho_J}$ and let~${f=f_{\Jhat}}$ and~${g=f_{J\Jhat}}$ be as 
in Lemma~\ref{le:LAMBDAFG}.  Then
\begin{equation*}
\begin{split}
\Om_J(\Jhat,\cL_vJ)
&=
\tfrac{1}{2}\int_M\trace\bigl(\Jhat J\cL_vJ\bigr)
- \int_M\bigl(ff_{Jv}-gf_v)\rho \\
&=
\int_M\Lambda_\rho(J,\Jhat)\wedge\iota(v)\rho 
- \int_Mfd\iota(Jv)\rho + \int_Mgd\iota(v)\rho 
=
0
\end{split}
\end{equation*}
because~${\Lambda_\rho(J,\Jhat)=-df\circ J+dg}$.
This proves Step~1.

\medskip\noindent{\bf Step~2.}
{\it Let~${J\in\sJ_{\INT,0}(M)}$ and let~${\Jhat\in\Om^{0,1}_J(M,TM)}$ 
such that~${\bar\p_J\Jhat=0}$ and~${\Om_J(\Jhat,\Jhat')=0}$ 
for all~${\Jhat'\in\Om^{0,1}_J(M,TM)}$ with~${\bar\p_J\Jhat'=0}$.
Then~${\Jhat\in\im\bar\p_J}$.}

\medskip\noindent
Choose a symplectic form~$\om$ such that~${J\in\sJ_\INT(M,\om)}$ 
and~${\om^\sn/\sn!=\rho_J}$, and choose functions~${F,G\in\Om^0(M)}$ 
such that~${d^*dF=-f_{J\Jhat}}$ and~${d^*dG=-f_\Jhat}$. 
Then~${\Lambda_\rho(J,\cL_{v_F+\nabla G}J)=\Lambda_\rho(J,\Jhat)}$ 
by Lemma~\ref{le:DIVH} and hence, by Step~1,
$$
\Om_{\rho,J}(\Jhat-\cL_{v_F+\nabla G}J,\Jhat') 
= \Om_J(\Jhat-\cL_{v_F+\nabla G}J,\Jhat') 
= \Om_J(\Jhat,\Jhat')
= 0
$$
for all~${\Jhat'\in\Om^{0,1}_J(M,TM)}$ with~${\bar\p_J\Jhat'=0}$.
Hence, by part~(i) of Theorem~\ref{thm:JINTZERO}, there 
exists a vector field~$v$ with~${\cL_vJ=\Jhat-\cL_{v_F+\nabla G}J}$.
This proves Step~2.

\bigbreak

\medskip\noindent{\bf Step~3.}
{\it We prove Theorem~\ref{thm:TEICH}.}

\medskip\noindent
By Step~1 the $2$-form~\eqref{eq:TEICHSYMP} on~$\sJ_{\INT,0}(M)$
descends to~$\sT_0(M)$ and by Step~2 
the induced $2$-form on~$\sT_0(M)$ is nondegenerate.
Its pullback under the diffeomorphism~$\iota_\rho$
in Lemma~\ref{le:iotarho} is the restricton of the 
symplectic form~$\Om_\rho$ on~$\sJ(M)$ 
to the subquotient~$\sT_0(M,\rho)$.
Hence it is closed and this proves part~(i).
Part~(ii) follows directly from the definitions and part~(iii) 
holds because the tangent space~${T_{[J]}\sT(M,\om)}$ 
is the quotient of the space of self-adjoint 
endomorphisms~${\Jhat=\Jhat^*\in\Om^{0,1}_J(M,TM)}$
with~${\bar\p_J\Jhat=0}$ by those generated by Ha\-miltonian
and gradient vector fields.  This proves Theorem~\ref{thm:TEICH}.
\end{proof}

\subsection*{A symplectic connection}

Consider the fibration
$
\sT_0(M,\om)\hookrightarrow \sE_0(M)\to\sB_0(M),
$
where~$\sB_0(M)$ denotes the space of isotopy classes 
of symplectic forms with real first Chern class zero 
which admit compatible complex structures
and~$\sE_0(M)$ denotes the space of isotopy classes 
of Ricci-flat K\"ahler structures~$(\om,J)$ on~$M$. Thus
\begin{equation}\label{eq:SYMP0}
\begin{split}
\sT_0(M,\om)
&:= 
\sJ_{\INT,0}(M,\om)/\Symp(M,\om)\cap\Diff_0(M),\\
\sJ_{\INT,0}(M,\om) 
&:= 
\left\{J\in\sJ_\INT(M,\om)\,|\,\Ric_{\om,J}=0\right\},\\
\sB_0(M) 
&:= 
\sS_0(M)/\Diff_0(M),\\
\sS_0(M)
&:=
\left\{\om\in\Om^2(M)\,\bigg|\,
\begin{array}{l}d\om=0,\om^\sn>0,\,c_1^\R(\om)=0,\\
\sJ_\INT(M,\om)\ne\emptyset
\end{array}\right\},\\
\sE_0(M) 
&:= 
\sK_0(M)/\Diff_0(M),\\
\sK_0(M) 
&:= 
\left\{(\om,J)\,|\,
\om\in\sS_0(M),\,J\in\sJ_\INT(M,\om),\,\Ric_{\om,J}=0
\right\}.
\end{split}
\end{equation}
These quotient spaces are finite-dimensional manifolds.
Here is a list of the real dimensions for the cases 
where~$M$ is the $2\sn$-torus, 
the K3 surface, 
the Enriques surface, 
the quintic in~$\CP^4$, 
the banana manifold~${\it B}$ in~\cite{BRYAN}, 
and a rigid Calabi--Yau $3$-fold~${\it JB}$ 
introduced recently by Jim Bryan (as yet unpublished).
\begin{center}
\begin{tabular}{||c||c|c|c|c|c||}
\hline
\multirow{2}{*}{$M$}  & $\sT_0(M)$ & $\cK_J$ & $\sE_0(M)$ & $\sB_0(M)$ & $\sT_0(M,\om)$ \\
& $2h^{\sn-1,1}_{M,L}$ & $h^{1,1}$ & $h^{1,1}+2h^{\sn-1,1}_{M,L}$
& $h^{1,1}+2h^{2,0}$ & $2h^{\sn-1,1}_{M,L}-2h^{2,0}$ \\
 \hline
 \hline
 $\T^{2\sn}$   & $2\sn^2$   & $\sn^2$ & $3\sn^2$ & $2\sn^2-\sn$ & $\sn^2+\sn$ \\
 ${\it K3}$ & 40 & 20   & 60 & 22 & 38 \\
 Enriques & 20 & 10 &  30 & 10 & 20 \\
 Quintic    & 202 & 1 &  203 & 1 & 202 \\
 ${\it B}$  & 16 & 20 &  36 & 20 & 16 \\
 ${\it JB}$   & 0 & 4 &  4 & 4 & 0 \\
 \hline
\end{tabular}
\end{center}

The symplectic form~\eqref{eq:TEICHSYMP} in Theorem~\ref{thm:TEICH}
gives rise to a closed $2$-form on~$\sE_0(M)$ 
which restricts to the canonical K\"ahler form on each fiber
and whose kernel at~$(\om,J)$ is the space~$H^{1,1}_J(M)\times\{0\}$.
It gives rise to a symplectic connection on~$\sE_0(M)$ 
as in~\cite[Chapter~6]{MS}.  To describe this connection,
it will be convenient to use the notation
\begin{equation}\label{eq:omhat20}
\begin{split}
(J^*\omhat)(u,u')
:=&\,\,
\omhat(Ju,Ju'),\\
(\iota(J)\omhat)(u,u')
:=&\,\,
\omhat(Ju,u')+\omhat(u,Ju').
\end{split}
\end{equation}
for~${\omhat\in\Om^2(M)}$ and~${u,u'\in\Vect(M)}$.
The $2$-forms~$J^*\omhat$ and~$\iota(J)\omhat$ 
are closed whenever~$\omhat$ is harmonic, 
because~${\omhat^{2,0}_J = \tfrac{1}{4}(\omhat-J^*\omhat) - \tfrac{\bi}{4}\iota(J)\omhat}$.

The tangent space of~$\sK_0(M)$ at~${(\om,J)}$ with~${\rho:=\om^\sn/\sn!}$ 
is the space of all pairs~${(\omhat,\Jhat)\in\Om^2(M)\times\Om^{0,1}_J(M,TM)}$
that satisfy the conditions
\begin{equation}\label{eq:TK1}
\omhat(u,u')-\omhat(Ju,Ju') = \om(\Jhat u,Ju') + \om(Ju,\Jhat u')
\end{equation}
for all~${u,u'\in\Vect(M)}$ and
\begin{equation}\label{eq:TK2}
d\omhat=0,\qquad
\bar\p_J\Jhat=0,\qquad
\Richat_\rho(J,\Jhat)+\tfrac{1}{2}d(d\inner{\omhat}{\om}\circ J)=0.
\end{equation}
We will strengthen the last condition in~\eqref{eq:TK2} and require  
\begin{equation}\label{eq:TK3}
\Lambda_\rho(J,\Jhat) = - d\inner{\omhat}{\om}\circ J.
\end{equation}
The definition of the connection is based on the next lemma.

\begin{lemma}\label{le:A}
Let~${(\om,J)\in\sK_0(M)}$ be a Ricci-flat K\"ahler structure,
denote by~${\inner{\cdot}{\cdot}:=\om(\cdot,J\cdot)}$ the K\"ahler metric,
define~${\rho:=\om^\sn/\sn!}$, and let~${\omhat\in\Om^2(M)}$ be a closed $2$-form.
Then, for~${\Jhat\in\Om^{0,1}_J(M,TM)}$, the following are equivalent.

\smallskip\noindent{\bf (a)}
$\Jhat$ satisfies~\eqref{eq:TK1}, \eqref{eq:TK2}, \eqref{eq:TK3}, 
and, for all~${\Jhat'\in\Om^{0,1}_J(M,TM)}$,
\begin{equation}\label{eq:HORIZONTAL}
\Jhat'=(\Jhat')^*,\quad\bar\p_J\Jhat'=0
\qquad\implies\qquad\Om_J(\Jhat,\Jhat') = 0.
\end{equation}

\smallskip\noindent{\bf (b)}
If~${v\in\Vect(M)}$ satisfies
\begin{equation}\label{eq:OMHATV}
d^*(\omhat-d\iota(v)\om)=0,\qquad
d^*\iota(v)\om=0, 
\end{equation}
and~${\omhat_0\in\Om^2(M)}$
and~${\Jhat_0\in\Om^{0,1}_J(M,TM)}$ are defined by
\begin{equation}\label{eq:OMJHATZERO}
\begin{split}
\omhat_0:=\omhat-d\iota(v)\om,\qquad
\inner{\Jhat_0\cdot}{\cdot} 
= \tfrac{1}{2}\bigl(\omhat_0 - J^*\omhat_0\bigr),
\end{split}
\end{equation}
then~${\Jhat=\cL_vJ+\Jhat_0}$.

\smallskip\noindent
Moreover, for every closed $2$-form~$\omhat$ there exists 
a unique~${\Jhat\in\Om^{0,1}_J(M,TM)}$ 
that satisfies the equivalent conditions~(a) and~(b).
\end{lemma}

\begin{proof}
We prove in three steps that~(a) is equivalent to~(b).

\medskip\noindent{\bf Step~1.}
{\it Suppose~$\Jhat_1$ and~$\Jhat_2$ satisfy~(a).  Then~${\Jhat_1=\Jhat_2}$.}

\medskip\noindent
The difference~${\Jhat:=\Jhat_1-\Jhat_2}$ satisfies~(a) 
with~${\omhat=0}$. Hence~${\Jhat=\Jhat^*}$ by~\eqref{eq:TK1},
and~${\bar\p_J\Jhat=0}$ by~\eqref{eq:TK2},
and~${\Lambda_\rho(J,\Jhat)=0}$ by~\eqref{eq:TK3}.
Let~${\Jhat'\in\Om^{0,1}_J(M,TM)}$ with~${\bar\p_J\Jhat'=0}$,
choose a vector field~$v'$ such that~${\bar\p_J^*(\Jhat'-\cL_{v'}J)=0}$,
and define~${(\Jhat'-\cL_{v'}J\bigr)^+:=\tfrac{1}{2}(\Jhat'-\cL_{v'}J+(\Jhat'-\cL_{v'}J)^*)}$.
Then~${\bar\p_J(\Jhat'-\cL_{v'}J)^+=0}$ by Lemma~\ref{le:JHATSTAR},
hence~${\Om_J(\Jhat,(\Jhat'-\cL_{v'}J)^+)=0}$ by~\eqref{eq:HORIZONTAL},
and this implies that~${\Om_J(\Jhat,\Jhat') = \Om_J(\Jhat,\Jhat'-\cL_{v'}J)
= \Om_J(\Jhat,(\Jhat'-\cL_{v'}J)^+) = 0}$.
Thus by Theorem~\ref{thm:TEICH} there exists 
a vector field~$v$ with~${\cL_vJ=\Jhat}$.
Since~${\Jhat=\Jhat^*}$, Lemma~\ref{le:SELFADJOINT} 
asserts that there exist functions~${F,H\in\Om^0(M)}$ and a vector field~$v_0$ 
such that~${v:=v_0+v_H+\nabla F}$ and~${\iota(v_0)\om}$ is a harmonic $1$-form.
Thus~${\cL_{v_0}J=0}$ by Lemma~\ref{le:HOLV} and
$
dd^*dF\circ J - dd^*dH
= \Lambda_\rho(J,\cL_vJ)
= 0
$
by Lemma~\ref{le:DIVH}.
Hence~$F$ and~$H$ are constant, 
so~${\Jhat=\cL_{v_0}J=0}$ and~${\Jhat_1=\Jhat_2}$.
This proves Step~1.

\medskip\noindent{\bf Step~2.}
{\it Suppose~${v\in\Vect(M)}$ satisfies~\eqref{eq:OMHATV}, 
define~$\omhat_0$ and~$\Jhat_0$ by~\eqref{eq:OMJHATZERO},
and define~${\Jhat:=\cL_vJ+\Jhat_0}$. Then~$\Jhat$ satisfies~(a).}

\medskip\noindent
By~\eqref{eq:OMJHATZERO} and~\eqref{eq:STARV}, we have
\begin{equation}\label{eq:fvfJv}
\begin{split}
f_v\rho
&= d\iota(v)\rho
= d\iota(v)\om\wedge\tfrac{\om^{\sn-1}}{(\sn-1)!}
= \inner{\omhat-\omhat_0}{\om}\rho,\\
f_{Jv}\rho
&= d\iota(Jv)\rho
= d{*\iota(v)\om}
= 0.
\end{split}
\end{equation}
Moreover,~${\omhat_0}$ is a harmonic $2$-form
and so is~${\omhat_0-J^*\omhat_0}$. 
Thus~${\bar\p_J\Jhat_0=0}$ and~${\bar\p_J^*\Jhat_0=0}$
by Lemma~\ref{le:PARALLEL},
hence~${\Lambda_\rho(J,\Jhat_0)=0}$ by Lemma~\ref{le:LAMBDAFG}, 
and therefore~$\Lambda_\rho(J,\Jhat) = \Lambda_\rho(J,\cL_vJ) 
= -df_v\circ J = -d\inner{\omhat}{\om}\circ J$ by~\eqref{eq:fvfJv}.
Since~$\Jhat_0$ is skew-adjoint by~\eqref{eq:OMJHATZERO},
we have~${\Om_J(\Jhat,\Jhat') = \Om_J(\cL_vJ,\Jhat')=0}$
for all~$\Jhat'$ with~${\Jhat'=(\Jhat')^*}$ and~${\bar\p_J\Jhat'=0}$
and, by Lemma~\ref{le:ONEONE},
$$
\inner{(\Jhat-\Jhat^*)\cdot}{\cdot}
= \inner{(\cL_vJ-(\cL_vJ)^*)\cdot}{\cdot} + 2\inner{\Jhat_0\cdot}{\cdot} 
= \omhat-J^*\omhat.
$$
Hence~$\Jhat$ satisfies~(a) and this proves Step~2.

\medskip\noindent{\bf Step~3.}
{\it (a) is equivalent to~(b).}

\medskip\noindent
By Step~2, (b) implies~(a). Now assume~$\Jhat$ satisfies~(a)
and~$v$ satisfies~\eqref{eq:OMHATV}.
Define~$\omhat_0$ and~$\Jhat_0$ by~\eqref{eq:OMJHATZERO}.
Then~$\cL_vJ+\Jhat_0$ satisfies~(a) by Step~2 and so~${\Jhat=\cL_vJ+\Jhat_0}$
by Step~1.  Thus~$\Jhat$ satisfies~(b) and this proves Step~3.

Thus we have established the equivalence of~(a) and~(b).
By Step~2 and Hodge theory there exists 
an element~${\Jhat\in\Om^{0,1}_J(M,TM)}$ that satisfies~(a).
Uniqueness was established in Step~1 and this proves Lemma~\ref{le:A}. 
\end{proof}

\bigbreak

\begin{theorem}[{\bf Symplectic Connection}]\label{thm:CONNECTION}
{\bf (i)}
Let~${(\om,J)\in\sK_0(M)}$ and let~${\rho:=\om^\sn/\sn!}$. 
Then there exists a unique linear map
\begin{equation*}
\sA_{\om,J}:\Om^2(M)\supset\ker d\to\Om^{0,1}_J(M,TM)
\end{equation*}
which assigns to every closed real valued 
$2$-form~${\omhat\in\Om^2(M)}$ the unique infinitesimal complex 
structure~${\Jhat=\sA_{\om,J}(\omhat)\in\Om^{0,1}_J(M,TM)}$
that satisfies the equivalent conditions~(a) and~(b) in Lemma~\ref{le:A}.

\smallskip\noindent{\bf (ii)}
The~$1$-form~$\sA$ is $\Diff(M)$-equivariant, i.e.\ 
\begin{equation}\label{eq:A1}
\sA_{\phi^*\om,\phi^*J}(\phi^*\omhat) = \phi^*\sA_{\om,J}(\omhat)
\end{equation}
for every~${(\om,J)\in\sK_0(M)}$, every closed $2$-form~${\omhat}$, 
and every orientation preserving diffeomorphism~${\phi:M\to M}$.  
Moreover, 
\begin{equation}\label{eq:A2}
\sA_{\om,J}(d\iota(v)\om)=\cL_vJ
\end{equation}
for all~${(\om,J)\in\sK_0(M)}$ and all~${v\in\Vect(M)}$ 
with~${d\iota(Jv)\om^\sn=0}$.

\smallskip\noindent{\bf (iii)}
The curvature of the connection~$\sA$ is a $\Diff_0(M)$-equivariant 
$2$-form on~$\sS_0(M)$ with values in the space of smooth
functions on the fiber~$\sT_0(M,\om)$. It assigns to every~$\om\in\sS_0(M)$ 
and every pair~$\omhat_1,\omhat_2$ of closed $2$-forms on~$M$ 
the Hamiltonian function~${\sH_{\om;\omhat_1,\omhat_2}:\sT_0(M,\om)\to\R}$
given by 
\begin{equation}\label{eq:CURVATURE}
\begin{split}
\sH_{\om;\omhat_1,\omhat_2}(J)
&:= 
-\Om_J\bigl(\sA_{\om,J}(\omhat_1),\sA_{\om,J}(\omhat_2)\bigr)  \\
&\phantom{:}=
\tfrac{1}{2}\int_M\bigl(\iota(J)(\omhat_1-d\lambdahat_1)\bigr)
\wedge\omhat_2\wedge\tfrac{\om^{\sn-2}}{(\sn-2)!}
\end{split}
\end{equation}
for~${J\in\sJ_\INT(M,\om)}$ with~${\Ric_{\om,J}=0}$, 
where the $1$-form~$\lambdahat_1\in\Om^1(M)$
is chosen such that~${d^*(\omhat_1-d\lambdahat_1)=0}$
with respect to the K\"ahler metric~${\inner{\cdot}{\cdot}:=\om(\cdot,J\cdot)}$.  
The Hamiltonian vector field on~$\sT_0(M,\om)$ generated by this 
function is the vertical part of the Lie bracket of the horizontal lifts 
of two vector fields on~$\sB_0(M)$ that take the 
values~$\omhat_i$ at~$\om$ (see~\cite[Lemma~6.4.8]{MS}).
\end{theorem}

\begin{proof}
Part~(i)  and~\eqref{eq:A2} follow directly from Lemma~\ref{le:A},
while~\eqref{eq:A1} follows by combining uniqueness in part~(i) 
with the naturality conditions in Theorem~\ref{thm:RICCI} 
and Theorem~\ref{thm:TEICH}. This proves~(i) and~(ii).

For part~(iii) we must verify the second equality in~\eqref{eq:CURVATURE}.
Fix a symplectic form~${\om\in\Om^2(M)}$ with real first Chern class zero,
define~${\rho:=\om^\sn/\sn!}$, and let~${\omhat_1,\omhat_2\in\Om^2(M)}$ be closed.
Let~${J\in\sJ_\INT(M,\om)}$ such that~${\Ric_{\rho,J}=0}$,
choose~${\lambdahat_i\in\Om^1(M)}$ 
such that~${d^*(\omhat_i-d\lambdahat_i)=0}$ and~${d^*\lambdahat_i=0}$ 
with respect to the K\"ahler metric~${\inner{\cdot}{\cdot}=\om(\cdot,J\cdot)}$,
and define~${v_i\in\Vect(M)}$ by~${\iota(v_i)\om:=\lambdahat_i}$.

\bigbreak

By~(i) and Lemma~\ref{le:A}, we have
\begin{equation}\label{eq:Jomlahat}
\begin{split}
\sA_{\om,J}(\omhat_i) 
= 
\Jhat_i+\cL_{v_i}J,\quad
\inner{\Jhat_i\cdot}{\cdot} 
= 
\tfrac{1}{2}\bigl((\omhat_i-d\lambdahat_i)-J^*(\omhat_i-d\lambdahat_i)\bigr).
\end{split}
\end{equation}
Since~${\Lambda_\rho(J,\Jhat_i)=0}$ 
and~${f_{Jv_i}=0}$ by~\eqref{eq:fvfJv},
equation~\eqref{eq:Jomlahat} yields
\begin{equation*}
\begin{split}
\sH_{\om;\omhat_1,\omhat_2}(J)
= 
-\Om_J\bigl(\Jhat_1+\cL_{v_1}J,\Jhat_2+\cL_{v_2}J\bigr) 
=
-\tfrac{1}{2}\int_M\trace\bigl(\Jhat_1J\Jhat_2\bigr)\rho. 
\end{split}
\end{equation*}
Now choose a local orthonormal frame~${e_1,\dots,e_{2\sn}}$.  Then
\begin{equation*}
\begin{split}
-\tfrac{1}{2}\trace\bigl(\Jhat_1J\Jhat_2\bigr)
&=
-\tfrac{1}{2}\sum_i\inner{J\Jhat_1e_i}{\Jhat_2e_i} 
=
-\tfrac{1}{2}\sum_i\inner{J\Jhat_1e_i}{e_j}\inner{e_j}{\Jhat_2e_i} \\
&=
\sum_{i<j}\inner{\Jhat_1e_i}{Je_j}\inner{e_j}{\Jhat_2e_i}  
=
\tfrac{1}{2}\inner{\iota(J)\tau_1}{\tau_2},
\end{split}
\end{equation*}
where $\tau_i:=\inner{\Jhat_i\cdot}{\cdot}
=\tfrac{1}{2}(\omhat_i-d\lambdahat_i-J^*(\omhat_i-d\lambdahat_i))$.
This $2$-form satisfies $(\tau_i)^{1,1}_J=0$
and hence~${{*\tau_i}=\tau_i\wedge\om^{\sn-2}/(\sn-2)!}$.
Moreover,~${\iota(J)\tau_i=\iota(J)(\omhat_i-d\lambdahat_i)}$.
Thus
\begin{equation*}
\begin{split}
\sH_{\om;\omhat_1,\omhat_2}(J)
&=
-\tfrac{1}{2}\int_M\trace\bigl(\Jhat_1J\Jhat_2\bigr)\rho 
=
\tfrac{1}{2}\int_M\bigl(\iota(J)\tau_1\bigr)\wedge{*\tau_2} \\
&=
\tfrac{1}{4}\int_M\bigl(\iota(J)(\omhat_1-d\lambdahat_1)\bigr)\wedge
\bigl(\omhat_2-d\lambdahat_2-J^*(\omhat_2-d\lambdahat_2)\bigr)
\wedge \tfrac{\om^{\sn-2}}{(\sn-2)!} \\
&=
\tfrac{1}{2}\int_M\bigl(\iota(J)(\omhat_1-d\lambdahat_1)\bigr)\wedge
\bigl(\omhat_2-d\lambdahat_2\bigr)
\wedge \tfrac{\om^{\sn-2}}{(\sn-2)!} \\
&=
\tfrac{1}{2}\int_M\bigl(\iota(J)(\omhat_1-d\lambdahat_1)\bigr)\wedge
\omhat_2\wedge \tfrac{\om^{\sn-2}}{(\sn-2)!}.
\end{split}
\end{equation*}
This proves~\eqref{eq:CURVATURE}.  The right hand side
of~\eqref{eq:CURVATURE} depends only on the cohomology classes
of~$\omhat_1$ and~$\omhat_2$.  Hence it is invariant under 
the action of~${\Diff_0(M)\cap\Symp(M,\om)}$ on~$J$, 
because~$\phi^*\omhat_i-\omhat_i$ is exact for~${\phi\in\Diff_0(M)}$.
Thus it descends to a function on~$\sT_0(M,\om)$.
This proves Theorem~\ref{thm:CONNECTION}.
\end{proof}

The quotient of the Calabi--Yau Teichm\"uller space by the mapping 
class group is the Calabi--Yau moduli space~${\sM_0(M):=\sJ_{\INT,0}(M)/\Diff^+(M)}$.
For each K\"ahlerable symplectic form~$\om$
with real first Chern class zero
there is a polarized Calabi--Yau moduli
space~${\sM_0(M,\om):=\sJ_{\INT,0}(M,\om)/\Symp(M,\om)}$,
the quotient of the polarized Teichm\"uller space~$\sT_0(M,\om)$
by the symplectic mapping class group.
The study of the geometry of these moduli spaces is a vast
and extremely active area of research in both mathematics
and physics.  An overview of the subject and many references 
can be found in the article~\cite{YAU3} by Shing-Tung Yau.


\appendix

\section{Torsion-free connections}\label{app:NABLA}

Let~$M$ be an oriented $2\sn$-manifold.
We prove that a nondegenerate $2$-form on~$M$
is preserved by a torsion-free connection if and only if it is closed,
and that an almost complex structure on~$M$ is preserved by a torsion-free
connection if and only if it is integrable.  We use the sign conventions
$$
[\cL_u,\cL_v]+\cL_{[u,v]}=0
$$
for the Lie bracket and
\begin{equation}\label{eq:NJ}
N_J(u,v) = [u,v] + J[Ju,v]+J[u,Jv]-[Ju,Jv]
\end{equation}
for the Nijenhuis tensor.  If~$\nabla$ is a torsion-free connection 
on~$TM$ then 
\begin{equation}\label{eq:NJnabla}
N_J(u,v) = (\Nabla{u}J)Jv  + (\Nabla{Ju}J)v - (\Nabla{v}J)Ju  - (\Nabla{Jv}J)u.
\end{equation}

\begin{lemma}\label{le:TORSION}
Let~$M$ be a $2\sn$-manifold.

\smallskip\noindent{\bf (i)}
An almost complex structure $J$ is integrable if and only if 
there exists a torsion-free connection~$\nabla$ on~$TM$ 
such that~${\nabla J=0}$.  If~$J$ is integrable 
and~$\rho\in\Om^{2\sn}(M)$ is a volume form inducing 
the same orientation as~$J$, then there exists a  torsion-free 
connection~$\nabla$ on~$TM$ such that~${\nabla\rho=0}$ and~${\nabla J=0}$.

\smallskip\noindent{\bf (ii)}
A nondegenerate $2$-form~${\om\in\Om^2(M)}$
is closed if and only if there exists a torsion-free 
connection~$\nabla$ on~$TM$ such that~${\nabla\om=0}$.
\end{lemma}

\begin{proof}
We prove part~(i).  If $\nabla$ is a torsion-free connection 
with~${\nabla J=0}$ it follows directly from~\eqref{eq:NJnabla}
that~${N_J=0}$. Conversely suppose~$J$ is integrable
and let~$\rho$ be a volume form on~$M$ inducing 
the same orientation as~$J$. Fix a background metric $g$ on~$M$. 
Then~${g_J:=g+J^*g}$ is a metric with respect to which~$J$
is skew-adjoint and, if~${\dvol_J\in\Om^{2\sn}(M)}$ is the volume form 
of this metric, then the metric~${g_{\rho,J}:=(\rho/\dvol_J)^{1/\sn}g_J}$
has the volume form~$\rho$.  Let~$\nabla$ be the Levi-Civita 
connection of the metric~$g_{\rho,J}$.  
Then~$\nabla$ is torsion-free and~$\nabla\rho=0$.  
Let~${\alpha(u):= \tfrac{1}{2}\trace(J(\nabla J)u)}$ and define 
\begin{equation}\label{eq:NABLA3}
\begin{split}
\Habla{u}v 
:=&\,\, 
\Nabla{u}v - \tfrac12J(\Nabla{u}J)v
- \tfrac{1}{4}J(\Nabla{v}J)u - \tfrac{1}{4}(\Nabla{Jv}J)u \\
&\,\, 
+ \frac{\alpha(u)v+\alpha(v)u-\alpha(Ju)Jv-\alpha(Jv)Ju}{2\sn+2}.
\end{split}
\end{equation}
Then~${\habla\rho=0}$,~${\habla J=0}$, and a calculation shows 
that~${\mathrm{Tor}^\habla = -\tfrac14N_J}$, so~$\habla$ 
is torsion-free if and only if~$J$ is integrable. This proves~(i). 

We prove part~(ii). If $\nabla$ is torsion-free
and~${\nabla\om=0}$ then
\begin{equation*}
\begin{split}
d\om(u,v,w)
&=
\cL_u(\om(v,w)) + \cL_v(\om(w,u)) + \cL_w(\om(u,v)) \\
&\quad
+ \om([v,w],u) + \om([w,u],v) + \om([u,v],w) \\
&=
\om([v,w]-\Nabla{w}v+\Nabla{v}w,u) 
+ \om([w,u]-\Nabla{u}w+\Nabla{w}u,v) \\
&\quad
+ \om([u,v]-\Nabla{v}u+\Nabla{u}v,w) 
=
0.
\end{split}
\end{equation*}
Conversely, suppose~$\om$ is a symplectic form and choose an 
almost complex structure~$J$ on~$M$ that is compatible with~$\om$,
so~$\inner{\cdot}{\cdot}:=\om(\cdot,J\cdot)$ is a Riemannian metric.
Let~$\nabla$ be its Levi-Civita connection. Then
\begin{equation}\label{eq:nabroh1}
\inner{(\Nabla{u}J)v}{w} 
+ \inner{(\Nabla{v}J)w}{u} 
+ \inner{(\Nabla{w}J)u}{v} 
= d\om(u,v,w)
= 0
\end{equation}
for all~${u,v,w\in\Vect(M)}$ by~\cite[Lemma~4.1.14]{MS}.  
Define
\begin{equation}\label{eq:nabroh2}
\DTabla{u}v := \Nabla{u}v + A(u)v,\qquad
A(u)v := -\tfrac13J\bigl((\Nabla{u}J)v+(\Nabla{v}J)u\bigr).
\end{equation}
This connection is torsion-free and
satisfies~${JA(u) + A(u)^*J = \Nabla{u}J}$
for every vector field~${u\in\Vect(M)}$
by a straight forward calculation.  Hence
\begin{equation*}
\begin{split}
\om(\DTabla{u}v,w) + \om(v,\DTabla{u}w)
&=
\inner{J\Nabla{u}v+JA(u)v}{w} + \inner{Jv}{\Nabla{u}w+A(u)w} \\
&=
\inner{(JA(u)+A(u)^*J)v}{w}
+ \inner{J\Nabla{u}v}{w} 
+ \inner{Jv}{\Nabla{u}w} \\
&=
\inner{(\Nabla{u}J)v}{w}
+ \inner{J\Nabla{u}v}{w} 
+ \inner{Jv}{\Nabla{u}w} \\
&=
\cL_u\inner{Jv}{w} 
=
\cL_u\bigl(\om(v,w)\bigr)
\end{split}
\end{equation*}
for all~${u,v,w\in\Vect(M)}$.  
This proves Lemma~\ref{le:TORSION}.
\end{proof}

\begin{lemma}\label{le:RICCI11}
Let~$M$ be an oriented $2\sn$-manifold, 
let~$\rho\in\Om^{2\sn}(M)$ be a positive volume form, 
let~$J\in\sJ_\INT(M)$ be a complex structure compatible 
with the orientation, and let~$\nabla$ be a torsion-free 
$\rho$-connection such that~${\nabla J=0}$.  
Then~$\trace(JR^\nabla)$ is a~$(1,1)$-form.
\end{lemma}

\begin{proof}
Since $\nabla$ is torsion-free,~$R^\nabla$ satisfies the 
first Bianchi identity. Thus
\begin{equation*}
\begin{split}
&
R(u,v)w+JR(Ju,v)w+JR(u,Jv)w-R(Ju,Jv)w \\
&=
R(u,v)w+JR(Ju,v)w+JR(u,Jv)w+R(Jv,w)Ju+R(w,Ju)Jv \\
&=
R(u,v)w+JR(Ju,v)w+JR(w,Ju)v+JR(u,Jv)w+JR(Jv,w)u \\
&=
R(u,v)w-JR(v,w)Ju-JR(w,u)Jv \\
&=
R(u,v)w+R(v,w)u+R(w,u)v 
=
0
\end{split}
\end{equation*}
and so~${JR(u,v)-R(Ju,v)-R(u,Jv)-JR(Ju,Jv)=0}$.  
Take the trace to obtain~${\trace(JR(u,v))=\trace(JR(Ju,Jv))}$.  
This proves Lemma~\ref{le:RICCI11}.
\end{proof}


\section{The Bochner--Kodaira--Nakano identity}\label{app:BKN}

This appendix gives a self-contained proof of the Bochner--Kodaira--Nakano
identity~\eqref{eq:BKN} for complex anti-linear $1$-forms 
with values in the tangent bundle.   This formula plays a central role 
in the proofs of Lemma~\ref{le:JHATSTAR} and Lemma~\ref{le:PARALLEL}. 
Assume throughout that~$(M,\om,J)$ is a $2\sn$-dimensional K\"ahler manifold 
with the volume form~${\rho:=\om^\sn/\sn!}$, denote by~$\nabla$ the Levi-Civita connection 
of the K\"ahler metric, by~${R^{\nabla\!}\in\Om^2(M,\End(TM))}$ the Riemann
curvature tensor, by~${\Ric_{\rho,J}:=\tfrac{1}{2}\trace(JR^{\nabla\!})\in\Om^2(M)}$
the Ricci-form, and define the complex linear skew-adjoint 
endomorphism~${Q\in\Om^0(M,\End(TM))}$ by
\begin{equation}\label{eq:Q}
\inner{Qu}{v} = \Ric_{\rho,J}(u,v)
\end{equation}
for~${u,v\in\Vect(M)}$.  Define the 
map~${\cT:\Om^{0,1}_J(M,TM)\to\Om^{0,1}_J(M,TM)}$ by
\begin{equation}\label{eq:T}
\cT(\Jhat)u := \sum_{i=1}^{2\sn}R^{\nabla\!}(e_i,u)\Jhat e_i
\end{equation}
for~${\Jhat\in\Om^{0,1}_J(M,TM)}$ and~${u\in\Vect(M)}$,
where~${e_1,\dots,e_{2\sn}}$ is a local orthonormal frame 
of the tangent bundle.  With this notation the operator~$\nabla^*\nabla$
on the space of sections of the endomorphism bundle is given by
\begin{equation}\label{eq:NN}
\nabla^*\nabla\Jhat=-\sum_i\left(\Nabla{e_i}\Nabla{e_i}\Jhat+\DIV(e_i)\Nabla{e_i}\Jhat\right)
\end{equation}
for~${\Jhat\in\Om^{0,1}_J(M,TM)}$. 

\begin{theorem}[{\bf Bochner--Kodaira--Nakano}]\label{thm:BKN}
Every $\Jhat\in\Om^{0,1}_J(M,TM)$ satisfies the equation
\begin{equation}\label{eq:BKNW}
\begin{split}
\bar\p_J^*\bar\p_J\Jhat+\bar\p_J\bar\p_J^*\Jhat
= \tfrac{1}{2}\nabla^*\nabla\Jhat + \tfrac{1}{2}[JQ,\Jhat] + \cT(\Jhat).
\end{split}
\end{equation}
\end{theorem}

\begin{proof}
See page~\pageref{proof:BKN}.
\end{proof}

The proof relies on the following three lemmas.
Throughout we use the notation~${e_1,\dots,e_{2\sn}}$ 
for a local orthonormal frame of the tangent bundle, 
and it will sometimes be convenient to choose the frame 
such that~${e_{\sn+i}=Je_i }$ for~${i=1,\dots,\sn}$.  
We will use the notation~$\DIV(u)$ for the divergence
of a vector field~${u\in\Vect(M)}$; thus~${\DIV(u)=\trace(\nabla u)}$
and~${\DIV(u)\rho=d\iota(u)\rho}$.  For~${u\in\Vect(M)}$ we denote by~$\cL_u$ 
the Lie derivative (of any object on~$M$) in the direction of~$u$;
thus~${\cL_uf=df\circ u}$ for~${f\in\Om^0(M)}$ and~${\cL_uv=[v,u]}$
for~${v\in\Vect(M)}$.  (Note the sign convention for the Lie bracket.)

\begin{lemma}\label{le:FRAME}
Let~${B:TM\otimes TM\to E}$ be any bilinear form on the tangent 
bundle with values in a vector bundle~${E\to M}$.  Then
\begin{equation}\label{eq:FRAME1}
\sum_{i=1}^{2\sn}\Bigl(B(e_i,\Nabla{u}e_i) +  B(\Nabla{u}e_i,e_i)\Bigr) = 0
\end{equation}
for all~${u\in\Vect(M)}$.  Moreover,
\begin{equation}\label{eq:FRAME2}
\sum_{i=1}^{2\sn}\left(\Nabla{e_i}e_i + \DIV(e_i)e_i\right) = 0.
\end{equation}
\end{lemma}

\begin{proof}
Define the coefficients~$c_{ij}$ by
$$
c_{ij}:=\inner{\Nabla{u}e_i}{e_j}
$$
so that~${\Nabla{u}e_i=\sum_jc_{ij}e_j}$.
Then~${c_{ij}+c_{ji}=\cL_u\inner{e_i}{e_j}=0}$, because the frame is orthonormal 
and~$\nabla$ is a Riemannian connection.  Hence
\begin{equation*}
\begin{split}
\sum_{i=1}^{2\sn}B(e_i,\Nabla{u}e_i) 
&= 
\sum_{i=1}^{2\sn}\sum_{j=1}^{2\sn}c_{ij}B(e_i,e_j)  \\
&= 
-\sum_{j=1}^{2\sn}\sum_{i=1}^{2\sn}c_{ji}B(e_i,e_j)  \\
&=
- \sum_{j=1}^{2\sn}B(\Nabla{u}e_j,e_j) 
\end{split}
\end{equation*}
and this proves~\eqref{eq:FRAME1}.   
Since
$$
0=\cL_{e_i}\inner{e_j}{e_i} = \inner{\Nabla{e_i}e_j}{e_i}+\inner{e_j}{\Nabla{e_i}e_i}
$$
for all~$i$ and~$j$, we also have
\begin{equation*}
\begin{split}
\sum_{i=1}^{2\sn}\inner{e_j}{\DIV(e_i)e_i + \Nabla{e_i}e_i} 
&=
\DIV(e_j) - \sum_{i=1}^{2\sn}\inner{\Nabla{e_i}e_j}{e_i}  \\
&= 
\DIV(e_j) - \trace(\nabla e_j) \\
&=
0
\end{split}
\end{equation*}
for all~$j$.  This proves~\eqref{eq:FRAME2} and Lemma~\ref{le:FRAME}. 
\end{proof}

\begin{lemma}\label{le:Q}
The endomorphism~$Q$ in~\eqref{eq:Q} is given by
\begin{equation}\label{eq:QFRAME}
\begin{split}
Qu
= 
- \tfrac{1}{2}\sum_{i=1}^{2\sn}R^{\nabla\!}(e_i,Je_i)u 
=
\sum_{i=1}^{2\sn} JR^{\nabla\!}(u,e_i)e_i
\end{split}
\end{equation}
for~${u\in\Vect(M)}$.
\end{lemma}

\begin{proof}
Let~${u,v\in\Vect(M)}$. Then by~\eqref{eq:Q} we have
\begin{equation*}
\begin{split}
\inner{Qu}{v}
&=
\Ric_{\rho,J}(u,v) \\
&=
\tfrac{1}{2}\trace(JR^{\nabla\!}(u,v)) \\
&=
\tfrac{1}{2}\sum_{i=1}^{2\sn}\inner{e_i}{JR^{\nabla\!}(u,v)e_i} \\
&=
- \tfrac{1}{2}\sum_{i=1}^{2\sn}\inner{R^{\nabla\!}(u,v)e_i}{Je_i} \\
&=
- \tfrac{1}{2}\sum_{i=1}^{2\sn}\inner{R^{\nabla\!}(e_i,Je_i)u}{v}.
\end{split}
\end{equation*}
This proves the first equality in~\eqref{eq:QFRAME}.
Moreover, it follows from the first Bianchi identity that
\begin{equation*}
\begin{split}
Qu 
&=
-\tfrac{1}{2}\sum_{i=1}^{2\sn}R^{\nabla\!}(e_i,Je_i)u  \\
&=
\tfrac{1}{2}\sum_{i=1}^{2\sn}\Bigl(
R^{\nabla\!}(u,e_i)Je_i + R^{\nabla\!}(Je_i,u)e_i
\Bigr)  \\
&=
\tfrac{1}{2}\sum_{i=1}^{2\sn}\Bigl(
R^{\nabla\!}(u,e_i)Je_i - R^{\nabla\!}(e_i,u)Je_i
\Bigr)  \\
&=
\sum_{i=1}^{2\sn}R^{\nabla\!}(u,e_i)Je_i  \\
&=
\sum_{i=1}^{2\sn}JR^{\nabla\!}(u,e_i)e_i.
\end{split}
\end{equation*}
Here the third step uses a frame that satisfies~${e_{\sn+i}=Je_i}$
for~${i=1,\dots,\sn}$.  This proves~\eqref{eq:QFRAME} and Lemma~\ref{le:Q}.
\end{proof}

\begin{lemma}\label{le:DBARSTAR}
Let~${\Jhat\in\Om^{0,1}_J(M,TM)}$ and~${\tau\in\Om^{0,2}_J(M,TM)}$.
Then
\begin{eqnarray}
\label{eq:DBARSTAR1}
\bar\p_J^*\Jhat 
\!\!\!&=&\!\!\! 
-\sum_{i=1}^{2\sn}(\Nabla{e_i}\Jhat)e_i
\in\Vect(M), \\
\label{eq:DBARSTAR2}
\bar\p_J^*\tau 
\!\!\!&=&\!\!\! 
-\sum_{i=1}^{2\sn}(\Nabla{e_i}\tau)(e_i,\cdot)
\in\Om^{0,1}_J(M,TM).
\end{eqnarray}
\end{lemma}

\begin{proof}
For~${u\in\Vect(M)}$ the formal adjoint operator of
the covariant derivative~$\Nabla{u}$
is given by~${\Nabla{u}^*=-\Nabla{u}-\DIV(u)}$.
Fix a vector field~${X\in\Vect(M)}$.  Since
the~$(1,0)$-form~${\p_JX=\tfrac{1}{2}(\nabla X-J(\nabla X)J)}$
is orthogonal to~$\Jhat$, we have
\begin{equation*}
\begin{split}
\inner{\bar\p_J^*\Jhat}{X}_{L^2} 
&=
\inner{\Jhat}{\bar\p_JX}_{L^2} 
=
\inner{\Jhat}{\nabla X}_{L^2} 
=
\int_M\trace\left(\Jhat^*\nabla X\right)\rho \\
&=
\sum_{i=1}^{2\sn}\int_M\inner{\Jhat e_i}{\Nabla{e_i}X} \rho 
=
-\sum_{i=1}^{2\sn}\int_M\inner{\Nabla{e_i}(\Jhat e_i)+\DIV(e_i)\Jhat e_i}{X} \rho \\
&=
-\sum_{i=1}^{2\sn}\int_M\inner{\Nabla{e_i}(\Jhat e_i)-\Jhat\Nabla{e_i}e_i}{X} \rho 
=
-\sum_{i=1}^{2\sn}\int_M\inner{(\Nabla{e_i}\Jhat)e_i}{X} \rho.
\end{split}
\end{equation*}
Here the penultimate step uses~\eqref{eq:FRAME2} in Lemma~\ref{le:FRAME}.
This proves~\eqref{eq:DBARSTAR1}.   Now
\begin{equation*}
\begin{split}
\inner{\bar\p_J^*\tau}{\Jhat}_{L^2} 
&=
\inner{\tau}{d^\nabla\Jhat}_{L^2} 
=
\tfrac{1}{2}\sum_{i,j=1}^{2\sn}\int_M
\inner{\tau(e_i,e_j)}{(\Nabla{e_i}\Jhat)e_j-(\Nabla{e_j}\Jhat)e_i} 
\rho \\
&=
\sum_{i,j=1}^{2\sn}\int_M
\inner{\tau(e_i,e_j)}{(\Nabla{e_i}\Jhat)e_j} 
\rho \\
&=
\sum_{i,j=1}^{2\sn}\int_M
\inner{\tau(e_i,e_j)}{\Nabla{e_i}(\Jhat e_j)-\Jhat\Nabla{e_i}e_j} 
\rho \\
&=
- \sum_{i,j=1}^{2\sn}\int_M
\inner{\Nabla{e_i}(\tau(e_i,e_j))+\DIV(e_i)\tau(e_i,e_j)-\tau(e_i,\Nabla{e_i}e_j)}{\Jhat e_j} 
\rho \\
&=
- \sum_{i,j=1}^{2\sn}\int_M
\inner{(\Nabla{e_i}\tau)(e_i,e_j)}{\Jhat e_j} 
\rho.
\end{split}
\end{equation*}
Here the last but one step uses equation~\eqref{eq:FRAME1} and the last
step uses equation~\eqref{eq:FRAME2} in Lemma~\ref{le:FRAME}.
This proves equation~\eqref{eq:DBARSTAR2} and Lemma~\ref{le:DBARSTAR}.
\end{proof}

\begin{proof}[Proof of Theorem~\ref{thm:BKN}]\label{proof:BKN}
Let $\Jhat\in\Om^{0,1}_J(M,TM)$ 
and let~${X:=\bar\p_J^*\Jhat\in\Vect(M)}$ 
and~${\tau:=\bar\p_J\Jhat\in\Om^{0,2}_J(M,TM)}$.   
Then, by Lemma~\ref{le:DBARSTAR}, we have
\begin{equation*}
\begin{split}
X 
&= - \sum_i(\Nabla{e_i}\Jhat)e_i, \\
(\bar\p_JX)(u) 
&= 
\tfrac{1}{2}\bigl(\Nabla{u}X + J\Nabla{Ju}X\bigr),\\
\tau(u,v) 
&= 
\tfrac{1}{2}\bigl((
\Nabla{u}\Jhat)v - (\Nabla{v}\Jhat)u - (\Nabla{Ju}\Jhat)Jv + (\Nabla{Jv}\Jhat)Ju
\bigr), \\
(\bar\p_J^*\tau)(u) 
&= 
- \sum_i(\Nabla{e_i}\tau)(e_i,u)
\end{split}
\end{equation*}
for all~${u,v\in\Vect(M)}$.  Hence, by Lemma~\ref{le:FRAME},
\begin{equation*}
\begin{split}
&(\bar\p_J^*\bar\p_J\Jhat+\bar\p_J\bar\p_J^*\Jhat)(u) 
=
(\bar\p_J^*\tau)(u) + (\bar\p_JX)(u) \\
&=
- \sum_i\Nabla{e_i}\bigl(\tau(e_i,u)\bigr)
+ \sum_i\tau(\Nabla{e_i}e_i,u)
+ \sum_i\tau(e_i,\Nabla{e_i}u) \\
&\quad
-\tfrac{1}{2}\sum_i\Bigl(
\Nabla{u}((\Nabla{e_i}\Jhat)e_i) + J\Nabla{Ju}((\Nabla{e_i}\Jhat)e_i)
\Bigr) \\
&=
- \tfrac{1}{2}\sum_i\Nabla{e_i}\Bigl(
(\Nabla{e_i}\Jhat)u - (\Nabla{u}\Jhat)e_i 
- (\Nabla{Je_i}\Jhat)Ju + (\Nabla{Ju}\Jhat)Je_i
\Bigr) \\
&\quad
- \tfrac{1}{2}\sum_i\DIV(e_i)\Bigl(
(\Nabla{e_i}\Jhat)u - (\Nabla{u}\Jhat)e_i 
- (\Nabla{Je_i}\Jhat)Ju - (\Nabla{Ju}\Jhat)Je_i
\Bigr) \\
&\quad
+ \tfrac{1}{2}\sum_i\Bigl(
(\Nabla{e_i}\Jhat)\Nabla{e_i}u - (\Nabla{\Nabla{e_i}u}\Jhat)e_i 
- (\Nabla{Je_i}\Jhat)J\Nabla{e_i}u + (\Nabla{J\Nabla{e_i}u}\Jhat)Je_i
\Bigr) \\
&\quad
-\tfrac{1}{2}\sum_i\Bigl(
(\Nabla{u}\Nabla{e_i}\Jhat)e_i 
+ (\Nabla{e_i}\Jhat)\Nabla{u}e_i
+ J(\Nabla{Ju}\Nabla{e_i}\Jhat)e_i
+ J(\Nabla{e_i}\Jhat)\Nabla{Ju}e_i
\Bigr) \\
&=
-\tfrac{1}{2}\sum_i\Bigl(
\Nabla{e_i}\Nabla{e_i}\Jhat+\DIV(e_i)\Nabla{e_i}\Jhat
\Bigr)u 
+\tfrac{1}{2}\sum_i\Bigl(
\Nabla{e_i}\Nabla{Je_i}\Jhat+\DIV(e_i)\Nabla{Je_i}\Jhat
\Bigr)Ju \\
&\quad
+ \tfrac{1}{2}\sum_i\Bigl(
\Nabla{e_i}\Nabla{u}\Jhat - \Nabla{u}\Nabla{e_i}\Jhat + \Nabla{[e_i,u]}\Jhat
\Bigr)e_i \\
&\quad
- \tfrac{1}{2}\sum_i\Bigl(
\Nabla{e_i}\Nabla{Ju}\Jhat - \Nabla{Ju}\Nabla{e_i}\Jhat + \Nabla{[e_i,Ju]}\Jhat
\Bigr)Je_i \\
&= 
\tfrac{1}{2}(\nabla^*\nabla\Jhat)u 
+\tfrac{1}{2}\sum_i\Bigl(
\Nabla{e_i}\Nabla{Je_i}\Jhat+\DIV(e_i)\Nabla{Je_i}\Jhat
\Bigr)Ju \\
&\quad
+ \tfrac{1}{2}\sum_i[R^{\nabla\!}(e_i,u),\Jhat]e_i - \tfrac{1}{2}\sum_i[R^{\nabla\!}(e_i,Ju),\Jhat]Je_i.
\end{split}
\end{equation*}
Here we have used equation~\eqref{eq:NN} and the formula
$$
[R^{\nabla\!}(u,v),\Jhat] = \Nabla{u}\Nabla{v}\Jhat-\Nabla{v}\Nabla{u}\Jhat+\Nabla{[u,v]}\Jhat
$$
for the Riemann curvature tensor.
Now use the identities~\eqref{eq:FRAME2} in Lemma~\ref{le:FRAME}
and~\eqref{eq:QFRAME} in Lemma~\ref{le:Q} to obtain
\begin{equation*}
\begin{split}
-[Q,\Jhat]
&=
\tfrac{1}{2}\sum_i[R^{\nabla\!}(e_i,Je_i),\Jhat] \\
&=
\tfrac{1}{2}\sum_i\Bigl(\Nabla{e_i}\Nabla{Je_i}\Jhat-\Nabla{Je_i}\Nabla{e_i}\Jhat+\Nabla{[e_i,Je_i]}\Jhat\Bigr) \\
&=
\tfrac{1}{2}\sum_i\Bigl(\Nabla{e_i}\Nabla{Je_i}\Jhat - \Nabla{\Nabla{e_i}(Je_i)}\Jhat\Bigr) 
- \tfrac{1}{2}\sum_i\Bigl(\Nabla{Je_i}\Nabla{e_i}\Jhat - \Nabla{\Nabla{Je_i}e_i}\Jhat\Bigr) \\
&=
\sum_i\Bigl(\Nabla{e_i}\Nabla{Je_i}\Jhat - \Nabla{\Nabla{e_i}(Je_i)}\Jhat\Bigr) \\
&=
\sum_i\Bigl(\Nabla{e_i}\Nabla{Je_i}\Jhat + \DIV(e_i)\Nabla{Je_i}\Jhat\Bigr).
\end{split}
\end{equation*}
This yields the formula
\begin{equation*}
\begin{split}
&(\bar\p_J^*\bar\p_J\Jhat+\bar\p_J\bar\p_J^*\Jhat)u \\
&= 
\tfrac{1}{2}(\nabla^*\nabla\Jhat)u 
- \tfrac{1}{2}[Q,\Jhat]Ju 
+ \tfrac{1}{2}\sum_i\Bigl(
[R^{\nabla\!}(e_i,u),\Jhat]e_i - [R^{\nabla\!}(e_i,Ju),\Jhat]Je_i
\Bigr) \\
&= 
\tfrac{1}{2}(\nabla^*\nabla\Jhat)u 
- \tfrac{1}{2}[Q,\Jhat]Ju 
+ \tfrac{1}{2}\sum_i\Bigl(
[R^{\nabla\!}(e_i,u),\Jhat]e_i + [R^{\nabla\!}(Je_i,Ju),\Jhat]e_i
\Bigr) \\
&= 
\tfrac{1}{2}(\nabla^*\nabla\Jhat)u 
- \tfrac{1}{2}[Q,\Jhat]Ju 
+ \sum_i[R^{\nabla\!}(e_i,u),\Jhat]e_i  \\
&= 
\tfrac{1}{2}(\nabla^*\nabla\Jhat)u 
- \tfrac{1}{2}[Q,\Jhat]Ju 
- \sum_i\Jhat R^{\nabla\!}(e_i,u)e_i
+ \sum_iR^{\nabla\!}(e_i,u)\Jhat e_i  \\
&= 
\tfrac{1}{2}(\nabla^*\nabla\Jhat)u 
- \tfrac{1}{2}[Q,\Jhat]Ju 
- \sum_i\Jhat JJR^{\nabla\!}(u,e_i)e_i
+ \cT(\Jhat)u \\
&= 
\tfrac{1}{2}(\nabla^*\nabla\Jhat)u 
- \tfrac{1}{2}Q\Jhat Ju 
+ \tfrac{1}{2}\Jhat QJu 
- \Jhat JQu
+ \cT(\Jhat)u \\
&=
\tfrac{1}{2}(\nabla^*\nabla\Jhat)u 
+ \tfrac{1}{2}JQ\Jhat u 
- \tfrac{1}{2}\Jhat JQu 
+ \cT(\Jhat)u \\
&=
\tfrac{1}{2}(\nabla^*\nabla\Jhat)u + \tfrac{1}{2}[JQ,\Jhat] u + \cT(\Jhat)u.
\end{split}
\end{equation*}
Here we have used~\eqref{eq:T} and Lemma~\ref{le:Q}.
This proves Theorem~\ref{thm:BKN}.
\end{proof}


\section{Bott--Chern cohomology}\label{app:BOTTCHERN}

Let~$M$ be a closed connected $2\sn$-manifold and let~$J$
be an almost complex structure on~$M$. 

\begin{lemma}\label{le:ddc}
The almost complex structure~$J$ is integrable if and only if 
$2$-form $d(df\circ J)$ is of type $(1,1)$ for every smooth 
function~${f:M\to\R}$. 
\end{lemma}

\begin{proof}
The assertion follows directly from the identity
\begin{equation}\label{eq:ddcNJ}
d(df\circ J)(u,v) - d(df\circ J)(Ju,Jv) = df(JN_J(u,v)).
\end{equation}
for all~${f\in\Om^0(M)}$ and all~${u,v\in\Vect(M)}$. 
To prove this equation, choose a Riemannian metric on~$M$ 
with respect to which the almost complex structure is skew-adjoint, 
let~$\nabla$ be the Levi-Civita connection of this metric, 
denote by~$\nabla f$ the gradient of~$f$,
and abbreviate~${\tau_f:=d(df\circ J)}$.  Then
\begin{equation*}
\begin{split}
\tau_f(u,v)
&= 
\cL_u(df(Jv)) 
- \cL_v(df(Ju)) + df(J[u,v]) \\
&= 
\cL_u\inner{\nabla{f}}{Jv} 
- \cL_v\inner{\nabla{f}}{Ju} 
+ \inner{\nabla{f}}{J[u,v]} \\
&= 
\inner{\Nabla{u}\nabla{f}}{Jv} 
- \inner{\Nabla{v}\nabla{f}}{Ju}  
+ \inner{\nabla{f}}{(\Nabla{u}J)v-(\Nabla{v}J)u}
\end{split}
\end{equation*}
and hence
\begin{equation*}
\begin{split}
\tau_f(Ju,Jv)
= 
-\inner{\Nabla{Ju}\nabla{f}}{v} + \inner{\Nabla{Jv}\nabla{f}}{u} 
+ \inner{\nabla{f}}{(\Nabla{Ju}J)Jv-(\Nabla{Jv}J)Ju}.
\end{split}
\end{equation*}
Take the difference of these equations and use the fact
that the covariant Hessian of~$f$ is symmetric to obtain
\begin{equation*}
\begin{split}
\tau_f(u,v)
- \tau_f(Ju,Jv)
&= 
\inner{\nabla{f}}
{(\Nabla{u}J)v-(\Nabla{v}J)u-(\Nabla{Ju}J)Jv+(\Nabla{Jv}J)Ju}  \\
&=
\inner{\nabla f}{JN_J(u,v)} \\
&=
df(JN_J(u,v)).
\end{split}
\end{equation*}
Here the second equality follows from~\eqref{eq:NJnabla}.
This proves Lemma~\ref{le:ddc}. 
\end{proof}

Throughout the remainder of this appendix we assume that~$J$ 
is integrable.  Then~$\bar\p\circ\bar\p=0$ and so
\begin{equation}\label{eq:ddc}
\bi\p\bar\p f=\bi d\bar\p f=-\tfrac12d(df\circ J)
\end{equation}
for every smooth function~${f:M\to\R}$.  

The Bott--Chern cohomology groups of~$M$ are defined by
$$
H^{p,q}_\BC(M;\C):=\frac{\left\{\tau\in\Om^{p,q}(M;\C)\,|\,\p\tau=0,\,\bar\p\tau=0\right\}}
{\left\{\p\bar\p\sigma\,|\,\sigma\in\Om^{p-1,q-1}(M;\C)\right\}}.
$$
These are finite-dimensional complex vector spaces~\cite{BC}. 
In the K\"ahler case the $\p\bar\p$-lemma asserts that
every exact $(p,q)$-form on~$M$ belongs to the image of 
the operator~${\p\bar\p:\Om^{p-1,q-1}(M;\C)\to\Om^{p,q}(M;\C)}$.
This implies that the Bott--Chern cohomology group
agrees with the deRham cohomology group
$$
H^{p,q}_\dR(M;\C) 
:= \frac{\left\{\tau\in\Om^{p,q}(M;\C)\,|\,d\tau=0\right\}}
{\left\{d\alpha\,|\,\alpha\in\Om^{p+q-1}(M;\C),\,d\alpha\in\Om^{p,q}(M)\right\}}
$$
and the direct sum of these groups is~$H^{p+q}_\dR(M;\C)$.
In general, there is a surjective map~${H^{p,q}_\BC(M;\C)\to H^{p,q}_\dR(M;\C)}$
which may have a nontrivial kernel. 

In the present paper the relevant case is~${p=q=1}$.  
Denote by~$\Om^{1,1}(M)$ the space of 
real valued~$2$-forms~${\tau\in\Om^2(M)}$ that 
satisfy~${\tau(u,v)=\tau(Ju,Jv)}$ for all~${u,v\in\Vect(M)}$ 
and define
\begin{equation}\label{eq:H11BC}
H^{1,1}_\BC(M)
:=
\frac{\left\{\tau\in\Om^{1,1}(M)\,|\,d\tau=0\right\}}
{\left\{d(df\circ J)\,|\,f\in\Om^0(M;\R)\right\}}.
\end{equation}
The kernel of the
homomorphism~${H^{1,1}_\BC(M)\to H^{1,1}_\dR(M)}$ 
has the dimension
\begin{equation}\label{eq:kappaJ}
\kappa(J) := \dim\frac{\left\{d\lambda\,|\,
\lambda\in\Om^1(M),\,d\lambda\in\Om^{1,1}(M)\right\}}
{\left\{d(df\circ J)\,|\,f\in\Om^0(M)\right\}}.
\end{equation}
To examine this invariant, choose a nondegenerate 
$2$-form~$\om$ that is compatible with~$J$,
so~${\inner{\cdot}{\cdot}:=\om(\cdot,J\cdot)}$ 
is a Riemannian metric on~$M$,
and denote by~${*:\Om^k(M)\to\Om^{2\sn-k}(M)}$
the Hodge $*$-operator of this metric. Then
\begin{equation}\label{eq:HODGE1}
{*\lambda} = -(\lambda\circ J)\wedge\frac{\om^{\sn-1}}{(\sn-1)!},\qquad
{*\left(\lambda\wedge\frac{\om^{\sn-1}}{(\sn-1)!}\right)} = -\lambda\circ J
\end{equation}
for~${\lambda\in\Om^1(M)}$. 
For~${\sn\ge2}$ the operator~${\tau\mapsto*(\tau\wedge\om^{\sn-2}/(\sn-2)!)}$
on~${\Om^2(M)}$ has trace zero and eigenvalues~$\pm1$ and~${\sn-1}$.
Define
\begin{equation*}
\begin{split}
\Om^\pm(M) 
:= 
\left\{\tau\in\Om^2(M)\,\Big|\,
{*\left(\tau\wedge\frac{\om^{\sn-2}}{(\sn-2)!}\right)}=\pm\tau\right\}.
\end{split}
\end{equation*}
Then~$\Om^{1,1}(M)=\Om^0(M)\om\oplus\Om^-(M)$ and
\begin{equation}\label{eq:HODGE2}
{*\left(\tau\wedge\frac{\om^{\sn-2}}{(\sn-2)!}\right)}
= \inner{\tau}{\om}\om - J^*\tau,\qquad
\inner{\tau}{\om}\frac{\om^\sn}{\sn!} := \tau\wedge\frac{\om^{\sn-1}}{(\sn-1)!},
\end{equation}
for all~$\tau\in\Om^2(M)$, where~${(J^*\tau)(u,v):=\tau(Ju,Jv)}$
for~${u,v\in\Vect(M)}$.

\bigbreak

\begin{lemma}\label{le:HODGE}
{\bf (i)} The numbers
\begin{equation}\label{eq:kappaJ1}
\begin{split}
\kappa_0(J,\om) 
&:= \dim\frac{\left\{\inner{d\lambda}{\om}\,|\,
\lambda\in\Om^1(M),\,d\lambda\in\Om^{1,1}(M)\right\}}
{\left\{\inner{d(df\circ J)}{\om}\,|\,f\in\Om^0(M)\right\}},\\
\kappa_1(J,\om) 
&:= \dim\left\{d\lambda\,|\,\lambda\in\Om^1(M),\,
d\lambda\in\Om^{1,1}(M),\,\inner{d\lambda}{\om}=0\right\}
\end{split}
\end{equation}
are finite and ${\kappa_0(J,\om)\in\{0,1\}}$.
Moreover, the number~$\kappa_0(J,\om)$ vanishes 
whenever~$\om^{\sn-1}$ is closed,
and~$\kappa_1(J,\om)$ vanishes 
whenever~$\om^{\sn-2}$ is closed.

\smallskip\noindent{\bf (ii)}
$\kappa(J)=\kappa_0(J,\om)+\kappa_1(J,\om)$.

\smallskip\noindent{\bf (iii)}
${\kappa(J)=0}$ if and only if for every exact~$(1,1)$-form~${\tau\in\Om^{1,1}(M)}$ 
there exists a function~${f\in\Om^0(M)}$ such that~${d(df\circ J)=\tau}$.
Moreover,~${\kappa(J)=0}$ whenever~$J$ admits a K\"ahler form.
\end{lemma}

\begin{proof}
We prove part~(i).  Define the operator~${L_0:\Om^0(M)\to\Om^0(M)}$ by
\begin{equation}\label{eq:L0}
L_0f := \inner{d(df\circ J)}{\om} 
= d^*df - \frac{(df\circ J)\wedge d(\om^{\sn-1}/(\sn-1)!)}{\om^\sn/\sn!}
\end{equation}
for~${f\in\Om^0(M)}$. Here the second equation follows from~\eqref{eq:HODGE1}.
Then~$L_0$ is a second order elliptic operator without zeroth order terms. 
Thus its kernel consists of the constant functions by the Hopf maximum principle. 
Moreover,~$L_0$ is an index zero Fredholm operator and so
has a one-dimensional cokernel.  Thus~${\kappa_0(J,\om)\in\{0,1\}}$.
If~$\om^{\sn-1}$ is closed then~$L_0=d^*d$ and the 
function~${\inner{d\lambda}{\om}}$ has mean value zero 
for all~${\lambda\in\Om^1(M)}$, so~${\kappa_0(J,\om)=0}$.

Next define the operator~${d^+:\Om^1(M)\to\Om^1(M)}$ by
\begin{equation}\label{eq:DPLUS}
d^+\lambda:= d\lambda +*\left(d\lambda\wedge\tfrac{\om^{\sn-2}}{(\sn-2)!}\right)
\end{equation}
for~$\lambda\in\Om^1(M)$.  Then it follows from~\eqref{eq:HODGE2} that
a $1$-form~${\lambda\in\Om^1(M)}$ satisfies~${d^+\lambda=0}$ if and only 
if~${d\lambda\in\Om^{1,1}(M)}$ and~${\inner{d\lambda}{\om}=0}$. 
Thus
\begin{equation}\label{eq:kappa1}
\kappa_1(J,\om) = \dim(d(\ker d^+)) = \dim(\ker d^+/\ker d).
\end{equation}
Now define the 
operator~${L_1:\Om^1(M)\to\Om^1(M)}$ by 
\begin{equation}\label{eq:L1}
L_1\lambda := d^*d^+\lambda+dd^*\lambda
= (d^*d+dd^*)\lambda - *\left(d\lambda\wedge d\tfrac{\om^{\sn-2}}{(\sn-2)!}\right).
\end{equation}
Then~$L_1$ is a second order elliptic operator and so has a finite dimensional kernel.
Its kernel contains the space~$\ker d^+\cap\ker d^*$, which in turn contains the 
space~${H^1(M):=\ker d\cap\ker d^*}$ of harmonic~$1$ forms. 
Moreover, the quotient~${(\ker d^+\cap\ker d^*)/H^1(M)}$
has dimension~$\kappa_1(J,\om)$ by~\eqref{eq:kappa1}. Thus 
$$
\dim H^1(M)+\kappa_1(J,\om)=\dim(\ker d^+\cap\ker d^*)\le\dim\ker L_1<\infty.
$$  
If~${d\om^{\sn-2}=0}$ then~${L_1=d^*d+dd^*}$ and so~${\kappa_1(J,\om)=0}$.  
This proves part~(i).

\bigbreak

To prove part~(ii), assume first that~${\kappa_0(J,\om)=1}$
and choose~${\lambda_0\in\Om^1(M)}$ 
such that~${d\lambda_0\in\Om^{1,1}(M)}$ 
and~${\inner{d\lambda_0}{\om}\notin\im L_0}$. 
Then
\begin{equation}\label{eq:OM0}
\Om^0(M)=\im L_0\oplus\R\inner{d\lambda_0}{\om}.
\end{equation}
We prove that
\begin{equation}\label{eq:OM11}
\Om^{1,1}(M)\cap\im d 
= \R d\lambda_0 \oplus d(\ker d^+) 
\oplus \{d(df\circ J)\,|\,f\in\Om^0(M)\}.
\end{equation}
First note that~${d\lambda_0\notin d(\ker d^+)}$ by~\eqref{eq:HODGE2} 
and~\eqref{eq:DPLUS}, and that~$d\lambda_0\ne d(df\circ J)$ for 
all~${f\in\Om^0(M)}$ because~${\inner{d\lambda_0}{\om}\notin\im L_0}$.
Moreover, if~${\lambda\in\Om^1(M)}$ satisfies~${d^+\lambda=0}$ 
and~${d\lambda=d(df\circ J)}$ for some~${f\in\Om^0(M)}$, then
$$
L_0f = \inner{d(df\circ J)}{\om} = \inner{d\lambda}{\om}=0,
$$
and so~$f$ is constant.  Thus the right hand side of~\eqref{eq:OM11}
is a direct sum decomposition.  Now let~${\lambda\in\Om^1(M)}$ 
with~${d\lambda\in\Om^{1,1}(M)}$. Then by~\eqref{eq:OM0} 
there exist a real number~$s$ and a function~$f\in\Om^0(M)$ such that
$$
\inner{d\lambda}{\om} = \inner{d(df\circ J)}{\om} + s\inner{d\lambda_0}{\om}.
$$
Define~${\lambda_1:=\lambda-df\circ J-s\lambda_0}$.
Then~${d\lambda_1\in\Om^{1,1}(M)}$ and~$\lambda_1\in\ker d^+$ 
by~\eqref{eq:HODGE2} and~\eqref{eq:DPLUS},
and we have~${d\lambda=sd\lambda_0+d\lambda_1+d(df\circ J)}$.
This proves~\eqref{eq:OM11}.  It follows from~\eqref{eq:kappa1} 
and~\eqref{eq:OM11} that~${\kappa(J)=1+\kappa_1(J,\om)}$.
This proves part~(ii) in the case~${\kappa_0(J,\om)=1}$.
The proof in the case~${\kappa_0(J,\om)=0}$ is analogous.

Part~(iii) follows directly from the definitions and parts~(i) and~(ii).
This proves Lemma~\ref{le:HODGE}.
\end{proof}
 
\begin{corollary}\label{cor:HODGE}
Let~$M$ be closed connected oriented smooth four-manifold,
let~$J$ be a complex structure on~$M$ that is compatible
with the orientation, let~${\om\in\Om^2(M)}$ be a 
nondegenerate $2$-form that is compatible with~$J$,
and equip~$M$ with the Riemannian 
metric~${\inner{\cdot}{\cdot}:=\om(\cdot,J\cdot)}$.
Then
$$
\kappa_1(J,\om)=0 
$$
and~${\kappa(J)=\kappa_0(J,\om)\in\{0,1\}}$.
Moreover, the following are equivalent.

\smallskip\noindent{\bf (i)}
${\kappa(J)=1}$.
 
\smallskip\noindent{\bf (ii)}
${\Om^0(M)=\left\{\inner{d\lambda}{\om}\,|\,
\lambda\in\Om^1(M),\,d\lambda\in\Om^{1,1}(M)\right\}}$.
 
\smallskip\noindent{\bf (iii)}
Every self-dual harmonic $2$-form~$\tau\in\Om^2(M)$
satisfies~${\inner{\tau}{\om}\equiv0}$.
 
\smallskip\noindent{\bf (iv)}
$H^{2,+}_{\om,J}(M)\subset\Om^{2,0}_J(M)\oplus\Om^{0,2}_J(M)$.
\end{corollary}
 
\begin{proof}
Since~$M$ has dimension four, $\om^{\sn-2}$ is the constant 
function~$1$ and so~${\kappa_1(J,\om)=0}$ by Lemma~\ref{le:HODGE}.
This shows that~${\kappa(J)=\kappa_0(J,\om)\in\{0,1\}}$.
The equivalence of~(i) and~(ii) follows from the fact
that the operator~$L_0$ in~\eqref{eq:L0} has a one-dimensional cokernel
and that~$\kappa_0(J,\om)=1$ if and only if 
$$
\im L_0 \subsetneq
\left\{\inner{d\lambda}{\om}\,\big|\,\lambda\in\Om^1(M),\,d\lambda\in\Om^{1,1}(M)\right\}
\subset\Om^0(M).
$$
Next observe that the $L^2$-orthogonal complement of the 
image of the operator~${d^+:\Om^1(M)\to\Om^+_{\om,J}(M)}$ 
is the space~$H^{2,+}_{\om,J}(M)$ of self-dual harmonic $2$-forms,
and that~(ii) holds if and only if~${\Om^0(M)\om\subset\im d^+}$ 
(see equation~\eqref{eq:HODGE2}).
Thus~(ii) holds if and only if the spaces~${H^{2,+}_{\om,J}(M)}$
and~$\Om^0(M)\om$ are $L^2$ orthogonal to each other, 
and this is equivalent to~(iii).  
The equivalence of~(iii) and~(iv) follows from the fact 
that the space~${\Om^{2,0}_J(M;\C)\oplus\Om^{0,2}_J(M;\C)}$
intersects~$\Om^2(M)$ in the space of all~${\tau\in\Om^2(M)}$
that satisfy~${\tau(u,v)+\tau(Ju,Jv)=0}$ for all~${u,v\in\Vect(M)}$, 
the $L^2$ orthogonal complement of~$\Om^0(M)\om$ in~$\Om^{2,+}_{\om,J}(M)$. 
This proves Corollary~\ref{cor:HODGE}.
\end{proof}

On a closed connected oriented smooth four-manifold~$M$
Corollary~\ref{cor:HODGE} shows that the set of complex 
structures~$J$ that satisfy~${\kappa(J)=0}$ is open, 
and that~${\kappa(J)=1}$ whenever~${b^{2,+}(M)=0}$.

\begin{remark}\label{rmk:c1BC}\rm
Let~$E$ be a holomorphic vector bundle over a complex manifold~$(M,J)$. 
Then, for every Hermitian metric~$h$ on~$E$, there exists 
a unique Hermitian connection $\nabla$ on~$E$
such that~${\bar\p^{\nabla}=\bar\p:\Om^0(M,E)\to\Om^{0,1}(M,E)}$.
The real valued $2$-form~${\frac{\bi}{2\pi}\trace^c(F^\nabla)}$
is a closed $(1,1)$-form which represents the first 
Chern class of~$E$.  If~${h'(s_1,s_2)=h(s_1,Qs_2)}$
is another Hermitian structure 
(with~$Q$ a section of the bundle of positive definite
Hermitian automorphisms with respect to~$h$) 
then the corresponding Hermitian connection is 
given by~${\nabla' = \nabla + Q^{-1}\p Q}$
and the complex trace of its curvature is 
$$
\trace^c(F^{\nabla'}) = \trace^c(F^\nabla) - \p\bar\p f,\qquad
f:= \mathrm{det}^c(Q):M\to\R.
$$
Thus by~\eqref{eq:ddc} the first Chern class
$$
c_1(E)=\left[\frac{\bi}{2\pi}\trace^c(F^\nabla)\right]_\dR\in H^2_\dR(M)
$$ 
in de Rham cohomology lifts to a well defined class
$$
c_{1,\BC}(E) := \left[\frac{\bi}{2\pi}\trace^c(F^\nabla)\right]_\BC\in H^{1,1}_\BC(M)
$$
in Bott--Chern cohmology, called the {\bf first  Bott--Chern class of~$E$}.
(For more details see~\cite{AT1,AT2,BGS,BC}).
\end{remark}


\section{Complex structures and n-forms}\label{app:NFORM}

Fix a closed connected oriented $2\sn$-manifold~$M$
and a complex line bundle~${L\to M}$ with a Hermitian form $\inner{s_1}{s_2}$ 
for~${s_1,s_2\in\Om^0(M,L)}$ (complex anti-linear 
in the first variable and complex linear in the second variable).  
Define
\begin{equation}\label{eq:cn}
c_\sn 
:= (-1)^{\frac{\sn(\sn+1)}{2}}\bi^\sn 
= \left\{\begin{array}{rl}
1,&\mbox{if }\sn\mbox{ is even},\\
-\bi,&\mbox{if }\sn\mbox{ is odd}.
\end{array}\right.
\end{equation}

\begin{lemma}\label{le:theta}
Let~$J\in\sJ(M)$ be an almost complex structure compatible
with the orientation. Then $c_1(TM,J)=c_1(L)\in H^2(M;\Z)$ 
if and only if there exists a nowhere vanishing 
$\sn$-form~${\theta\in\Om^{\sn,0}_J(M,L)}$.  If this holds then
\begin{equation}\label{eq:rhotheta}
\rho := c_\sn\winner{\theta}{\theta} \in\Om^{2\sn}(M)
\end{equation}
is a positive volume form on~$M$. 
\end{lemma}

\begin{proof}
The first Chern class of~$(TM,J)$ agrees with minus the first Chern class
of the complex line bundle~${\Lambda^{\sn,0}_JT^*M}$.
Hence~${c_1(TM,J)=c_1(L)}$ if and only if~${E:=\Lambda^{\sn,0}_JT^*M\otimes L}$ 
admits a a trivialization or, equivalently, a nowhere vanishing section,
and such a section is an~$(\sn,0)$-form~${\theta\in\Om^{\sn,0}_J(M,L)}$.

To show that, for any nowhere vanishing 
$\sn$-form~${\theta\in\Om^{\sn,0}_J(M)}$, 
the formula~\eqref{eq:rhotheta} defines a positive volume 
form on~$M$, fix an element~${m\in M}$ and choose
a complex isomorphism~${(\C^\sn,\bi)\to(T_mM,J)}$.
Let~${z_i=x_i+\bi y_i}$ for~${i=1,\dots,\sn}$
be the coordinates on~$\C^\sn$.   
Then there is an element~${\lambda\in L_m}$ 
(the fiber of~$L$ over~$m$) such that
$$
\theta_m = \lambda dz_1\wedge\cdots\wedge dz_\sn.
$$
Hence 
\begin{equation*}
\begin{split}
\rho_m
&= 
c_\sn \winner{\theta_m}{\theta_m} \\
&=
\frac{(-1)^{\frac{\sn(\sn-1)}{2}}}{\bi^\sn}\abs{\lambda}^2
d\overline{z}_1\wedge\cdots\wedge d\overline{z}_\sn
\wedge dz_1\wedge\cdots\wedge dz_\sn \\
&= 
2^\sn\abs{\lambda}^2\frac{d\overline{z}_1\wedge dz_1}{2\bi}
\wedge\cdots\wedge\frac{d\overline{z}_\sn\wedge dz_\sn}{2\bi} \\
&=
2^\sn\abs{\lambda}^2dx_1\wedge dy_1\wedge\cdots\wedge dx_\sn\wedge dy_\sn.
\end{split}
\end{equation*}
Thus~$\rho$ is a positive volume form on~$M$ 
and this proves Lemma~\ref{le:theta}.
\end{proof}

\begin{lemma}\label{le:THODGE}
Let~${J\in\sJ(M)}$ be an almost complex structure compatible with the orientation, 
let~${\theta\in\Om^{\sn,0}_J(M,L)}$ be a nowhere vanishing $(\sn,0)$-form,
let~${\om\in\Om^2(M)}$ be a nondegenerate $2$-form that is compatible 
with~$J$ such that
\begin{equation}\label{eq:rhom}
\frac{\om^\sn}{\sn!} = c_\sn\winner{\theta}{\theta}=:\rho,
\end{equation}
and let~${*:\Om^{p,q}_J(M,L)\to\Om^{\sn-q,\sn-p}_J(M,L)}$ be the Hodge $*$-operator 
of the Riemannian metric~${\inner{\cdot}{\cdot}:=\om(\cdot,J\cdot)}$.
Then the following holds.

\smallskip\noindent{\bf (i)}
For every~${\Jhat\in\Om^{0,1}_J(M,TM)}$ there is 
a unique~${\beta\in\Om^{\sn-1,1}_J(M,L)}$ such that
\begin{equation}\label{eq:JHATBETA}
\bi\iota(u)\beta-\iota(Ju)\beta = \iota(\Jhat u)\theta
\end{equation}
for all~${u\in\Vect(M)}$.  

\smallskip\noindent{\bf (ii)}
For every~${\beta\in\Om^{\sn-1,1}_J(M,L)}$ there exists 
a unique~${\Jhat\in\Om^{0,1}_J(M,TM)}$ such 
that~\eqref{eq:JHATBETA} holds for all~${u\in\Vect(M)}$. 

\smallskip\noindent{\bf (iii)}
Suppose~${\beta\in\Om^{\sn-1,1}_J(M,L)}$ and~${\Jhat\in\Om^{0,1}_J(M,TM)}$ 
satisfy equation~\eqref{eq:JHATBETA}.  Then
\begin{equation}\label{eq:JHATBETA1}
\bi\iota(u){*\beta} - \iota(Ju){*\beta} = -c_\sn\iota(\Jhat^*u)\theta
\end{equation}
for all~${u\in\Vect(M)}$. Moreover, 
\begin{equation}\label{eq:JHATBETA2}
\Jhat=\Jhat^*
\quad\iff\quad 
*\beta=-c_\sn \beta
\quad\iff\quad
\beta\wedge\om=0,
\end{equation}
\begin{equation}\label{eq:JHATBETA3}
\Jhat+\Jhat^*=0
\quad\iff\quad 
*\beta=c_\sn \beta
\quad\iff\quad 
\beta\in\Om^{\sn-2,0}_J(M,L)\wedge\om.
\end{equation}

\smallskip\noindent{\bf (iv)}
Suppose~${\beta\in\Om^{\sn-1,1}_J(M,L)}$ and~${\Jhat\in\Om^{0,1}_J(M,TM)}$ 
satisfy equation~\eqref{eq:JHATBETA} and let~${\omhat\in\Om^2(M)}$.  Then 
\begin{equation}\label{eq:omhatbeta}
\om\wedge\beta+\omhat\wedge\theta=0\quad\iff\quad
\begin{array}{l}
\omhat(u,v)-\omhat(Ju,Jv)=\inner{(\Jhat-\Jhat^*)u}{v} \\
\mbox{for all }u,v\in\Vect(M).
\end{array}
\end{equation}

\smallskip\noindent{\bf (v)}
Let~${\beta,\beta'\in\Om^{\sn-1,1}_J(M,L)}$ and~${\Jhat,\Jhat'\in\Om^{0,1}_J(M,TM)}$ 
be given such that the pairs~$(\beta,\Jhat)$ and~$(\beta',\Jhat')$ satisfy~\eqref{eq:JHATBETA}.  
Then the pointwise inner product of~$\beta$ and~$\beta'$ is given by 
\begin{equation}\label{eq:BETAinner}
\inner{\beta}{\beta'}
= \Re\left(\frac{\winner{\beta}{*\beta'}}{\rho}\right)
= \tfrac{1}{8}\trace\bigl(\Jhat^*\Jhat'\bigr)\rho.
\end{equation}
Moreover, we have
\begin{equation}\label{eq:BETAsymp}
c_\sn\winner{\beta}{\beta'}
= 
- \tfrac{1}{8}\trace\bigl(\Jhat\Jhat'\bigr)\rho
+ \tfrac{\bi}{8}\trace\bigl(\Jhat J\Jhat'\bigr)\rho.
\end{equation}
\end{lemma}

\begin{proof}
Define~${\beta\in\Om^\sn(M,L)}$ by
\begin{equation*}
\begin{split}
\beta(v_1,\dots,v_\sn)
&:=
\theta(-\tfrac{1}{2}J\Jhat v_1,v_2,\dots,v_\sn) \\
&\quad\;
+ \theta(v_1,-\tfrac{1}{2}J\Jhat v_2,v_3,\dots,v_\sn) \\
&\quad\;
+ \cdots 
+ \theta(v_1,\dots,v_{\sn-1},-\tfrac{1}{2}J\Jhat v_\sn)
\end{split}
\end{equation*}
for~${v_1,\dots,v_\sn\in\Vect(M)}$. Then
\begin{equation*}
\begin{split}
\beta(Ju,v_2,\dots,v_\sn)
+ \theta(\Jhat u,v_2,\dots,v_\sn) 
&=
\theta(\tfrac{1}{2}\Jhat u,v_2,\dots,v_\sn) \\
&\quad
+ \theta(Ju,-\tfrac{1}{2}J\Jhat v_2,v_3,\dots,v_\sn) \\
&\quad
+ \cdots 
+ \theta(Ju,v_2,\dots,v_{\sn-1},-\tfrac{1}{2}J\Jhat v_\sn) \\
&= 
\bi\beta(u,v_2,\dots,v_\sn)
\end{split}
\end{equation*}
for all~${u,v_2,\dots,v_\sn\in\Vect(M)}$. 
Thus~$\beta$ is an $(\sn-1,1)$-form that 
satisfies equation~\eqref{eq:JHATBETA}.
If~$\beta'$ is another $(\sn-1,1)$-form that 
satisfies equation~\eqref{eq:JHATBETA}, 
then~${\iota(Ju)(\beta'-\beta)=\bi\iota(u)(\beta'-\beta)}$,
thus~${\beta'-\beta\in\Om^{\sn,0}_J(M,L)}$, and so~${\beta'=\beta}$. 
This proves~(i).

We prove part~(ii). Thus let~${\beta\in\Om^{\sn-1,1}_J(M,L)}$
be given. Then for every vector field~${u\in\Vect(M)}$ the 
$(\sn-1)$-form~${\bi\iota(u)\beta-\iota(Ju)\beta}$
is of type~${(\sn-1,0)}$ and hence can be written in the 
form~${\iota(v)\theta}$ for some vector field~${v\in\Vect(M)}$
that is uniquely determined by~$u$.  This shows that there 
exists a unique section~${\Jhat\in\Om^0(M,\End(TM))}$
of the endomorphism bundle that satisfies~\eqref{eq:JHATBETA}
for all~${u\in\Vect(M)}$. By~\eqref{eq:JHATBETA} we have
$$
\iota(\Jhat Ju)\theta 
= \bi\iota(Ju)\beta+\iota(u)\beta 
= -\bi\iota(\Jhat u)\theta 
= -\iota(J\Jhat u)\theta
$$
for all~${u\in\Vect(M)}$ and thus~${\Jhat J+J\Jhat=0}$.
This proves~(ii). 

We prove part~(iii). It suffices to consider the trivial line bundle 
and the standard structures on~$\R^{2\sn}$ with the 
coordinates~${x_1,\dots,x_\sn,y_1,\dots,y_\sn}$.
They are given by 
\begin{equation}\label{eq:JOMloc}
J = \left(\begin{array}{rr} 0 & -\one \\ \one & 0 \end{array}\right),\qquad
\om = \sum_{i=1}^\sn dx_i\wedge dy_i,\qquad
\theta = \frac{dz_1}{\sqrt{2}}\wedge\cdots\wedge\frac{dz_\sn}{\sqrt{2}},
\end{equation}
where~${z_i:=x_i+\bi y_i}$ for~${i=1,\dots,\sn}$. 
A complex anti-linear endomorphism has the form
\begin{equation}\label{eq:JHATloc}
\Jhat = \left(\begin{array}{rr} A & B \\ B & -A \end{array}\right),\qquad
A+\bi B = (a_{ij})_{i,j=1,\dots,\sn}\in\C^{\sn\times\sn}.
\end{equation}
The corresponding~$(\sn-1,1)$-form ${\beta\in\Om^{\sn-1,1}_J(\R^{2\sn})}$
is given by 
\begin{equation}\label{eq:BETAloc}
\beta = \frac{1}{2\bi}\sum_{i,j=1}^\sn a_{ij}
\frac{dz_1}{\sqrt{2}}\wedge\cdots\frac{dz_{i-1}}{\sqrt{2}}
\wedge\frac{d\oz_j}{\sqrt{2}}
\wedge\frac{dz_{i+1}}{\sqrt{2}}\wedge\cdots\wedge\frac{dz_\sn}{\sqrt{2}}.
\end{equation}
Now~$\Jhat^*$ is represented by the transposed 
matrix~${A^T+\bi B^T=(a_{ji})_{i,j=1,\dots,\sn}}$ and
$$
{*\beta} = \frac{-c_\sn}{2\bi}\sum_{i,j=1}^\sn a_{ji}
\frac{dz_1}{\sqrt{2}}\wedge\cdots\frac{dz_{i-1}}{\sqrt{2}}
\wedge\frac{d\oz_j}{\sqrt{2}}
\wedge\frac{dz_{i+1}}{\sqrt{2}}\wedge\cdots\wedge\frac{dz_\sn}{\sqrt{2}}.
$$
This proves~\eqref{eq:JHATBETA1}. Now~\eqref{eq:JHATBETA2} 
and~\eqref{eq:JHATBETA3} follow from~\eqref{eq:JHATBETA1} 
and the eigenspace decomposition of the Hodge $*$-operator 
on~$\Om^{\sn-1,1}(M)$.  This proves~(iii).

We prove part~(iv). Continue the notation in the proof 
of part~(iii), so~${J,\om,\theta,\Jhat,\beta}$ 
are as in~\eqref{eq:JOMloc}, \eqref{eq:JHATloc}, and~\eqref{eq:BETAloc}. 
Then a $2$-form~$\omhat\in\Om^2(M)$
satisfies~${\omhat(u,v)-\omhat(Ju,Jv)=\inner{(\Jhat-\Jhat^*)u}{v}}$
for all~${u,v\in\Vect(M)}$ if and only if its ${(0,2)}$-part is given 
by~${\omhat^{0,2}=-\tfrac{1}{4}\sum_{i,j}a_{ij}d\oz_i\wedge d\oz_j}$,
and this in turn is equivalent to the 
equation~${\omhat\wedge\theta=-\om\wedge\beta}$.
This proves~(iv).

We prove part~(v). Continue the notation in the proof of part~(iii)
and use the same notation for~$(\beta',\Jhat')$
with~$A,B,a_{ij}$ replaced by~$A',B',a_{ij}'$.  Then 
\begin{equation*}
\begin{split}
\obeta\wedge*\beta'
&=
\frac{-c_\sn}{4} \sum_{i,j=1}^\sn\sum_{k,\ell=1}^\sn 
\oa_{ij}a_{\ell k}' 
\frac{d\oz_1}{\sqrt{2}}\wedge\cdots\wedge\frac{d\oz_{i-1}}{\sqrt{2}}
\wedge\frac{dz_j}{\sqrt{2}}
\wedge\frac{d\oz_{i+1}}{\sqrt{2}}\wedge\cdots\wedge\frac{d\oz_\sn}{\sqrt{2}}
\\
&\qquad
\wedge
\frac{dz_1}{\sqrt{2}}\wedge\cdots\wedge\frac{dz_{k-1}}{\sqrt{2}}
\wedge\frac{d\oz_\ell}{\sqrt{2}}
\wedge\frac{dz_{k+1}}{\sqrt{2}}\wedge\cdots\wedge\frac{dz_\sn}{\sqrt{2}} \\
&=
\frac{c_\sn}{4} \sum_{i,j=1}^\sn
\oa_{ij}a_{ji}' 
\frac{d\oz_1}{\sqrt{2}}\wedge\cdots\wedge\frac{d\oz_\sn}{\sqrt{2}}
\wedge\frac{dz_1}{\sqrt{2}}\wedge\cdots\wedge\frac{dz_\sn}{\sqrt{2}} \\
&=
\frac{1}{4}\sum_{i,j=1}^\sn\oa_{ij}a_{ji}' c_\sn\otheta\wedge\theta \\
&=
\frac{1}{4}\trace(A-\bi B)^T(A'+\bi B')\rho.
\end{split}
\end{equation*}
Thus~${\Re(\obeta\wedge *\beta') 
= \tfrac{1}{4}\trace(A^TA'+B^TB')\rho 
= \tfrac{1}{8}\trace(\Jhat^T\Jhat')\rho}$
and this proves equation~\eqref{eq:BETAinner}.  
Moreover, the pair~$(\oc_\sn*\beta,-\Jhat^*)$ 
satisfies~\eqref{eq:JHATBETA} by part~(iii).  
Thus~${\Re(c_\sn\obeta\wedge\beta')
= \Re(\overline{\oc_\sn*\beta}\wedge\beta')
= - \tfrac{1}{8}\trace(\Jhat\Jhat')\rho}$. 
This confirms~\eqref{eq:BETAsymp} for the real part. 
The formula for the imaginary part holds because 
both sides of the equation are complex linear 
in~$\Jhat'$ with respect to the complex 
structure~${\Jhat'\mapsto J\Jhat'}$. 
This proves~(v) and Lemma~\ref{le:THODGE}.
\end{proof}

The next lemma adapts an observation by Donaldson 
in~\cite[Lemma~1]{DON4} to the present setting.

\begin{lemma}\label{le:THETA}
Let~$\rho$ be a positive volume form and let~$J\in\sJ(M)$ 
be a positive almost complex structure such that~${c_1(TM,J)=c_1(L)\in H^2(M;\Z)}$. 
Then the following are equivalent.

\smallskip\noindent{\bf (i)}  
$J$ is integrable.

\smallskip\noindent{\bf (ii)}  
There exists a nowhere vanishing 
$\sn$-form~$\theta\in\Om^{\sn,0}_J(M,L)$
and a Hermitian connection~$\Nabla{L}$ on~$L$ 
such that~${d^\Nabla{L}\theta=0}$ 
and~${c_\sn\winner{\theta}{\theta}=\rho}$.

\smallskip\noindent
If~(i) holds then the pair~$(\Nabla{L},\theta)$ in~(ii) is uniquely 
determined by~$J$ up to unitary gauge equivalence.  
If~(ii) holds then
\begin{equation}\label{eq:RICF}
(F^\Nabla{L})_J^{0,2}=0,\qquad
\Ric_{\rho,J}=\bi F^\Nabla{L}.
\end{equation}
If~(i) and~(ii) hold and~$\nabla$ is a torsion-free 
connection on~$TM$ that satisfies~${\nabla J=0}$, 
then~${\nabla\rho=0}$ if and only if~${\nabla\theta=0}$.
\end{lemma}

\begin{proof}
We prove that~(i) implies~(ii). By Lemma~\ref{le:theta} there exists
a nowhere vanishing $(\sn,0)$-form~${\theta\in\Om^{\sn,0}_J(M,L)}$
such that~${c_\sn\winner{\theta}{\theta}=\rho}$.
Choose any Hermitian connection~$\Nabla{\,0}$ on~$L$.
Then~${d^\Nabla{\,0}\theta\in\Om^{n,1}_J(M,L)}$ because~$J$ is integrable 
and hence there exists a unique $1$-form~${\eta\in\Om^{0,1}_J(M)}$ 
such that~${\eta\wedge\theta = d^\Nabla{\,0}\theta}$.
Define the Hermitian connection~$\Nabla{L}$ 
by~${\Nabla{L} := \Nabla{\,0}+\oeta-\eta}$. Then
$$
d^\Nabla{L}\theta 
= d^\Nabla{\,0}\theta + (\oeta-\eta)\wedge\theta 
= \oeta\wedge\theta 
= 0
$$
because~${\oeta\in\Om^{1,0}_J(M)}$. This shows that~(i) implies~(ii). 
Moreover,~(ii) implies 
$$
(F^\Nabla{L})^{0,2}_J\wedge\theta 
= F^\Nabla{L}\wedge\theta
= d^\Nabla{L}d^\Nabla{L}\theta
= 0
$$
and hence~${(F^\Nabla{L})^{0,2}_J=0}$.

We prove uniqueness in~(ii).
If~$(\theta',\Nabla{L}')$ is any other pair as in~(ii),
then there exists a unique unitary gauge
transformation~${g:M\to S^1}$ such that
$$
\theta'=g^{-1}\theta
$$
Hence the $1$-form~${\alpha:=\Nabla{L}'-\Nabla{L}\in\Om^1(M,\bi\R)}$
satisfies
$$
0 
= d^{\Nabla{L}'}\theta' 
= d^{\Nabla{L}+\alpha}(g^{-1}\theta)
= \alpha\wedge g^{-1}\theta + dg^{-1}\wedge\theta
= (\alpha^{0,1}-g^{-1}\bar\p g)\wedge g^{-1}\theta.
$$
Hence~${\alpha^{0,1}=g^{-1}\bar\p g}$ and 
so~${\alpha=g^{-1}\bar\p_Jg-\og^{-1}\p_J\og = g^{-1}dg}$
because~$g^{-1}dg$ is a $1$-form on~$M$ with values in~$\bi\R$.  
Thus
$$
\Nabla{L}' =\Nabla{L}+g^{-1}dg = g^*\Nabla{L}
$$
and this proves uniqueness up to unitary gauge equivalence. 

\bigbreak

We prove that~(ii) implies~(i).  If~${\theta\in\Om^{\sn,0}_J(M,L)}$ 
and~$\Nabla{L}$ is any complex connection on~$L$ 
then~${(d^\Nabla{L}\theta)^{\sn-1,2} = \tfrac{1}{4}\iota(N_J)\theta}$,
where
\begin{equation}\label{eq:NJtheta}
\begin{split}
&(\iota(N_J)\theta)(v_1,\dots,v_{\sn+1}) \\
&:= \sum_{i<j}(-1)^{i+j-1}
\theta\left(N_J(v_i,v_j),v_1,\dots,\widehat{v_i},
\dots,\widehat{v_j},\dots,v_{\sn+1}\right)
\end{split}
\end{equation}
for~${v_1,\dots,v_{\sn+1}\in\Vect(M)}$. 
If~${d^\Nabla{L}\theta=0}$ it follows that~${\iota(N_J)\theta=0}$.  
If~$\theta$ vanishes nowhere this implies $N_J=0$.  
To see this, fix two vector fields $v_1,v_2$.
Then~$\iota(N_J(v_1,v_2))\theta$ is a nonzero $(\sn-1,0)$-form 
while the remaining summands on the right in~\eqref{eq:NJtheta} 
are of type~$(\sn-2,1)$ or~$(\sn-3,2)$.
This implies that~${\iota(N_J(v_1,v_2))\theta=0}$ and 
hence~$N_J(v_1,v_2)=0$ because~$\theta$ vanishes nowhere.
Thus $N_J=0$ and therefore $J$ is integrable.

Now assume~(ii) and let~$\nabla$ be a torsion-free connection on~$TM$
that satisfies~${\nabla J=0}$ and~${\nabla\rho=0}$.  Such a connection exists
by part~(i) of Lemma~\ref{le:TORSION}.  For~${u\in\Vect(M)}$ the 
$\sn$-form ${\Nabla{u}\theta\in\Om^\sn(M,L)}$ is defined by 
\begin{equation*}
\begin{split}
(\Nabla{u}\theta)(v_1,\dots v_\sn) 
&:=
\Nabla{L,u}\bigl(\theta(v_1,\dots,v_\sn)\bigr) \\
&\quad\;
- \theta(\Nabla{u}v_1,v_2,\dots,v_\sn)
- \cdots - \theta(v_2,\dots,v_{\sn-1},\Nabla{u}v_\sn)
\end{split}
\end{equation*}
Since~${\nabla J=0}$, this is an~$(\sn,0)$-form. 
Hence there exists a unique complex valued 
$1$-form~${\alpha\in\Om^1(M,\C)}$ such that
$$
\Nabla{u}\theta = \alpha(u)\theta
$$
for all~${u\in\Vect(M)}$.  Now the equation~${d^\Nabla{L}\theta=0}$
can be expressed in the form
\begin{equation*}
\begin{split}
(\Nabla{u}\theta)(v_1,\dots v_\sn) 
= \sum_{i=1}^n(-1)^{i-1}(\Nabla{v_i}\theta)(u,v_1,\dots,v_{i-1},v_{i+1},\dots,v_\sn)
\end{split}
\end{equation*}
for~${u,v_1,\dots,v_\sn\in\Vect(M)}$.  The right hand side of this equation
is complex linear in~$u$ and this implies~${\alpha\in\Om^{1,0}(M)}$, i.e.\ 
$$
\alpha(Ju) = \bi\alpha(u)
$$
for all~${u\in\Vect(M)}$.  
Since~$\rho=c_\sn\winner{\theta}{\theta}$ and~${\nabla\rho=0}$, 
we also have~${\RE\alpha=0}$, hence~${\alpha=0}$, and so~$\nabla\theta=0$.  
This implies~$\trace^c(R^\nabla)=F^\Nabla{L}$ and therefore 
$$
\Ric_{\rho,J}=\tfrac{1}{2}\trace(JR^\nabla)
=\bi\trace^c(R^\nabla)=\bi F^\Nabla{L}.
$$
This proves Lemma~\ref{le:THETA}. 
\end{proof}

\begin{lemma}\label{le:BETA}
Let~${\rho\in\Om^{2\sn}(M)}$ be a positive volume form,
let~${J\in\sJ_\INT(M)}$, let~$\Nabla{L}$ be a Hermitian connection on~$L$,
and let~${\theta\in\Om^{\sn,0}_J(M,L)}$ be nowhere vanishing 
such that~${d^\Nabla{L}\theta=0}$ and~${c_\sn\winner{\theta}{\theta}=\rho}$.  
Then the following holds.

\smallskip\noindent{\bf (i)}
Let~${v\in\Vect(M)}$ and define~${\Jhat:=\cL_vJ}$ 
and~${\beta:=\bar\p^\Nabla{L}_J\iota(v)\theta\in\Om^{\sn-1,1}_J(M,L)}$
and~${h:=\tfrac{1}{2}(f_v-\bi f_{Jv})}$.
Then~${d^\Nabla{L}\iota(v)\theta=\beta+h\theta}$
and~\eqref{eq:JHATBETA} holds for all~$u$.

\smallskip\noindent{\bf (ii)}
Suppose~${\Jhat\in\Om^{0,1}_J(M,TM)}$ and~${\beta\in\Om^{\sn-1,1}_J(M,L)}$ 
satisfy~\eqref{eq:JHATBETA}.  Then
\begin{equation}\label{eq:dbarJHATBETA1}
\bar\p_J^*\Jhat=0\qquad\iff\qquad\bigl(\bar\p^\Nabla{L}_J\bigr)^*\beta = 0,
\end{equation}
\begin{equation}\label{eq:dbarJHATBETA2}
\bar\p_J\Jhat=0\qquad\iff\qquad\bar\p^\Nabla{L}_J\beta = 0.
\end{equation}

\smallskip\noindent{\bf (iii)}
Let~$\Jhat$ and~$\beta$ be as in~(ii) and let~${\Lambda_\rho(J,\Jhat)}$ 
be as in~\eqref{eq:LAMBDA}.  Then
\begin{equation}\label{eq:BETALAMBDA}
\bi\p_J^{\Nabla{L}}\beta + \tfrac{1}{2}\Lambda_\rho(J,\Jhat)\wedge\theta=0.
\end{equation}

\smallskip\noindent{\bf (iv)}
Let~$\Jhat$ and~$\beta$ be as in~(ii) with~${\bar\p_J\Jhat=0}$
and let~${\Ric_{\rho,J}}$ and~${\Richat_\rho(J,\Jhat)}$
be as in Theorem~\ref{thm:RICCI}.
Then~${\Ric_{\rho,J}\wedge\beta + \Richat_\rho(J,\Jhat)\wedge\theta=0}$.

\smallskip\noindent{\bf (v)}
Let~$\Jhat$ and~$\beta$ be as in~(ii) with~${\bar\p_J\Jhat=0}$
and assume~${F^\Nabla{L}=0}$ and~$J$ admits a K\"ahler form. 
Then there exists a unique function~${h\in\Om^0(M,\C)}$ such that
\begin{equation}\label{eq:BETAH}
d^\Nabla{L}\bigl(\beta + h\theta\bigr) = 0,\qquad
\int_Mh\rho=0.
\end{equation}
Moreover, ${h=\tfrac{1}{2}(f-\bi g)}$ 
in the notation of Lemma~\ref{le:LAMBDAFG}, 
and~${\beta+h\theta\in\im d^\Nabla{L}}$
if and only if there exists a vector field~$v$
such that~${\Jhat=\cL_vJ}$.
\end{lemma}

\begin{proof}
Fix a torsion-free connection~$\nabla$ 
such that~${\nabla J=0}$ and~${\nabla\theta=0}$.  
Next define~${\cL_v^{\Nabla{L}}\alpha
:=d^{\Nabla{L}}\iota(v)\alpha+\iota(v)d^{\Nabla{L}}\alpha}$
for~${\alpha\in\Om^k(M,L)}$ and~${v\in\Vect(M)}$
Then
\begin{equation*}
\begin{split}
(\cL_v^\Nabla{L}\alpha)(v_1,\dots,v_k)
= 
\Nabla{L,v}\bigl(\alpha(v_1,\dots,v_k)\bigr) 
- \sum\alpha(\dots,v_{i-1},[v_i,v],v_{i+1},\dots)
\end{split}
\end{equation*}
for~${v,v_1,\dots,v_k\in\Vect(M)}$ 
and~${\cL^{\Nabla{L}}_v\theta=d^{\Nabla{L}}\iota(v)\theta}$.
Hence, by the Leibniz rule,
\begin{equation*}
\begin{split}
(d^{\Nabla{L}}\iota(v)\theta)(v_1,\dots,v_\sn)
= \theta(\Nabla{v_1}v,v_2,\dots,v_\sn)
+ \cdots + \theta(v_1,\dots,v_{\sn-1},\Nabla{v_\sn}v)
\end{split}
\end{equation*}
for all~${v,v_1,\dots,v_\sn\in\Vect(M)}$. 
Since~$\theta$ is complex multi-linear this implies
\begin{equation}\label{eq:dbeta3}
\begin{split}
\bi\iota(u)d^{\Nabla{L}}\iota(v)\theta
- \iota(Ju)d^{\Nabla{L}}\iota(v)\theta
= \iota(J\Nabla{u}v-\Nabla{Ju}v)\theta
= \iota((\cL_vJ)u)\theta
\end{split}
\end{equation}
for all~${u,v\in\Vect(M)}$.
Hence
\begin{equation*}
\begin{split}
\iota((\cL_vJ)u)\theta 
= 
\bi\iota(u)(d^{\Nabla{L}}\iota(v)\theta)^{\sn-1,1}_J
- \iota(Ju)(d^{\Nabla{L}}\iota(v)\theta)^{\sn-1,1}_J
= 
\bi\iota(u)\beta - \iota(Ju)\beta
\end{split}
\end{equation*}
for all~${u,v}$ and this proves~\eqref{eq:JHATBETA}.
The equation~${\winner{\theta}{\p^\Nabla{L}_J\iota(v)\theta}=\winner{\theta}{h\theta}}$ 
follows directly from the definitions and the formula~${\rho=c_\sn\winner{\theta}{\theta}}$.
This proves~(i).

\bigbreak

We prove part~(ii).  The equivalence in~\eqref{eq:dbarJHATBETA1}
follows from the identity
\begin{equation*}
\begin{split}
\inner{\beta}{\bar\p^\Nabla{L}_J\iota(v)\theta}_{L^2}
= 
\tfrac{1}{8}\int_M\trace\bigl(\Jhat^*\cL_vJ\bigr)\rho 
= 
\tfrac{1}{4}\inner{\Jhat}{J\bar\p_Jv}_{L^2} 
= 
\tfrac{1}{4}\inner{\bar\p_J^*\Jhat}{Jv}_{L^2}
\end{split}
\end{equation*}
for~${v\in\Vect(M)}$.  Here we have used part~(v) 
of Lemma~\ref{le:THODGE} as well as~(i). 
To prove~\eqref{eq:dbarJHATBETA2}, 
define~${\alpha_u\in\Om^\sn(M,L)}$ by
\begin{equation}\label{eq:dbeta1}
\alpha_u := \bi\iota(u)d^{\Nabla{L}}\beta - \iota(Ju) d^{\Nabla{L}}\beta
\end{equation}
for~${u\in\Vect(M)}$.  We will prove that,
for all~${u,v\in\Vect(M)}$,
\begin{equation}\label{eq:dbeta2}
\bi\iota(v)\alpha_u - \iota(Jv)\alpha_u = \iota(J\bar\p_J\Jhat(u,v))\theta.
\end{equation}
Equation~\eqref{eq:dbeta2} shows 
that~${\bar\p_J\Jhat=0}$ if and only if~${\alpha_u\in\Om^{\sn,0}_J(M,L)}$
for every vector field~${u\in\Vect(M)}$.  By~\eqref{eq:dbeta1} this is equivalent to
the condition~${d^{\Nabla{L}}\beta\in\Om^{\sn,1}_J(M,L)}$ or, equivalently, 
to~${\bar\p_J^{\Nabla{L}}\beta = (d^{\Nabla{L}}\beta)_J^{\sn-1,2}=0}$.

To prove~\eqref{eq:dbeta2}, fix a torsion-free connection~$\nabla$ 
that satisfies~${\nabla J=0}$ and~${\nabla\theta=0}$. 
Then it follows from~\eqref{eq:dbeta3} with~$v$ 
replaced by~${\Jhat v}$ that
\begin{equation}\label{eq:dbeta4}
\begin{split}
&\bi\iota(u)d^{\Nabla{L}}\iota(\Jhat v)\theta
- \iota(Ju)d^{\Nabla{L}}\iota(\Jhat v)\theta \\
&= 
\iota\bigl(J(\Nabla{u}\Jhat)v-(\Nabla{Ju}\Jhat)v\bigr)\theta 
+ \iota\bigl(J\Jhat\Nabla{u}v - \Jhat\Nabla{Ju}v\bigr)\theta
\end{split}
\end{equation}
for all~${u,v\in\Vect(M)}$.  Moreover,
\begin{equation}\label{eq:dbeta5}
\begin{split}
\alpha_u
=
\bi\iota(u)d^{\Nabla{L}}\beta 
- \iota(Ju)d^{\Nabla{L}}\beta 
=
\bi\cL^{\Nabla{L}}_u\beta 
- \cL^{\Nabla{L}}_{Ju}\beta 
- d^{\Nabla{L}}\iota(\Jhat u)\theta
\end{split}
\end{equation}
for all~${u\in\Vect(M)}$ by~\eqref{eq:JHATBETA}.
With this understood, we obtain
\begin{equation*}
\begin{split}
\bi\iota(v)\alpha_u - \iota(Jv)\alpha_u 
&= 
- \iota(v)\cL^{\Nabla{L}}_u\beta 
- \bi\iota(v)\cL^{\Nabla{L}}_{Ju}\beta 
- \bi\iota(v)d\iota(\Jhat u)\theta \\
&\quad
-\bi\iota(Jv)\cL^{\Nabla{L}}_u\beta 
+ \iota(Jv)\cL^{\Nabla{L}}_{Ju}\beta 
+ \iota(Jv)d^{\Nabla{L}}\iota(\Jhat u)\theta \\
&=
- \cL^{\Nabla{L}}_u\iota(v)\beta - \iota([u,v])\beta
- \bi\cL^{\Nabla{L}}_{Ju}\iota(v)\beta - \bi\iota([Ju,v])\beta \\
&\quad
-\bi\cL^{\Nabla{L}}_u\iota(Jv)\beta - \bi\iota([u,Jv])\beta
+ \cL^{\Nabla{L}}_{Ju}\iota(Jv)\beta + \iota([Ju,Jv])\beta \\
&\quad
- \bi\iota(v)d\iota(\Jhat u)\theta + \iota(Jv)d\iota(\Jhat u)\theta\\
&=
\bi\iota(u)d^{\Nabla{L}}\iota(\Jhat v)\theta 
-  \iota(Ju)d^{\Nabla{L}}\iota(\Jhat v)\theta 
- \bi\iota(v)d^{\Nabla{L}}\iota(\Jhat u)\theta \\
&\quad
+ \iota(Jv)d^{\Nabla{L}}\iota(\Jhat u)\theta 
+ \iota(J\Jhat[u,v])\theta 
+ \iota(\Jhat J[Ju,Jv])\theta  \\
&=
\iota\bigl(J(\Nabla{u}\Jhat)v - (\Nabla{Ju}\Jhat)v
-J(\Nabla{v}\Jhat)u + (\Nabla{Jv}\Jhat)u\bigr)\theta \\
&= 
\iota(J\bar\p_J\Jhat(u,v))\theta.
\end{split}
\end{equation*}
Here the first equality follows from~\eqref{eq:dbeta5}, 
the third from~\eqref{eq:JHATBETA},
and the fourth from~\eqref{eq:dbeta4}.
This proves~\eqref{eq:dbeta2} and~(ii). 

\bigbreak

We prove part~(iii).  Since~${\bi\p^{\Nabla{L}}_J\beta\in\Om^{n,1}_J(M,L)}$,
there is an~${\eta\in\Om^{0,1}_J(M)}$ such that
\begin{equation}\label{eq:dbetaeta}
\bi\p^{\Nabla{L}}_J\beta + \eta\wedge\theta = 0.
\end{equation}
Now let~${v\in\Vect(M)}$.  
Then the pair~$(\cL_vJ,\bar\p^{\Nabla{L}}_J\iota(v)\theta)$
satisfies~\eqref{eq:JHATBETA} by~(i).  
Hence, by~\eqref{eq:LAMBDARHO} and~\eqref{eq:BETAsymp}, we have
\begin{equation*}
\begin{split}
\tfrac{1}{4}\int_M\Lambda_\rho(J,\Jhat)\wedge\iota(v)\rho
&=
\tfrac{1}{8}\int_M\trace\bigl(\Jhat J\cL_vJ\bigr)\rho \\
&=
\Im\left(\int_Mc_\sn\winner{\beta}{\bar\p^{\Nabla{L}}_J\iota(v)\theta}\right) \\
&=
\Re\left(\int_Mc_\sn\winner{(\bi\beta)}{d^{\Nabla{L}}\iota(v)\theta}\right) \\
&=
(-1)^{\sn+1}\Re\left(\int_Mc_\sn\winner{(\bi d^{\Nabla{L}}\beta)}{\iota(v)\theta}\right) \\
&=
-\Re\left(
\int_Mc_\sn\winner{(\iota(v)\bi\p^{\Nabla{L}}_J\beta)}{\theta}
\right) \\
&=
\Re\left(
\int_M\overline{\eta(v)}c_\sn\winner{\theta}{\theta}
\right) \\
&=
\int_M\Re(\eta)\wedge\iota(v)\rho. 
\end{split}
\end{equation*}
Here the penultimate equality follows from~\eqref{eq:dbetaeta}.
Thus~${\Re(\eta)=\tfrac{1}{4}\Lambda_\rho(J,\Jhat)}$,
hence~${\eta = \tfrac{1}{2}\Lambda_\rho(J,\Jhat)^{0,1}_J}$,
and so~\eqref{eq:BETALAMBDA} follows 
from~\eqref{eq:dbetaeta}. This proves~(iii). 

We prove part~(iv).  Since~${\bar\p_J\Jhat=0}$ it follows 
from part~(ii) that~${\bar\p^{\Nabla{L}}_J\beta=0}$.  
Hence~${\bi d^{\Nabla{L}}\beta+\tfrac{1}{2}\Lambda_\rho(J,\Jhat)\wedge\theta=0}$
by~\eqref{eq:BETALAMBDA} and so, by Lemma~\ref{le:THETA}, we have
$$
\Ric_{\rho,J}\wedge\beta
= \bi F^{\Nabla{L}}\wedge\beta
= \bi d^{\Nabla{L}}d^{\Nabla{L}}\beta
= -\tfrac{1}{2}d\Lambda_\rho(J,\Jhat)\wedge\theta
= - \Richat_\rho(J,\Jhat)\wedge\theta.
$$
If~$\Ric_{\rho,J}$ is nondegenerate, the assertion follows directly
from Lemma~\ref{le:ONEONE} and part~(iv) of Lemma~\ref{le:THODGE}.  
This proves~(iv). 

We prove part~(v).   Since~${F^{\Nabla{L}}=0}$ 
we have~${\Ric_{\rho,J}=0}$ by Lemma~\ref{le:THETA}.
Hence Lemma~\ref{le:LAMBDAFG} asserts that there exists a unique 
pair of smooth functions~${f,g\in\Om^0(M)}$ of mean value zero 
such that~${\Lambda_\rho(J,\Jhat)=-df\circ J+dg}$.
Let~${h:=\tfrac{1}{2}(f-\bi g)}$.  
Then $\bar\p_Jh=-\tfrac{\bi}{2}\Lambda_\rho(J,\Jhat)^{0,1}_J$
and so~$d^{\Nabla{L}}(\beta+h\theta)=0$ by~\eqref{eq:BETALAMBDA}.
If~${\beta+h\theta\in\im d^\Nabla{L}}$,
choose a K\"ahler form~$\om$ such that~${\om^\sn/\sn!=\rho}$
and use the identity~${d^\Nabla{L}(d^\Nabla{L})^*+(d^\Nabla{L})^*d^\Nabla{L}
=2(\bar\p_J^\Nabla{L}(\bar\p_J^\Nabla{L})^*+(\bar\p_J^\Nabla{L})^*\bar\p_J^\Nabla{L})}$
to deduce that~${\beta=\bar\p^{\Nabla{L}}_J\iota(v)\theta}$ for some
vector field~$v$. This proves Lemma~\ref{le:BETA}.
\end{proof}

If~${\Jhat\in\Om^{0,1}_J(M,TM)}$ satisfies~${\bar\p_J\Jhat=0}$,
then the~$\sn$-form
$$
\thetahat:=\beta+h\theta\in\Om^\sn(M,L)
$$ 
in part~(v) of Lemma~\ref{le:BETA} should be 
thought of as the tangent vector associated to~$\Jhat$ in the 
projective space of closed complex valued $\sn$-forms modulo scaling.  
Namely, if ${t\mapsto J_t}$ is a smooth path of (integrable)
complex structures such that~${\p_t|_{t=0}J_t=\Jhat}$,
and~${t\mapsto\theta_t\in\Om^{\sn,0}_{J_t}(M,L)}$
is a smooth path of nowhere vanishing closed $(\sn,0)$-forms, 
then~${\p_t|_{t=0}\theta_t\in\thetahat+\C\theta}$.

\begin{corollary}\label{cor:THETAHAT}
Let~$M$ be an closed connected oriented $2\sn$-mani\-fold, 
let~$J$ be a complex structure on~$M$ with real first Chern class
zero and nonempty K\"ahler cone, let~${L\to M}$ be a Hermitian
line bundle equipped with a flat connection~$\Nabla{L}$
such that~${c_1(L)=c_1(TM,J)\in H^2(M;\Z)}$,
let~${\theta\in\Om^{\sn,0}_J(M,L)}$ be a nowhere 
vanishing $(\sn,0)$-form such that~${d^\Nabla{L}\theta=0}$,
and define~${\rho:=c_\sn\winner{\theta}{\theta}}$.
For~${i=1,2}$ let~${\Jhat_i\in\Om^{0,1}_J(M,TM)}$  
such that~${\bar\p_J\Jhat_i=0}$, let~${\beta_i\in\Om^{\sn-1,1}_J(M,L)}$ 
satisfy~\eqref{eq:JHATBETA} for all~${u\in\Vect(M)}$ with~${\Jhat=\Jhat_i}$,
let~$h_i\in\Om^0(M,\C)$ be the unique function that 
satisfies~\eqref{eq:BETAH} with~${\beta=\beta_i}$, 
and define~${\thetahat_i := \beta_i+h_i\theta}$. Then
\begin{equation}\label{eq:THETAHAT}
\begin{split}
\Re\left(c_\sn\int_M\winner{\thetahat_1}{\thetahat_2}\right)
&= 
- \tfrac{1}{8}\int_M\trace(\Jhat_1\Jhat_2)\rho 
+ \int_M\Re(\oh_1h_2)\rho,  \\
\Im\left(c_\sn\int_M\winner{\thetahat_1}{\thetahat_2}\right)
&= 
\tfrac{1}{8}\int_M\trace(\Jhat_1J\Jhat_2)\rho 
+ \int_M\Im(\oh_1h_2)\rho.
\end{split}
\end{equation}
\end{corollary}

\begin{proof}
This follows directly from~\eqref{eq:BETAinner} 
and the definition of~$\thetahat_i$. 
\end{proof}

The discussion in this appendix is inspired by Donaldson's symplectic form 
on the space of complex structures on a Fano manifold in~\cite{DON4}.
He proved in~\cite[Theorem~1]{DON4} in the Fano case that the Hermitian
form
$$
(\Jhat_1,\Jhat_2)\mapsto c_\sn\int_M\winner{\thetahat_1}{\thetahat_2}
$$
is negative definite on the space of complex structures 
compatible with a fixed symplectic form~$\om$.  
In the Calabi--Yau case (with the symplectic form not fixed) 
this Hermitian form on the kernel 
of~${\bar\p_J:\Om^{0,1}_J(M,TM)\to\Om^{0,2}_J(M,TM)}$
vanishes on the image 
of the operator~${\bar\p_J:\Om^0(M,TM)\to\Om^{0,1}_J(M,TM)}$
and descends to a well-defined and nondegenerate, 
but indefinite, Hermitian form on the quotient 
space~${\ker\bar\p_J/\im\bar\p_J=T_{[J]}\sT_0(M)}$.
Its imaginary part is the symplectic form 
on Teichm\"uller space in Theorem~\ref{thm:TEICH}.


\end{document}